\documentclass[11pt]{amsart}
\usepackage[top=4.2cm, bottom=4cm, left=2.4cm, right=2.4cm]{geometry}
\usepackage[utf8]{inputenc}
\usepackage[USenglish]{babel}
\usepackage[T1]{fontenc} 
\usepackage{mathrsfs}

\usepackage{mathtools}
\usepackage{amsmath}
\usepackage{amssymb,epsfig}
\usepackage{caption}
\usepackage{subcaption}

\usepackage{amsthm}
\usepackage[absolute]{textpos}
\usepackage{mathtools,emptypage}

\usepackage[bookmarks=true]{hyperref}
\usepackage{xcolor}
\hypersetup{
    colorlinks,
    linkcolor={red!50!black},
    citecolor={blue!50!black},
    urlcolor={blue!80!black},
}

\usepackage{todonotes}

\mathtoolsset{showonlyrefs} 

\newcommand{\avint}{{\mathop{\,\rlap{-}\!\!\int}\nolimits}}
\newcommand{\N}{\mathbb{N}}

\newcommand{\R}{\mathbb{R}}

\newcommand{\ppi}{{\mbox{\boldmath$\pi$}}}
\newcommand{\sfd}{{\sf d}}

\newcommand{\sfq}{{\sf q}}
                           
\newcommand{\supp}{\mathop{\rm supp}\nolimits}   
  
\renewcommand{\d}{{\mathrm d}}

\newcommand{\restr}[1]{\lower3pt\hbox{$|_{#1}$}}
\newcommand{\la}{{\langle}}                  
\newcommand{\ra}{{\rangle}} 
\newcommand{\weakto}{\rightharpoonup}
\newcommand{\limi}{\varliminf}
\newcommand{\lims}{\varlimsup}
\newcommand{\e}{{\rm{e}}}                          
\newcommand{\X}{{\rm X}}
\newcommand{\Y}{{\rm Y}}
\newcommand{\lip}{{\rm lip}}
\newcommand{\Lip}{{\rm Lip}}

\renewcommand{\div}{{\rm div}}
\newcommand{\mm}{\mathfrak m}                             

\newcommand{\mIGC}[1]{({\mm}{\sf IGC})_{#1}}
\newcommand{\mEGC}[1]{({\mm}{\sf EGC})_{#1}}

\newcommand{\Hil}{\mathcal{H}}
\newcommand{\dom}{\mathrm{Dom}}
\newcommand{\PI}{{\sf PI}}

\newcommand{\proj}{\pi_{\rm f}}

\makeatletter
\newcommand{\mytag}[2]{%
  \text{#1}%
  \@bsphack
  \begingroup
    \@onelevel@sanitize\@currentlabelname
    \edef\@currentlabelname{%
      \expandafter\strip@period\@currentlabelname\relax.\relax\@@@%
    }%
    \protected@write\@auxout{}{%
      \string\newlabel{#2}{%
        {\color{black}#1}%
        {\thepage}%
        {\@currentlabelname}%
        {\@currentHref}{}%
      }%
    }%
  \endgroup
  \@esphack
}
\makeatother

\allowdisplaybreaks

\newtheorem{theorem}{Theorem}[section]

\newtheorem{corollary}[theorem]{Corollary}
\newtheorem{lemma}[theorem]{Lemma}
\newtheorem{proposition}[theorem]{Proposition}

\theoremstyle{definition}
\newtheorem{definition}[theorem]{Definition}

\newtheorem{example}[theorem]{Example}

\newcounter{Counter}

\newtheorem{remark}[theorem]{Remark}
\newcommand{\RCD}{{\sf RCD}}
\newcommand{\CD}{{\sf CD}}

\title{First-order heat content asymptotics on \texorpdfstring{$\RCD(K,N)$}{RCD(K,N)} spaces}

\author[Emanuele Caputo]{Emanuele Caputo}\address[Emanuele Caputo]{University of Jyvaskyla, Department of Mathematics and Statistics, P.O. Box 35 (MaD), FI-40014
University of Jyvaskyla, Finland}\email{emanuele.e.caputo@jyu.fi}
\author[Tommaso Rossi]{Tommaso Rossi}\address[Tommaso Rossi]{Institut f\"ur Angewandte Mathematik, Universit\"at Bonn, Endenicher Allee 60, 53115 Bonn}\email{rossi@iam.uni-bonn.de}

\date{\today}

\begin{document}

\begin{abstract}
In this paper, we prove first-order asymptotics on a bounded open set of the heat content when the ambient space is an ${\sf RCD}(K,N)$ space, under a regularity condition for the boundary that we call measured interior geodesic condition of size $\epsilon$. We carefully study such a condition, relating it to the properties of the disintegration of the signed distance function from $\partial \Omega$ studied in \cite{CM20}.
\end{abstract}


\maketitle
\tableofcontents

\section{Introduction}
The goal of this paper is to show a first-order heat content asymptotics in the setting of $\RCD(K,N)$ spaces. Let us state the problem in the setting of a smooth Riemannian manifold $(M,g)$ and $\Omega\subset M$ open and bounded. We consider the function $u \colon [0,\infty) \times \Omega \to \mathbb{R}$ satisfying (formally)
\begin{equation*}
\begin{cases}
\partial_t u(t,x)-\Delta u(t,x) = 0 &\text{ for all }(t,x) \in (0,\infty) \times \Omega,\\
u(t,x) = 0 & \text{ for all }(t,x) \in (0,\infty) \times \partial \Omega,\\
u(0,x) = 1 &\text{ for all } x \in \Omega.\\
\end{cases}
\end{equation*}
The heat content is the function defined as
\begin{equation}
\label{eq:heatcontent_introduction}
Q_{\Omega}(t):= \int_{\Omega} u(t,x)\,\d {\rm Vol}_g(x)\qquad\text{for every }t > 0.    
\end{equation}
In \cite{vdBG94}, Van den Berg and Gilkey proved the existence of a complete asymptotic expansion in $\sqrt{t}$ as $t\to0$, for open and bounded subsets $\Omega\subset M$ with smooth boundary. Moreover, they computed explicitly the coefficients of the expansion up to order $4$. Here, we report the expansion up to order $1$, at time $t = 0$:
\begin{equation}
\label{eq:heat_content_riemannian}
Q_{\Omega}(t)= {\rm Vol}_g(\Omega)-\sqrt{\frac{4t}{\pi}} \sigma_g(\partial \Omega)+\frac{t}{2}\int_{\partial \Omega}H\,\d \sigma_g+ o(t),\qquad\text{as }t\to 0,
\end{equation}
where $\sigma_g$ is the surface measure of $\partial \Omega$ and $H \colon \partial \Omega \to \mathbb{R}$ is the mean curvature of $\partial \Omega$ in $M$. The asymptotics \eqref{eq:heat_content_riemannian} unveil a deep connection between the geometry of $\partial\Omega$ and the small-time behavior of the heat content of $\Omega$, suggesting an effective strategy to investigate the curvature invariants of $\partial\Omega$, such as its mean curvature. This motivates our interest in studying such a problem in the non-smooth setting of $\RCD$ spaces. The asymptotic analysis of the heat content was initiated in the Euclidean setting in \cite{vdB-D,vdB-LG}. A first non-flat case was studied in \cite{vdB-hemisphere}, where the authors computed the heat content asymptotics to order $2$ for the upper hemisphere. For smooth domains in a Riemannian manifold, the existence of an asymptotic expansion in $\sqrt{t}$ at arbitrary order was established in \cite{vdBG94}, and the coefficients were computed iteratively in \cite{S98}. Recently, the heat content asymptotics has been proved for non-characteristic domains in the sub-Riemannian setting in \cite{RR21}. Finally, we also mention that similar problems have been studied in relation with different boundary conditions, see for example \cite{vdB-D-G,D-G,MR3358065,ARR21}.



\subsubsection*{\texorpdfstring{$\RCD$}{RCD} spaces}
In recent years, the theory of $\RCD$ spaces has experienced a surge of interest. Such spaces form a regular class of metric measure spaces where second-order calculus tools are available. The definition of $\RCD(K,N)$, for a metric measure space $(\X,\sfd,\mm)$, has been given in \cite{Gigli12}, by enforcing the curvature-dimension condition $\CD(K,N)$ with infinitesimal Hilbertianity. The $\CD$ condition has been introduced independently in the seminal works by Sturm in \cite{Sturm06I}, \cite{Sturm06II} and 
Lott-Villani in \cite{Lott-Villani09}, and it can be regarded as a synthetic notion for $(\X,\sfd,\mm)$ of having the Ricci curvature bounded below by $K\in\R$ and the dimension bounded above by $N\in (1,\infty]$.
While, the infinitesimal Hilbertianity ensures the linearity of the heat flow, ruling out Finsler geometries. Among its many merits, the $\RCD$ condition is consistent with the smooth Riemannian setting, and stable with respect to $pmGH$-convergence. Moreover, as a consequence of the results of \cite{Erbar-Kuwada-Sturm13} and \cite{AMS16}, the definition of $\RCD$ space can be equivalently formulated in terms of a distributional version of the $N$-Bochner inequality. 
For a complete historical account to the subject, we refer to the survey \cite{Ambrosio2018} and the previous theory on $\RCD(K,\infty)$ (introduced in \cite{AmbrosioGigliSavare11-2}). We recall that the class of spaces we consider includes the class of Ricci limit spaces, as introduced and studied in \cite{Cheeger-Colding97I},\cite{Cheeger-Colding97II},\cite{Cheeger-Colding97III}, and the class of Alexandrov spaces, when endowed with the appropriate Hausdorff measure, as proved in \cite{Petrunin11}.

\subsubsection*{The heat content on \texorpdfstring{$\RCD$}{RCD} spaces and main result}
For an $\RCD(K,N)$ space $(\X,\sfd,\mm)$, the Dirichlet heat flow of an open and bounded set $\Omega\subset\X$ can be defined in a classical way as the gradient flow of the local energy $E_{\Omega} \colon L^2(\Omega,\mm) \to [0,+\infty]$, where
\begin{equation*}
E_{\Omega}(u) := \begin{cases}
\int_{\Omega} |D u|^2\,\d \mm &\text{ if }u \in W^{1,2}_0(\Omega),\\
+\infty&\text{ otherwise.}
\end{cases}
\end{equation*}
Here $W^{1,2}_0(\Omega)$ is the set of local Sobolev functions with zero boundary condition. Then, letting $(0,\infty)\ni t\mapsto u_t\in W_0^{1,2}(\Omega)$ be the Dirichlet heat flow starting from $\chi_\Omega\in L^2(\Omega,\mm)$, we may define 
$Q_\Omega$ is defined in an analogous way to \eqref{eq:heatcontent_introduction}, namely 
\begin{equation}
    Q_{\Omega}(t):= \int_{\Omega} u_t(x)\,\d \mm(x)\qquad\text{for every }t > 0.  
\end{equation}
%
Our goal is to generalize \eqref{eq:heat_content_riemannian} to the $\RCD$ setting, however, already in smooth Riemannian manifolds, the regularity of $\partial \Omega$ comes into play. Therefore, we introduce a notion, called \emph{measured interior geodesic condition} at scale $\epsilon$ (in short $\mIGC{\epsilon}$ condition), which quantifies the regularity of the boundary of $\Omega$. More precisely, we say that $\partial\Omega$ satisfies the $\mIGC{\epsilon}$ condition if the geodesics minimizing the distance from the boundary, and of length at least $\epsilon$, cover the tubular neighborhood of size $\epsilon$, up to $\mm$-null sets (see Definition \ref{def:migc} for a precise statement). Define the signed distance from $\partial\Omega$, i.e. 
\begin{equation}
    \delta(x) := (\chi_{\Omega} -\chi_{\X \setminus \Omega})\,\sfd(x,\partial \Omega),\qquad\forall\,x\in\X,
\end{equation} 
and let us give further insights about the $\mIGC{\epsilon}$ condition, see Section \ref{sec:1D-localization} for details. 
\begin{itemize}
    \item The $\mIGC{\epsilon}$ condition of $\partial \Omega$ is an intermediate condition between an $\epsilon$-uniform interior ball condition and $\epsilon$-uniform interior ball condition plus an exterior ball condition at every point of $\partial \Omega$, cf. Proposition \ref{prop:ball_vs_IGC_epsilon};
    \item Consider the disintegration of $\mm$ associated with $\delta$ as studied in \cite{CM20}. This gives a `parameterization of the tubular neighborhood in normal directions' and we can characterize the $\mIGC{\epsilon}$ condition in terms of the fact that almost every transport rays inside $\Omega$ has length at least $\epsilon$, cf. Proposition \ref{lem:equivalence_migceps_lengthtransp}; 
    \item Under the $\mIGC{\epsilon}$ condition, the Laplacian of $\delta$, defined as a Radon functional can be represented as an $L^1$ function, in tubular neighborhood of size $\epsilon$ inside $\Omega$, cf. Corollary \ref{cor:mild_integrability_laplacian}.
\end{itemize}
Let state our main result, see Theorem \ref{thm:first_order_asymptotics_main}.
\begin{theorem}[First-order asymptotics]
\label{thm:first_order_asymptotics}
Let $(\X,\sfd,\mm)$ be an $\RCD(K,N)$ space for $K\in\R$ and $N\in (1,\infty)$ and let $\Omega\subset\X$ be open and bounded. 
Assume $\partial\Omega$ satisfies the $\mIGC{\epsilon}$ condition. Moreover, assume that there exists $\rho >0$ such that 
\begin{equation}
\label{eq:higher_int_delta_intro}
    \Delta\delta \in L^{1+\rho}(\{ 0 < \delta < \epsilon \}).
\end{equation}
Then, the heat content associated with $\Omega$ admits the following asymptotic expansion
\begin{equation}
    Q_\Omega(t)=\mm(\Omega)-\sqrt{\frac{4t}{\pi}}{\rm Per}(\Omega)+O\left(t^{\frac{2(1+\rho)-1}{2(1+\rho)}}\right)\qquad\text{as }t\to 0^+.
\end{equation}
\end{theorem}

Let us remark that, under the $\mIGC{\epsilon}$ condition for $\partial \Omega$, we have $\Delta \delta \in L^1(\{ 0 < \delta < \epsilon \})$ (see Corollary \ref{cor:mild_integrability_laplacian}); without any further assumption on the integrability of $\Delta \delta$, the error term appearing in the asymptotic expansion would be of order $\sqrt{t}$, preventing access to the first-order asymptotics. For this reason, it is natural to assume a slightly better integrability. On the other hand, from Theorem \ref{thm:first_order_asymptotics}, it seems necessary to assume $\Delta \delta \in L^{\infty}(\{ 0 < \delta < \epsilon \})$ in order to get a second-order asymptotics. Such a condition, together with the $\mIGC{\epsilon}$, is satisfied by a domain with a uniform interior and exterior ball condition, cf. Section \ref{sec:1D-localization}. 

The strategy we pursue was firstly proposed by Savo in \cite{S98} in the Riemannian setting and recently adapted in \cite{RR21} for obtaining the heat content asymptotics associated with non-characteristic domains in subriemannian manifolds. Indeed, the assumption of absence of characteristic points in \cite{RR21} acts as a regularity assumption on $\Omega$; in the non-smooth setting of $\RCD(K,N)$ spaces, a similar role is played by the $\mIGC{\epsilon}$ condition.

\subsubsection*{Strategy of the proof} 
Let us describe the main tools needed in order to prove Theorem \ref{thm:first_order_asymptotics}.
First of all, an important property of the solution to the Dirichlet heat equation $u_t$ in $\Omega$ concerns its relation with the heat flow (over the whole space) of $\chi_\Omega \in L^2(\mm)$. More generally, we prove a version of the Kac's principle of `not feeling the boundary' in infinitesimally Hilbertian metric measure spaces (see Corollary \ref{coro:kacs_principle}): for a nonnegative function $f\in L^\infty(\mm)$, given a compact set $K \Subset \Omega$, then
\begin{equation*}
\|h_t f- h_t^\Omega f \|_{L^1(K,\mm)} = o(t),\qquad\text{as }t\to 0^+,
\end{equation*}
where $h_tf$ denotes the heat flow of $f$ over $\X$. For details in the Riemannian setting, we refer the reader to \cite{H95}. Second of all, this property allows to decompose the term $Q_\Omega(t)$, and so the study of its asymptotics, in two terms: in a compact set in $\Omega$ and in a neighborhood the boundary. In a compact set inside $\Omega$, the Dirichlet heat flow `behaves' as the heat flow due to the aforementioned Kac's principle, hence it is negligible in the asymptotics. Thus, the only contribution to the asymptotic expansion of $Q_\Omega(t)$ comes from a neighborhood of the boundary, namely it is enough to study the small-time asymptotics of 
\begin{equation}
\label{eq:int_u_t_tub_neigh}
    \int_{\{0<\delta<\epsilon\}} u_t\,\d\mm,
\end{equation}
which depends on the regularity of $\partial \Omega$. To do so, we study a PDE associated with the quantity \eqref{eq:int_u_t_tub_neigh}. In particular, we show that a mean value lemma (cf. Proposition \ref{prop:Fsecond}) holds for the function 
\begin{equation}
    F(t,r):=\int_{\{0<\delta<r\}} u_t\,\d\mm,\qquad\text{for }t>0,\ r\in (0,\epsilon].
\end{equation}
We refer to \cite{S01} for the mean value lemma in the Riemannian setting. The former is a statement on the second distributional derivative of $F$ in the $r$-direction and permits to show that $F$ satisfies a suitable non-homogeneous heat equation, whose solution can be explicitly expressed via the Duhamel's principle (cf. Lemma \ref{prop:duhamel}).

\subsubsection*{Final comments}
The idea of considering the relation between the heat equation and the theory of sets of finite in perimeter is not new in the literature. In the setting of ${\sf RCD}$ spaces, it has been recently proved in \cite{BrenaPasqualettoPinamonti2022} that the short-time behaviour of the heat flow (over the whole space $\X$) can be used to characterize functions of bounded variation. This improves a previous result in the more general setting of ${\sf PI}$ spaces obtained in \cite{MaMiSh16}.  
%

Regarding future directions, we would like to improve Theorem \ref{thm:first_order_asymptotics} and obtain a second-order asymptotics. As mentioned before, in the setting of Riemannian manifolds, this coefficient is the integral of the mean curvature of $\partial \Omega$ with respect to the surface measure, see \eqref{eq:heat_content_riemannian}. In the non-smooth setting, we expect to be able to express the second-order coefficient in terms of the notion of mean curvature given in \cite{Ket20} (see also \cite{BKMW20}).
\subsection*{Structure of the paper}
The presentation is organized as follows.
In Section \ref{sec:preliminaries}, we recall some preliminaries on calculus on metric measure spaces, the definition of $\PI$ and $\RCD(K,N)$ spaces with some relevant properties for our purposes. In Section \ref{sec:heat_on_Omega}, we present the heat equation on $\Omega$ with Dirichlet boundary conditions and we prove the Kac's principle. Section \ref{sec:1D-localization} is devoted to 1-dimensional localization technique applied to the signed distance function from $\partial \Omega$, the definition of $\mIGC{\epsilon}$ condition and their relation. In Section \ref{sec:mean_value_lemma} we prove the mean value lemma in the non-smooth setting. Finally, Section \ref{sec:first_order_asymptotics} contains the proof of Theorem \ref{thm:first_order_asymptotics}.
\subsection*{Acknowledgments}
E.C. acknowledges support from the Academy of Finland Grant No. 314789 and the kind hospitality of University of Bonn.
T.R. acknowledges support from the Deutsche Forschungsgemeinschaft (DFG, German Research Foundation) through the collaborative research centre ``The mathematics of emerging effects'' (CRC 
1060, Project-ID 211504053) and the kind hospitality of University of Jyväskylä.
\section{Preliminaries}
\label{sec:preliminaries}
Throughout the paper, we consider a metric measure space $(\X,\sfd,\mm)$, i.e.\ $(\X,\sfd)$ is a complete and separable metric space and $\mm$ is a non-negative Borel measure, finite on bounded sets.
\subsection{Calculus in metric measure spaces}
We recall the notion of absolutely continuous curves with values in a metric space. Let $I\subset\R$ be an interval and denote by $C(I,\X)$ the metric space of continuous curves $\gamma \colon I\to \X$, endowed with the supremum distance. Then, the space of \emph{absolutely continuous curves} $AC(I,\X)$ is defined as the space of all $\gamma \in C(I,\X)$ for which there exists $ g \in L^1(I)$ such that $g\ge 0$ and
\begin{equation}
\label{eq:absolutely_cont_def}
    \sfd(\gamma_t,\gamma_s) \le \int_t^s g_r\,\d r,\qquad\text{for every }t<s \text{ in } I.
\end{equation} 
We say that $\gamma\in C(I,\R)$ is locally absolutely continuous on $I$ if $\gamma\in AC(J,\R)$, for any interval $J\subsetneq I$.
Given $\gamma \in AC(I,\X)$, we can define the \emph{metric speed} $|\dot\gamma_t|$ of $\gamma$ for a.e. $t$, namely we can prove that
\begin{equation}
    \exists\,\lim_{h\to 0}\frac{\sfd(\gamma_{t+h},\gamma_t)}{|h|}=:|\dot\gamma_t|\quad \text{for a.e. }t\in I \qquad\text{and}\qquad |\dot\gamma_t| \in L^1(I).
\end{equation} 
Therefore, we can define the \emph{length} of a curve $\gamma\in AC(I,\R)$ as $\ell(\gamma) = \int_I |\dot\gamma_t|\,\d t$. Moreover, for $p\in (1,\infty)$, we say that $\gamma\in AC^p(I,\X)$ if \eqref{eq:absolutely_cont_def} holds with $g\in L^p(I)$. In that case, the metric speed is in $L^p(I)$. Finally, note that, if $\X=\Hil$ is a Hilbert space, then a curve $\gamma\in AC^p(I,\Hil)$ is differentiable (in a classical sense) for a.e. $t\in I$ and the metric speed is the norm of its derivative.
\subsubsection{Sobolev spaces on metric measure spaces} For any metric space $(\Y,\sfd_\Y)$, let $\mathscr{P}(\Y)$ be the set of probability measures on $\Y$. Denote by $\e_t \colon C([0,1],\X) \to \X$; $e_t(\gamma)=\gamma_t$ the evaluation map at time $t\in [0,1]$. We say that $\ppi \in \mathscr{P}(C([0,1],\X))$ is a \emph{test plan} if:
\begin{itemize}
    \item[i)] there exists a constant $C>0$ such that $(\e_t)_\#\ppi \le C\mm$, for any $t\in [0,1]$;
    \item[ii)] $\ppi$ is concentrated on $AC^2([0,1],\X)$ and 
    \begin{equation}
        \int \int_0^1 |\dot{\gamma}_t|^2\,\d t\,\d \ppi(\gamma)<+\infty.
    \end{equation}
\end{itemize}
Let $f \colon \X \to \R$ be a Borel function. We say that $f$ belongs to the \emph{Sobolev class} and write $f\in S^2(\X)$ if there exists $0 \le G \in L^2(\mm)$ such that, for every test plan $\ppi$, 
\begin{equation}
\label{eq:weak_up_grad_def}
    \int |f(\gamma_1)-f(\gamma_0)|\,\d \ppi(\gamma)\le \int \int_0^1 G(\gamma_t)|\dot{\gamma}_t|\,\d t\,\d \ppi(\gamma).
\end{equation} 
Such a function $G$ is called a weak upper gradient of $f$. We may define the \emph{minimal weak upper gradient} of $f$, denoted by $|Df|$, as the $\mm$-a.e. minimal non-negative function satisfying \eqref{eq:weak_up_grad_def}, namely:
\begin{equation}
    |D f| \le G \quad \mm\text{-a.e.}\quad \text{for every weak upper gradient }G.
\end{equation}
Finally, we define the \emph{Sobolev space} on $\X$ as $W^{1,2}(\X):=L^2(\mm) \cap S^2(\X)$, endowed with the norm:
\begin{equation}
\label{eq:2_sobolev_norm}
    \| f \|_{W^{1,2}(\X)}:= \| f \|_{L^2(\mm)}+ \| |D f| \|_{L^2(\mm)}.
\end{equation}
Note that, in general, $W^{1,2}(\X)$ is only a Banach space. Moreover, denoting by $\Lip(\X)$ the set of Lipschitz functions on $\X$, $\Lip(\X)\cap W^{1,2}(\X)$ is dense in $W^{1,2}(\X)$ with respect to the norm \eqref{eq:2_sobolev_norm}, as a consequence of the density in energy of Lipschitz functions, cf. \cite{AmbrosioGigliSavare11-3}. 
Define the \emph{slope} of a function $f \colon \X \to \mathbb{R}$ as follows:
\begin{equation}
\lip(f)(x):=    
    \begin{cases}
        \displaystyle{\lims_{y \to x} \frac{|f(x)-f(y)|}{\sfd(x,y)}} & \text{if } x\in\X \text{ is not isolated}, \\
        0 & \text{otherwise.} 
    \end{cases}
\end{equation}
One can check that if $f\in \Lip(\X)$, then $\lip (f) \le \Lip (f)$, where $\Lip (f)$ is the Lipschitz constant of $f$.
We recall the definition of the Cheeger energy by means of relaxation of the $L^2$-norm of the local Lipschitz constant with respect to $L^2$-convergence. We define ${\rm Ch} \colon L^2(\mm) \to [0,+\infty]$ as
\begin{equation}
\label{eq:cheeger_energy}
    {\rm Ch}(f):= \inf \left\{ \limi_{n \to +\infty} \frac{1}{2}\int (\lip f_n)^2\, \d \mm\,,\, f_n \to f \text{ in }L^2(\mm) \text{ and } f_n \in \Lip_{loc}(\X)\right\},
\end{equation} 
where $\Lip_{loc}(\X)$ is the set of locally Lipschitz functions on $\X$. It is proved in \cite[Thm.\ 6.2]{AmbrosioGigliSavare11} that $W^{1,2}(\X)= \{ f \in L^2(\mm) \,\text{s.t. }{\rm Ch}(f)< \infty \}$
and
\begin{equation}
\label{eq:representation_cheeger}
    {\rm Ch}(f) = \frac{1}{2}\int_\X |D f|^2\,\d \mm. 
\end{equation}

We conclude this paragraph recalling the definition of infinitesimally Hilbertian metric measure space, firstly introduced and studied in \cite{Gigli12}.

\begin{definition}[Infinitesimal Hilbertianity]
We say that $(\X,\sfd,\mm)$ is \emph{infinitesimally Hilbertian} if $W^{1,2}(\X)$ is a Hilbert space. 
\end{definition}

\subsubsection{The language of normed modules}
We introduce the notion of differential and gradient of a function in metric measure spaces. We assume the reader to be familiar with the language of $L^2(\mm)$-normed $L^\infty(\mm)$-module, see \cite{Gigli14} for further details. In particular, as proved in \cite[Thm.\ 2.8]{Gigli14}, there exists a couple $(L^2(T^*\X),\d)$ such that $L^2(T^*\X)$ is an $L^2(\mm)$-normed $L^\infty(\mm)$-module and
\begin{itemize}
    \item[i)] $\d \colon S^{2}(\X) \to L^2(T^*\X)$ is linear and $|\d f|_* = |D f|$ $\mm$-a.e., where $|\cdot|_*\colon L^2(T^*\X) \to L^2(\mm)$ denotes the pointwise norm;
    \item[ii)] The set $\{ \sum_{i=1}^n \chi_{E_i} \d f_i\mid f_i \in S^{2}(\X), E_i\subset\X \text{ Borel}\}$ is dense in $L^2(T^*\X)$.
\end{itemize}
Moreover, the couple is unique in the sense that, given another couple $(\mathscr{M},\tilde{\d})$ verifying items i) and ii) as above, then there exists an isomorphism $\Phi \colon L^2(T^*\X) \to \mathscr{M}$ of normed modules which preserves the pointwise norm. $L^2(T^*\X)$ is called the \emph{cotangent module} and $\d$ is the \emph{differential}.
We report below some of the calculus tools for the differential (see \cite[Sec.\ 2]{Gigli14}):
\begin{description}
    \item[Leibniz rule.] Let $f,g \in S^{2}(\X) \cap L^\infty(\mm)$, then $f g \in S^{2}(\X)$ and $\d (fg) = f \d g+ \d f g$;
    \item[Chain rule.] Let $f \in S^{2}(\X)$ and $\varphi \in \Lip(\mathbb{R})$, then $\varphi \circ f \in S^{2}(\X)$ and $\d (\varphi \circ f) = \varphi' \circ f\,\d f$;
    \item[Closure of the differential.] Let $\{f_n\} \subset S^{2}(\X)$ be such that $f_n \to f$ $\mm$-a.e., for some Borel $f$ and assume $\d f_n \weakto w$ weakly in $L^2(T^*\X)$ for some $w \in L^2(T^*\X)$, then $f \in S^{2}(\X)$ and $\d f =w$.
    \item[Locality.] Let $f,g \in S^{2}(\X)$, then $\chi_{\{f = g\}} \d f= \chi_{\{f = g\}} \d g$.
\end{description}
We define the \emph{tangent module} $L^2(T\X)$ as the dual (in the sense of modules) of $L^2(T^*\X)$. Note that $(\X,\sfd,\mm)$ is infinitesimally Hilbertian if and only if $L^2(T^*\X)$ and $L^2(T\X)$ are Hilbert $L^2(\mm)$-normed $L^\infty(\mm)$-module. In this case, there exists an isomorphism $I \colon L^2(T^*\X) \to L^2(T\X)$ of normed modules which preserves the pointwise norm and it is possible to define the gradient of a Sobolev function as the dual of its differential.

\begin{definition}[Gradient]
Let $(\X,\sfd,\mm)$ be an infinitesimally Hilbertian metric measure space. Then, the \emph{gradient} is the operator $\nabla:=I\circ\d\colon S^2(\X)\rightarrow L^2(T\X)$. Equivalently, for $f\in S^2(\X)$, its gradient $\nabla f\in L^2(T\X)$ is the unique vector field satisfying
\begin{equation}
    \d f(\nabla f)=|\nabla f|^2=|\d f|_*^2,\qquad\mm\text{-a.e.},
\end{equation}
where $|\cdot|\colon L^2(T\X) \to L^2(\mm)$ denotes the pointwise norm on $L^2(T\X)$.
\end{definition}
The operator $\nabla$, being the dual of the differential, enjoys the same properties as $\d$. Moreover, it is possible to prove that for $f,g\in S^2(\X)$, one has 
\begin{equation}
\label{eq:scalar_product_gradients}
    \d f(\nabla g)=\d g(\nabla f)=\lim_{\epsilon \to 0} \frac{|D (f+\epsilon g)|^2-|Df|^2}{\epsilon} ,\qquad\mm\text{-a.e.}.
\end{equation}
Identity \eqref{eq:scalar_product_gradients} provides a notion of pointwise scalar product of gradients setting $\la\nabla f,\nabla g\ra:=\d f(\nabla g)$, and in addition, $\la\nabla f,\nabla g\ra\in L^1(\mm)$. In particular, it holds that
\begin{equation}
    |\la \nabla f, \nabla g \ra| \le |D f|\,|D g|,\quad \mm\text{-a.e.,}\qquad \text{provided }f,g \in S^2(\X).
\end{equation}

Let $\mu$ be a Borel non-negative measure. We define $L^0(\mu)$ as the vector space of Borel measurable functions, modulo equivalence $\mu$-a.e.. We define $L^0(T\X)$ as the completion of $L^2(T\X)$ with respect to the distance
\begin{equation}
\sfd_{L^0}(v,w):= \int |v-w| \wedge 1\,\d \tilde{\mm}, \quad \text{for every }v,w \in L^2(T\X),
\end{equation}
where $\tilde{\mm} \in \mathscr{P}(\X)$ is such that $\tilde{\mm} \ll \mm \ll \tilde{\mm}$.
Note that we regard $L^0(T\X)$ as a topological vector space, with topology induced by $\sfd_{L^0}$ (which is affected by the choice of $\tilde{\mm}$ but not its induced topology).
In particular, $L^0(T\X)$ can be endowed with the structure of $L^0(\mm)$-normed $L^0(\mm)$-module. We define 
\begin{equation}
    L^p(T\X):= \{ v \in L^0(T\X) \mid |v| \in L^p(\mm) \},\qquad \text{for any }p \in [1,\infty].
\end{equation} 
This set can be endowed with the structure of $L^p(\mm)$-normed $L^\infty(\mm)$-module.
It is convenient to introduce the notion of \emph{restriction} of normed modules. Given $E \in \mathscr{B}(\X)$, we define \begin{equation}
    L^p(T\X)\restr{E}:=\{ \chi_E v\mid v \in L^p(T\X) \}\qquad\text{ for any }p \in \{ 0\} \cup [1,\infty].
\end{equation}
\subsubsection{Gradient flows on Hilbert spaces}
Let $\left(\Hil,\la\cdot,\cdot\ra\right)$ be a Hilbert space and $E\colon \Hil\rightarrow (-\infty,+\infty]$ be a convex functional on $\Hil$. We denote by $\dom(E)=\{u\in\Hil\mid E(u)<\infty\}$ the domain of $E$. The \emph{subdifferential} of $E$ at $u \in \dom(E)$, is defined as
\begin{equation}
    \partial^{-} E(u):=\{z \in \Hil \mid I(v) \ge I(u)+ \la z, v-u \ra \text{ for every }v \in \Hil \}.
\end{equation}
We recall here some properties of the theory of gradient flows of convex and lower semicontinuos functionals on Hilbert spaces, see \cite{B73book}, \cite{K67} for further details.
\begin{theorem}[Gradient flow on Hilbert spaces]
\label{thm:gradient_flows}
Let $E \colon \Hil \to [0,+\infty]$ be a convex and lower semicontinuous functional and let $x \in \overline{\dom(E)}$. Then, there exists a unique locally absolutely continuous curve $[0,+\infty) \ni t \mapsto x_t \in \Hil$ such that
\begin{equation}
    x_0=x \qquad\text{and}\qquad \dot x_t \in -\partial^- E(x_t)\quad \text{for a.e. } t \in [0,+\infty).
\end{equation}
Such a curve is called the \emph{gradient flow of $E$ starting from $x$}. 
\end{theorem}

Using Theorem \ref{thm:gradient_flows}, we can define the \emph{heat flow} on a metric measure space. We introduce the notion of divergence and Laplacian.

\begin{definition}[Divergence]
\label{def:divergence}
Let $(\X,\sfd,\mm)$ be infinitesimally Hilbertian. 
We say that an element $w \in L^2(T\X)$ belongs to $D(\div)$ if there exists $h \in L^2(\mm)$ such that
\begin{equation}
\label{eq:divergence}
    -\int g h \,\d \mm = \int \d g(w)\,\d \mm,\qquad\text{for every } g \in W^{1,2}(\X).
\end{equation}
We set $\div\, w:=h$. Note that $h$ is uniquely determined since $W^{1,2}(\X)$ is dense in $L^2(\mm)$.
\end{definition}
The divergence operator as defined above is linear. We recall a Leibniz formula for the divergence operator, whose proof is a straightforward modification of \cite[Prop.\ 4.2.7]{GP19}.
\begin{proposition}
\label{prop:leibniz_divergence}
Let $(\X,\sfd,\mm)$ be infinitesimally Hilbertian. Let $w\in L^\infty(T\X) \cap D(\div)$ and let $f \in L^\infty(\mm) \cap W^{1,2}(\X)$. Then, $fw\in D(\div)$ and
\begin{equation}
\label{eq:leibniz_divergence}
    \div(fw) = f\div\,w+\la \nabla f,w\ra, \qquad\mm\text{-a.e. on }\X.
\end{equation}
\end{proposition}

\begin{definition}[Laplacian]
\label{def:laplacian}
Let $(\X,\sfd,\mm)$ be infinitesimally Hilbertian and let $f \in W^{1,2}(\X)$. We say that $f \in D(\Delta)$, if there exists $h \in L^2(\mm)$ such that
\begin{equation}
    \label{eq:definition_laplacian}
    \int_{\X} h g\, \d \mm  = -\int_{\X} \la \nabla f, \nabla g \ra\,\d \mm \qquad \text{for every }g\in W^{1,2}(\X).
\end{equation}
We set $\Delta f:= h$, where $h$ is uniquely determined by the density.
\end{definition}
Consider the Cheeger energy ${\rm Ch}$ defined in \eqref{eq:cheeger_energy}, which is convex and $L^2(\mm)$-lower semicontinuous. 
Let $f \in L^2(\mm)$, then the gradient flow $t \mapsto h_t f$ of ${\rm Ch}$ starting from $f$ (which exists and is unique by Theorem \ref{thm:gradient_flows}) is called the \emph{heat flow} of $f$. In particular, we have that $\partial_t h_t f \in -\partial^- {\rm Ch}(h_t f)$ for a.e.\ $t$. We can characterize the subdifferential of the Cheeger energy in terms of the Laplacian. Indeed, if $f \in W^{1,2}(\X)$, then $f \in D(\Delta)$ if and only if $\partial^- {\rm Ch}(f) \neq \emptyset$. In this case, we have $\partial^{-} {\rm Ch}(f) = \{ - \Delta f \}$.
 Therefore, the curve $[0,\infty) \ni t \mapsto h_t f \in L^2(\mm)$ is locally absolutely continuous and
\begin{equation}
    \partial_t h_t f = \Delta h_t f\qquad \text{for a.e.\ }t>0.
\end{equation}
\subsection{\texorpdfstring{$\PI$}{PI} spaces}
We say that $(\X,\sfd,\mm)$ is \emph{locally uniformly doubling} if, for every $R >0$, there exists $C_D(R) >0$ such that for every $x \in \X$ and $r \le R$ we have
\begin{equation}
\label{eq:doubling_def}
\mm(B_{2r}(x)) \le C_D\mm(B_{r}(x))
\end{equation}
We refer to $C_R$ as the doubling constant up to scale $R$. A consequence of the definition of the doubling assumption is that $(\X,\sfd)$ is proper, i.e.\ closed and bounded sets are compact. We say that $(\X,\sfd,\mm)$ satisfies a \emph{weak local (1-1) Poincar\'{e} inequality} provided for every $R>0$ there exists $C_P(R)>0$ and $\lambda \ge 1$ such that for every $f \colon \X \to \R$ Lipschitz, $x \in \X$, $0<r<R$ we have 
$$ \avint_{B_r(x)} |f-f_{B_r(x)}|\,\d \mm \leq C_P r \avint_{B_{\lambda r}(x)} \lip(f)\,\d \mm,$$
where $f_{B_r(x)}:=\frac{1}{\mm(B_r(x))}\int_{B_r(x)} f \,\d \mm$.
\begin{definition}[$\PI$ space]
We say that $(\X,\sfd,\mm)$ is a $\PI$ space if it is locally uniformly doubling and satisfies a weak local (1-1) Poincar\'{e} inequality.
\end{definition}
This class of spaces is relevant for our presentation because the results in \cite{Cheeger00} apply. In particular, as proven therein, given $f \in W^{1,2}(\X) \cap \Lip(\X)$, we have 
\begin{equation}
|D f| = \lip f \qquad\mm\text{-a.e.,}    
\end{equation}
while the inequality $\leq$ holds in general metric measure spaces.

\subsubsection{Signed distance function in metric measure spaces}
We denote by $\sfd_x(y):= \sfd(x,y)$. Since the map is $1$-Lipschitz, we have that $\lip \,\sfd_x \le 1$; on the other hand, if $(\X,\sfd)$ is a length space, we have that $\lip \, \sfd_x \ge 1$, whence $ \lip \,\sfd_x  = 1$. If we assume that $(\X,\sfd,\mm)$ is a $\PI$ space, due to \cite{Cheeger00}, we get that $|D \sfd_x| = 1$.
By similar arguments, we can prove the same result for the signed distance function from the boundary of an open set.
\begin{definition}[Signed distance]
Let $(\X,\sfd)$ be a metric space and let $\Omega\subset \X$ be open. Denote by $\sfd(x,\partial\Omega)=\inf\{ \sfd(x,y)\mid y\in \partial\Omega\}$, for any $x\in \X$. Then, the signed distance function from $\partial\Omega$ is
\begin{equation}
\label{eq:signed_distance_fun}
    \delta\colon \X\rightarrow\R;\qquad \delta(x):=\sfd(x,\partial \Omega) \chi_{\Omega}(x)-\sfd(x,\partial\Omega) \chi_{\X \setminus {\Omega}}(x).
\end{equation}
\end{definition}
\begin{proposition}
\label{prop:eikonal_weak}
Let $(\X,\sfd,\mm)$ be a geodesic $\PI$ space. Let $\Omega\subset \X$ be open and consider $\delta$ defined as in \eqref{eq:signed_distance_fun}. Then, we have that $\delta \in \Lip(\X)$ and 
\begin{equation}
\label{eq:eikonal_weak}
|D \delta| = 1\quad\mm\text{-a.e.}
\end{equation}
\end{proposition}
\begin{proof}
For convenience, let us denote by $\delta^+:=\sfd(\cdot,\partial\Omega)$. We claim that $\lip\,\delta^+\equiv 1$.
By triangle inequality $\delta^+$ is $1$-Lipschitz, thus, as before, we have that $\lip\, \delta^+(y) \le 1$ for every $y\in \X$. Let us prove the converse inequality for fixed $y \in \X$. Let us consider $R>0$ such that $B_R(y) \cap \partial \Omega \neq \emptyset$. Since $\PI$ spaces are proper, $\overline{B_R(y)} \cap \partial \Omega$ is compact, hence there exists $z =z_y \in \partial\Omega$ such that $\sfd(y,z) = \delta^+(y)$. We consider a geodesic $\gamma\colon[0,1]\to \X$ such that $\gamma_0 = z$ and $\gamma_1 = y$. Then, since by definition, $\delta^+(\gamma_t)\le \sfd(\gamma_t,z)$, we deduce that
\begin{equation}
    1 \ge \frac{|\delta^+(\gamma_t)-\delta^+(y)|}{\sfd(\gamma_t,y)} \ge \frac{-\delta^+(\gamma_t)+\sfd(y,z)}{\sfd(\gamma_t,y)}\ge \frac{-\sfd(\gamma_t,z)+\sfd(y,z)}{\sfd(\gamma_t,y)} =1,
\end{equation} 
thus having that $\lip \,\delta^+(y) \ge 1$ for every $y\in\X$ and proving the claim. 
The fact that $|D \delta^+| = 1$ $\mm$-a.e.\ follows by the fact that on $\PI$ spaces $\lip\,\delta^+ = |D \delta^+|$ $\mm$-a.e.\ (see \cite{Cheeger00}).
To prove the result for $\delta$, we argue as follows. By definition, we have that 
\begin{equation}
    \delta(y) =
\begin{cases}
            \delta^+(y) &\text{if }y\in\Omega,\\            -\delta^+(y) &\text{if }y\in\X\setminus\Omega,
\end{cases}
\end{equation}
hence we conclude applying the locality of the minimal weak upper gradient on $\Omega$ and $\X \setminus \Omega$, yielding $|D \delta| = |D \delta^+|$ $\mm$-a.e.\ and thus concluding the proof.
\end{proof}

\subsubsection{Sets of finite perimeter on metric measure spaces}
We recall the definition of a set of finite perimeter, following \cite{Mir03}.
\begin{definition}[Perimeter and sets of finite perimeter]
Let $(\X,\sfd,\mm)$ be a metric measure space and let $E \subset \X$ be Borel, $U\subset \X$ be open. The perimeter of $E$ in $U$, ${\rm Per}(E,U)$ is defined as
\begin{equation*}
    {\rm Per}(E,U):= \inf \left\{ \limi_{n \to \infty} \int_U {\rm lip}\,f_n\,\d \mm:\,u_n \in {\rm Lip}_{loc}(U),\, u_n \to \chi_E \in L^1_{loc}(U,\mm)\right\}.
\end{equation*}
We say that $E$ is a set of finite perimeter if ${\rm Per}(E,\X) < \infty$.
\end{definition}
As shown \cite{Mir03}, the set function ${\rm Per}(E,\cdot)$ defined on open sets is the restriction of a Borel measure, defined as
\begin{equation*}
    {\rm Per}(E,B):= \inf \{ {\rm Per}(E,U)\mid B \subset U,\, U\subset\X \text{ open} \}\,\qquad\text{for every }B \in \mathscr{B}(\X).
\end{equation*}
In a metric measure space $(\X,\sfd,\mm)$, given a set $E \subset \X$, we define its capacity as
\begin{equation*}
    {\rm Cap}(E):=\inf \left\{ \| f \|_{W^{1,2}(\X)} : f \in W^{1,2}(\X) \text{ and } f \ge 1\quad \mm\text{-a.e. on a neighborhood of }E\right\}
\end{equation*}
If $(\X,\sfd,\mm)$ is a $\PI$ space, and $\Omega\subset\X$ is an open set of finite perimeter, then
\begin{equation}
    \label{eq:perllcap}
    {\rm Per}(\Omega,\cdot)\ll {\rm Cap}.
\end{equation}
This is a consequence of two facts:
\begin{enumerate}
    \item ${\rm Per}(\Omega,\cdot) \ll \mathscr{H}^{{\rm cod\text{-}1}}\restr{\partial^e \Omega} \ll \mathscr{H}^{{\rm cod\text{-}1}}$ from \cite[Thm.\ 5.3]{Ambrosio02};
    \item ${\mathscr{H}}^{{\rm cod\text{-}1}} \ll {\rm Cap}$ from \cite[Thm.\ 1.12]{BPS19}.
\end{enumerate}
The measure $\mathscr{H}^{{\rm cod\text{-}1}}$ is the codimension-one Hausdorff measure and $\partial^e\Omega$ is the essential boundary of $\Omega$, as considered in \cite{Ambrosio02}. Since, for what concerns this work, they enter into play only for the above mentioned results, we refer the reader to \cite{Ambrosio02} for their definitions and properties.

\subsection{\texorpdfstring{$\RCD(K,N)$}{RCD(K,N)} spaces and second onder calculus}
\label{sec:second_order_calculus}
We introduce the definition of a more regular class of metric measure spaces, the so-called $\RCD(K,N)$ spaces, firstly introduced in \cite{Gigli12} enforcing the $\CD(K,N)$ with infinitesimally Hilbertianity. For the present (equivalent) formulation, see \cite{Erbar-Kuwada-Sturm13}.

\begin{definition}[$\RCD(K,N)$ spaces]
\label{def:rcdKinfty}
Let $K \in \mathbb{R}$ and $N \in (1,\infty)$. Then a metric measure space $(\X,\sfd,\mm)$ is an $\RCD(K,N)$ space if the following conditions hold:
\begin{itemize}
\item[i)] there exists $C >0$ and a point $x \in \X$ such that $\mm(B_r(x)) \le Ce^{C r^2}$, for every $r>0$;
\item[ii)] Sobolev-to-Lipschitz property: for every $f \in W^{1,2}(\X)$ with $|D f| \in L^\infty(\mm)$, there exists a Lipschitz function $\tilde{f}$ which is a representative of $f$ and $\||D f|\|_{L^\infty(\mm)} = \Lip(\tilde{f})$;
\item[iii)] the space $(\X,\sfd,\mm)$ is infinitesimally Hilbertian;
\item[iv)] for every $f \in D(\Delta)$ and $0 \le g \in D(\Delta) \cap L^\infty(\mm)$ such that $\Delta f \in W^{1,2}(\X)$ and $\Delta g \in L^\infty(\mm)$, it holds that
\[ \frac{1}{2} \int |D f|^2\,\Delta g \, \d \mm \ge \int \left( \frac{(\Delta f)^2}{N}+\la \nabla f, \nabla \Delta f \ra + K |D f|^2 \right) g \, \d \mm. \]
\end{itemize}
\end{definition}
We point out that $\CD(K,N)$ spaces are $\PI$ spaces. This is a consequence of the fact that they are locally uniformly doubling, since in this setting Bishop-Gromov monotonicity formula holds (see \cite{Sturm06II}), and satisfy a local Poincar\'{e} inequality (see \cite{Rajala12}).
Moreover, it follows from the definition that ${\sf CD}(K,N)$ spaces are geodesic, hence also length spaces.
In the case of $\RCD(K,N)$ spaces, we have at disposal functions with more regularity.
In particular, we recall the existence of good cut-off functions, as proved in \cite{AMS16}.
\begin{proposition}[{\cite[Lem.\ 6.7]{AMS16}}]
\label{good_cut_off}
Let $(\X,\sfd,\mm)$ be an ${\sf RCD}(K,N)$ space with $K\in\R$ and $N \in (1,\infty)$. Then, for every compact $E \subset \X$ and open and relatively compact $G\subset \X$ such that $E \subset G$, there exists $\varphi \colon \X \to \mathbb{R}$ such that $0 \le \varphi \le 1$ and $\varphi = 1$ in a neighborhood of $E$ and ${\rm supp}\varphi \subset G$. Moreover, $\Delta \varphi \in L^\infty(\mm)$ and $|D \varphi| \in W^{1,2}(\X)$.
\end{proposition}
We report here a result on the structure of geodesics in $\RCD$ spaces, improving a previous result of \cite{TapioKTS14}. Recall that we say two geodesics $\gamma_1,\gamma_2 \colon [0,1] \to \X$ branch if 
\begin{equation}
\gamma_1 \restr{[0,t]} = \gamma_{2} \restr{[0,t]},\qquad\text{for some }t\in (0,1), 
\end{equation}
but $\gamma_1 \neq \gamma_2$. We say that a metric space $(\X,\sfd)$ is \emph{non-branching} if no couple of geodesics branch. 

\begin{theorem}[{\cite[Thm.\ 1.3]{Deng20}}]
\label{thm:nonbranch_rcd}
Let $(\X,\sfd,\mm)$ be an $\RCD(K,N)$ space for some $K \in \mathbb{R}$ and $N \in (1,\infty)$. Then $(\X,\sfd,\mm)$ is non-branching.
\end{theorem}

\section{The Dirichlet heat flow on \texorpdfstring{$\Omega$}{Omega} and the Kac's principle}
\label{sec:heat_on_Omega}
\subsection{Cheeger energy and local Sobolev spaces}
Let $(\X,\sfd,\mm)$ be a metric measure space and let $\Omega \subset \X$ be open. 
For $f\in L^2(\Omega,\mm)$, we define its \emph{Cheeger energy} on $\Omega$ as 
\[ {\rm Ch}_{\Omega}(f):= \inf \left\{ \limi_{n \to \infty} \frac{1}{2} \int_\Omega (\lip f_n)^2\, \d \mm\,,\, f_n \to f \text{ in }L^2(\Omega,\mm) \text{ and } f_n \in \Lip_{loc}(\Omega)\right\}.\]
Then, the \emph{local Sobolev space} $W^{1,2}(\Omega)$ is the set where ${\rm Ch}_\Omega$ is finite, i.e.
$$W^{1,2}(\Omega):=\{ f \in L^2(\Omega,\mm) \,\text{s.t. }{\rm Ch}_{\Omega}(f)< \infty \}$$
endowed with the following norm:
\begin{equation}
\label{eq:loc_sob_norm}
    \|f\|^2_{W^{1,2}(\Omega)}:=\|f\|^2_{L^2(\Omega,\mm)}+{\rm Ch}_{\Omega}(f),\qquad f\in W^{1,2}(\Omega).
\end{equation}
Using the $L^2(\Omega,\mm)$-lower semicontinuity of ${\rm Ch}_{\Omega}$, arguing similarly as in the case of $W^{1,2}(\X)$, one can show that $(W^{1,2}(\Omega), \| \cdot \|_{W^{1,2}(\Omega)})$ is a Banach space. Our definition is equivalent to the one given in \cite[Def.\ 2.14]{AmbrosioHonda18}, as proven in the appendix in Theorem \ref{thm:equivalence_local_sobolev}.
Moreover, the equivalence between the definition by Ambrosio-Honda and the one due to Cheeger in \cite[Def.\ 2.2]{Cheeger00} is proved in \cite[Rmk.\ 2.15]{AmbrosioHonda18}.
We also set 
\begin{equation}
    W^{1,2}_0(\Omega):= \overline{\Lip_{bs}(\Omega)}^{W^{1,2}(\Omega)},
\end{equation}
where $\Lip_{bs}(\Omega)$ is the set of Lipschitz functions with bounded support in $\Omega$. Our goal is to define a Dirichlet energy on the space $W^{1,2}_0(\Omega)$. To do so, we introduce the following extension operator: 
\begin{equation}
\label{eq:ext_op}
    Tf=g, \qquad\text{where }g(x)=
    \begin{cases}
        f(x)\quad\text{if }x\in\Omega,\\
        0\qquad\ \text{otherwise},
    \end{cases}\qquad\forall\,f\in\Lip_{bs}(\Omega).
\end{equation}
%

\begin{proposition}
\label{prop:extension_op}
Let $(\X,\sfd,\mm)$ be a metric measure space and let $T\colon \Lip_{bs}(\Omega) \rightarrow W^{1,2}(\X)$ be the linear operator defined in \eqref{eq:ext_op}. Then, $T$ extends to an isometry from $\left(W^{1,2}_0(\Omega),\|\cdot\|_{W^{1,2}(\Omega)}\right)$ to $\left(W^{1,2}(\X),\|\cdot\|_{W^{1,2}(\X)}\right)$.
\end{proposition}
\begin{proof}
Let $f\in\Lip_{bs}(\Omega)$. Since outside of $\Omega$, the function $Tf$ is identically zero the $L^2$-norm is preserved, namely $\|Tf\|_{L^2(\mm)}=\|f\|_{L^2(\Omega,\mm)}$. Now, we claim that,
\begin{equation}
\label{eq:claim}
    {\rm Ch}_{\Omega}(f)={\rm Ch}(Tf),\qquad \forall\,f\in\Lip_{bs}(\Omega).
\end{equation}
The inequality ${\rm Ch}_{\Omega}(f)\leq{\rm Ch}_{\X,2}(Tf)$ is trivial since any sequence of locally Lipschitz function on $\X$ restricts to a sequence of locally Lipschitz functions on $\Omega$ with the same slope. The converse inequality can be proved as follows. Let $\{f_n\}\subset \Lip_{loc}(\Omega)$ such that
\begin{equation}
\label{eq:competitor_Ch}
    f_n \xrightarrow[L^2(\Omega,\mm)]{} f\qquad\text{and}\qquad {\rm Ch}_{\Omega}(f)=\limi_{n \to \infty} \frac{1}{2} \int_\Omega (\lip f_n)^2\, \d \mm.
\end{equation}
Consider a cut-off function $\varphi\in\Lip(\X)$ such that $\varphi\equiv 1$ on $\supp{f}$ and $\varphi\equiv 0$ on the complement of $\Omega$. Define the auxiliary sequence $\tilde f_n=\varphi f_n$ and we shall prove that the sequence $\{\tilde f_n\}$ is a competitor for the Cheeger energy of $Tf$ in $\X$. Indeed, first of all, $\tilde f_n\to Tf$ in $L^2(\mm)$. Secondly, $\tilde f_n$ satisfies the pointwise Leibniz rule
\begin{equation}
    \lip \tilde f_n(x)\leq \varphi(x) \lip f_n(x)+f_n(x)\lip\,\varphi(x),\qquad \forall\,x\in\Omega,
\end{equation}
therefore, applying the Young's inequality, we obtain for any $\varepsilon>0$, 
\begin{equation}
\label{eq:young_ineq}
    (\lip \tilde f_n(x))^2\leq (1+\varepsilon)\varphi^2(x) (\lip f_n(x))^2+\left(1+\frac{1}{\varepsilon}\right)f_n^2(x)(\lip\,\varphi(x))^2,\qquad \forall\,x\in\Omega.
\end{equation}
Finally, integrating \eqref{eq:young_ineq} and noticing that $\lip\tilde f_n$ is $0$ outside of $\Omega$ by locality, we have
\begin{equation}
    \int_\X(\lip\tilde f_n)^2\,\d\mm=\int_\Omega(\lip\tilde f_n)^2\,\d\mm\leq (1+\varepsilon)\int_\Omega\varphi^2 (\lip f_n)^2 \,\d\mm+\left(1+\frac{1}{\varepsilon}\right)\int_\Omega f_n^2(\lip\,\varphi)^2\,\d\mm.
\end{equation}
Using the properties of $\{f_n\}$ in \eqref{eq:competitor_Ch}, we conclude that 
\begin{equation}
    {\rm Ch}(Tf)\leq\limi_{n\to\infty}\frac{1}{2}\int_\X(\lip\tilde f_n)^2\,\d\mm\leq (1+\varepsilon){\rm Ch}_{\Omega}(f),   
\end{equation}
for any $\varepsilon>0$, proving the claim \eqref{eq:claim}. This proves that $T$ is a continuous isometry on $\Lip_{bs}(\Omega)$. We can extend the operator $T$ on $W_0^{1,2}(\Omega)$ by density, setting
\begin{equation}
    Tf:=\lim_{n\to\infty} Tf_n, \qquad \text{where }\ \{f_n\}\subset \Lip_{bs}(\Omega)\ \text{ and } \ f_n\xrightarrow[W^{1,2}(\Omega)]{} f.
\end{equation}
This concludes the proof.
\end{proof}

With a slight abuse of notation, in the rest of the paper, we will use the same name for a function $f\in W_0^{1,2}(\Omega)$ and its extension, if there is no confusion. A first byproduct of Proposition \ref{prop:extension_op} is an analogue to \eqref{eq:representation_cheeger} for ${\rm Ch}_\Omega$, indeed
\begin{equation}
    \label{eq:local_cheeger_to_global_cheeger}
    {\rm Ch}_{\Omega}(f) = {\rm Ch}(T f) = \frac{1}{2} \int_{\X} |D T f|^2\,\d \mm,\qquad\forall f \in W^{1,2}_0(\Omega).
\end{equation}
The locality of the minimal weak upper gradient on $\X$ and the definition of $|D f|$ in Appendix \ref{app:local_sobolev_spaces} gives that $|D Tf| = |D f|$ $\mm$-a.e.\ on $\Omega$, which, together with \eqref{eq:local_cheeger_to_global_cheeger}, yields that 
\begin{equation}
\label{eq:representation_cheegeromega}
    {\rm Ch}_{\Omega}(f) = \frac{1}{2} \int_{\Omega} |D f|^2\,\d \mm, \qquad\forall f \in W^{1,2}_0(\Omega).
\end{equation}
From \eqref{eq:local_cheeger_to_global_cheeger} we also deduce that ${\rm Ch}_\Omega$ is convex. Moreover, if $\X$ is infinitesimally Hilbertian, then $W^{1,2}(\Omega)$ is Hilbert, cf. Remark \ref{rmk:W12_Hilbert}, and the energy ${\rm Ch}_{\Omega}$ is quadratic and, in particular, we have
\begin{equation}
\label{eq:representation_cheegeromega_gradients}
    {\rm Ch}_{\Omega}(f) = \frac12\int_\Omega \la\nabla f,\nabla f\ra\,\d\mm, \qquad\forall f \in W^{1,2}_0(\Omega),
\end{equation}
where the equality follows from \eqref{eq:from_carre_to_scprod} of Appendix \ref{app:local_sobolev_spaces}. We refer the reader to Appendix \ref{app:local_sobolev_spaces} for the definition of $\nabla f \in L^2(T\X) \restr{\Omega}$ for $f \in W^{1,2}_0(\Omega)$.

A second consequence of Proposition \ref{prop:extension_op} is the following property, whose proof is based upon the results in \cite{DebinGigliPasqualetto21}. We refer to \cite{DebinGigliPasqualetto21} for the precise definitions and statements.

\begin{proposition}
\label{prop:trace_W120}
Let $(\X,\sfd,\mm)$ be a $\PI$ space and let $\Omega\subset \X$ be an open set of finite perimeter. Then, for every $f \in W^{1,2}_0(\Omega)$, it holds that ${\sf QCR}(f)=0$ ${\rm Per}(\Omega,\cdot)$-a.e., where 
\begin{equation}
    {\sf QCR} \colon W^{1,2}(\X) \to {\rm QC}(\X)
\end{equation}
is the map associating to $f$ is unique quasi-continuous representative. 
\end{proposition}
\begin{proof}
Let $f\in W^{1,2}_0(\Omega)$ and let $\{f_n\}_{n\in\N} \subset \Lip_{bs}(\Omega)$ be a sequence such that $f_n \to f$ in $W^{1,2}(\Omega)$. Applying Proposition \ref{prop:extension_op}, we deduce that $f_n \to f$ in $W^{1,2}(\X)$. Moreover, by construction, $f_n = 0$ everywhere on $\partial \Omega$. We claim that, up to subsequences, $f_n \to {\sf QCR}(f)$, ${\rm Cap}$-a.e..
First of all, by \cite[Thm.\ 1.20]{DebinGigliPasqualetto21}, the map ${\sf QCR}$ is well-defined, linear and continuous, hence $\sfd_{{\rm QC}}(f_n,{\sf QCR}(f)) \to 0$. Note that ${\sf QCR}(f_n)=f_n$ since $f_n$ is Lipschitz. Second of all, applying \cite[Prop.\ 1.17, (ii)]{DebinGigliPasqualetto21}, the canonical embedding
\begin{equation}
    \big({\rm QC}(\X),\sfd_{{\rm QC}}\big)\hookrightarrow \big(L^0({\rm Cap}),\sfd_{{\rm Cap}}\big)
\end{equation}
is continuous, then 
 $\sfd_{{\rm Cap}}(f_n,{\sf QCR}(f)) \to 0$ as $n \to \infty$. Finally, applying \cite[Prop.\ 1.12]{DebinGigliPasqualetto21}, there exists a (not relabeled) subsequence for which $f_n \to {\sf QCR}(f)$ ${\rm Cap}$-a.e., proving the claim. As a consequence, ${\sf QCR}(f) = 0$ ${\rm Cap}$-a.e.\ on $\partial \Omega$. 
Since $(\X,\sfd,\mm)$ is a $\PI$ space, we have that ${\rm Per}(\Omega,\cdot)$ is concentrated on $\partial^e \Omega \subset \partial \Omega$. This fact, together with \eqref{eq:perllcap}, yields that ${\sf QCR}(f) = 0$ ${\rm Per}(\Omega,\cdot)$-a.e..
\end{proof}
\begin{remark}
\label{rmk:representative_set_dir_hf}
The local Sobolev space with zero boundary values remains unchanged if we remove from the open set $\Omega\subset\X$ a closed set with zero capacity. More precisely, let $E\subset\Omega$ be a closed set with ${\rm Cap}(E)=0$, then
\begin{equation}
    W_0^{1,2}(\Omega)=W_0^{1,2}(\Omega\setminus E).
\end{equation}
Analogously, by locality of the differential of a function, 
${\rm Ch}_{\Omega}(f)={\rm Ch}_{\Omega\setminus E}(f)$.
\end{remark}

\subsection{The Dirichlet heat flow}
We are in position to define the Dirichlet heat flow in $\Omega$ as the gradient flow of the Dirichlet energy associated to ${\rm Ch}_\Omega$ in infinitesimally Hilbertian metric measure spaces $(\X,\sfd,\mm)$. Precisely, define $E_\Omega \colon L^2(\Omega,\mm)\to [0,+\infty]$ as
\begin{equation}
\label{eq:energy}
    E_{\Omega}(f):= \begin{cases}
     {\rm Ch}_{\Omega}(f) & \text{if }f \in W^{1,2}_0(\Omega),\\
     +\infty  & \text{if }f \in L^2(\Omega,\mm) \setminus W^{1,2}_0(\Omega).
    \end{cases}
\end{equation}
On the one hand, using the convexity of ${\rm Ch}_\Omega$, it is straightforward to check that $E_{\Omega}$ is convex. On the other hand, $E_\Omega$ is lower semicontinuous with respect to $L^2(\Omega,\mm)$-topology. Indeed, let $f_n \to f$ in $L^2(\Omega,\mm)$ and we may assume that $\limi_n E_{\Omega}(f_n)< \infty$, otherwise there is nothing to prove. Hence for $n$ sufficiently large $f_n \in W^{1,2}_0(\Omega)$ and $\{f_n\}$ is bounded in the $W^{1,2}(\Omega)$-norm. Then, being $(\X,\sfd,\mm)$ infinitesimally Hilbertian, there exists a subsequence (not relabeled) such that $f_n\rightharpoonup f$ in $W_0^{1,2}(\Omega)$. Then, by Mazur's lemma, there exists a sequence $\{s_n\}\subset W_0^{1,2}(\Omega)$ of convex combinations of $f_n$ such that $s_n\to f$ in $W^{1,2}(\Omega)$. Since $\{s_n\}\subset W_0^{1,2}(\Omega)$, by a diagonal argument we can also assume that $\{s_n\}\subset \Lip_{bs}(\Omega)$, implying that $f\in W_0^{1,2}(\Omega)$. We conclude that $E_\Omega$ is $L^2(\Omega,\mm)$-lower semicontinuous, from the analogous property of ${\rm Ch}_\Omega$. Finally, by construction $D(E_{\Omega}) = W^{1,2}_0(\Omega)$ is dense in $L^2(\Omega,\mm)$, since $\Lip_{bs}(\Omega)$ is dense in $L^2(\Omega,\mm)$. Therefore, we can apply Theorem \ref{thm:gradient_flows}, obtaining the existence of the gradient flow of $E_\Omega$. This solves the Dirichlet heat equation on $\Omega$, in the sense explained below.

\begin{definition}[Dirichlet Laplacian]
Let $(\X,\sfd,\mm)$ be an infinitesimally Hilbertian metric measure space. Given $f \in W^{1,2}_0(\Omega)$, we say that $f \in D(\Delta_\Omega)$, if there exists $h \in L^2(\Omega,\mm)$ such that
\begin{equation}
    \label{eq:definition_dirichlet_laplacian}
    \int_{\Omega} h\, f\, \d \mm  = -\int_{\Omega} \la \nabla f, \nabla g \ra\,\d \mm \quad \text{for every }g\in W^{1,2}_0(\Omega).
\end{equation}
The element $h$, uniquely determined by the density of $W^{1,2}_0(\Omega)$ in $L^2(\Omega,\mm)$, will be denoted by $\Delta_\Omega f$.
\end{definition}
\begin{proposition}
\label{prop:diricletlaplacian_vs_subdiffEomega}
Let $(\X,\sfd,\mm)$ be infinitesimally Hilbertian. Define $E_{\Omega} \colon L^2(\mm) \to [0,+\infty]$ as in \eqref{eq:energy}. Given $f \in W^{1,2}_0(\X)$, we have
\begin{equation}
    f \in D(\Delta_\Omega) \Leftrightarrow \partial^- E_{\Omega}(f) \neq \emptyset.
\end{equation} 
In this case, it holds that $\partial^- E_{\Omega}(f) = \{- \Delta_{\Omega} f\}$.
\end{proposition}
We omit the proof of this proposition, being a straightforward modification of \cite[Prop.\ 5.2.4]{GP19} relying on \eqref{eq:from_carre_to_scprod}. This statement allows to build the Dirichlet heat flow as the gradient flow of $E_\Omega$. 

\begin{definition}[Dirichlet heat flow]
\label{def:dirichlet_heat_flow}
Let $(\X,\sfd,\mm)$ be an infinitesimally Hilbertian metric measure space and let $\Omega\subset\X$ be open. For every $f \in L^2(\Omega,\mm)$, the \emph{Dirichlet heat flow} of $f$ is the gradient flow of $E_{\Omega}$ starting from $f$, and it is denoted as $[0,1]\ni t \mapsto h_t^{\Omega} f\in W_0^{1,2}(\Omega)$. By Proposition \ref{prop:diricletlaplacian_vs_subdiffEomega}, the Dirichlet heat flow satisfies the following: 
\begin{equation}
\label{eq:dirichlet_heat_equation}
\begin{cases}
\partial_t h_t^{\Omega} f = \Delta_{\Omega} h_t^{\Omega} f, & \text{for a.e. }t >0,\\
h_t^\Omega f \xrightarrow{t\to 0^+} f, &\text{in }L^2(\Omega,\mm),
\end{cases}
\end{equation}
where the derivative $\partial_th_t^\Omega f$ has to be intended in the $L^2(\Omega,\mm)$-sense. 
\end{definition}
For a different construction of the Dirichlet heat flow using spectral theory, we refer the reader to \cite{ZhangZhu19}.
\begin{definition}[Heat content]
\label{def:heat_content}
Let $(\X,\sfd,\mm)$ be a metric measure space and let  $\Omega\subset \X$ be an open and bounded set. Let $u_t\in W^{1,2}_0(\Omega)$ be the Dirichlet heat flow \eqref{eq:dirichlet_heat_equation} with initial datum $f=\chi_\Omega\in L^2(\Omega,\mm)$, namely $u_t := h_t^\Omega \chi_\Omega$, for every $t>0$. Then, we define the \emph{heat content} as
\begin{equation}
    \label{eq:def_heat_content}
    Q_\Omega(t)=\int_\Omega u_t \,\d \mm, \qquad\forall\, t>0.
\end{equation}
\end{definition}

\begin{remark}
\label{rem:heat_content_under_changeofcaprep}
Note that, according to Remark \ref{rmk:representative_set_dir_hf}, if ${\rm Cap}(E)=0$, for $f\in L^2(\Omega,\mm)= L^2(\Omega\setminus E,\mm)$, we have $h_t^\Omega f=h_t^{\Omega \setminus E} f$ for every $t >0$. The analogous property holds for the heat content.
\end{remark}

\begin{proposition}[Weak maximum principle]
\label{prop:weak_max_principle}
Let $(\X,\sfd,\mm)$ be infinitesimally Hilbertian, let $\Omega \subset \X$ be open and bounded and let $C \ge 0$. Consider $f \in L^2(\Omega,\mm)$ such that $f \le C$ (respectively $f \ge -C$) $\mm$-a.e.. Then, for every $t \ge 0$, $h_t^{\Omega} f \le C$ (respectively $h_t^{\Omega} f \ge -C$) $\mm$-a.e..
\end{proposition}
\begin{proof}
We fix $\epsilon > 0$ and define $\varphi(t):= \sqrt{\epsilon^2+(t-C)^2}-\epsilon$ for $t >C$ and $\varphi(t):= 0$ for $t \le C$. In particular, we have that $\varphi \in C^2(\mathbb{R})$, $\varphi' \le 1$ and $\varphi'' \le \epsilon^{-\frac{1}{2}}$.
Let $F(t):=\int_{\Omega} \varphi(h_t^{\Omega} f)\,\d \mm$. Then the estimate 
\begin{equation}
    |F(t)-F(s)| \le \Lip\,\varphi\,\mm(\Omega)^{\frac{1}{2}}\| h_t^{\Omega} f - h_s^{\Omega} f\|_{L^2(\Omega,\mm)},\qquad\text{for every }t,s>0
\end{equation}
gives that $F \in C^0([0,\infty)) \cap {AC}_{loc}(0,\infty)$. Therefore, $F$ is differentiable for a.e.\ $t>0$ with derivative given by
\begin{equation}
\label{eq:derivative_aux_funct_wmpr}
    F'(t) = \int_{\Omega} \varphi'(h_t^{\Omega} f)\,\Delta_{\Omega} h_t^{\Omega} f\,\d \mm \qquad\text{for a.e. }t >0.
\end{equation}
Note that, by construction, $\varphi'\in C^1(\mathbb{R})$, it has bounded derivative and satisfies $\varphi'(0) = 0$. Thus, by application of Proposition \ref{prop:composition_W120} we deduce that $\varphi' \circ h_t^\Omega f \in W^{1,2}_0(\Omega)$, therefore we can integrate by parts in \eqref{eq:derivative_aux_funct_wmpr} obtaining
\begin{equation}
    F'(t) = -\int_{\Omega} \varphi''(h_t^{\Omega} f)\,|\nabla  h_t^{\Omega} f|^2\,\d \mm \le 0 \qquad\text{ for a.e. }t >0,
\end{equation}
where the last inequality follows by the convexity of $\varphi$. Hence, since $F\geq 0$ for every $t >0$, and $F$ is continuous up to $t=0$, we have
\begin{equation}
    0 \le F(t) \le F(0) = 0,\qquad\forall\,t>0,
\end{equation} 
yielding that $\varphi(h_t^{\Omega}f) = 0$ $\mm$-a.e. on $\Omega$, which, by the very definition of $\varphi$, gives that $h_t^{\Omega} f \le C$ $\mm$-a.e.. For the lower bound, we can argue similarly, by considering as test function $\tilde{\varphi}(x):=\varphi(-x)$.
\end{proof}
\subsection{Kac's principle on metric measure spaces}
\label{sec:kac's principle}
In this section, we derive a proof of the Kac's principle, by means of functional analytic arguments. The main technical tool is Lemma \ref{prop:decay_pde}, which rely on the following auxiliary operator.

\begin{definition}[Restricted Laplacian]
\label{def:restricted_laplacian}
Let $(\X,\sfd,\mm)$ be an infinitesimally Hilbertian metric measure space. Given $f\in W^{1,2}(\Omega)$, we say that $f\in D\big(\Delta_{\restr{\Omega}}\big)$ if there exists a function $h\in L^2(\Omega,\mm)$ such that
\begin{equation}
    \int_\Omega hg\,\d\mm=-\int_\Omega \la\nabla f,\nabla g \ra \,\d\mm,\qquad \text{for every }g\in \Lip_{bs}(\Omega).
\end{equation}
Note that the function $h\in L^2(\Omega,\mm)$ is unique, as a consequence of the density of $\Lip_{bs}(\Omega)$ in $L^2(\Omega,\mm)$. In this case, we set $h:=\Delta_{\restr{\Omega}}f$.
\end{definition}
It is straightforward to check, by the very definition of $W^{1,2}_0(\Omega)$, that the function $g$ in Definition \ref{def:restricted_laplacian} can be taken in $W^{1,2}_0(\Omega)$.
Moreover, by definition,
\begin{equation}
\label{eq:rel_laplacians}
    D\big(\Delta_{\restr{\Omega}}\big)\cap W_0^{1,2}(\Omega)=D(\Delta_\Omega)\qquad\text{and}\qquad D(\Delta)_{\restr{\Omega}}\subset D\big(\Delta_{\restr{\Omega}}\big),
\end{equation}
and, whenever they exist, the Laplacians coincide. Note that the inclusion above may be strict in general. Let us recall the Leibniz rule for the restricted Laplacian: its proof is a straightforward modification of the argument in \cite[Thm.\ 4.29]{Gigli12}. 
\begin{proposition}
\label{prop:chain_rule_restricted_laplacian}
Let $\Omega \subset \X$. Let $u, v \in L^\infty(\Omega,\mm) \cap D(\Delta \restr{\Omega})$. Moreover, assume that $|\nabla u| \in L^\infty(\Omega,\mm)$. Then $u v \in D(\Delta\restr{\Omega})$ and
\begin{equation}
\label{eq:chain_restricted_laplacian}
    \Delta\restr{\Omega}(u v) = v\Delta\restr{\Omega}u  +u\Delta\restr{\Omega}v +2 \la \nabla u, \nabla v \ra \quad \mm\text{-a.e. in }\Omega.
\end{equation}
\end{proposition}
\begin{proof}
Recall that, by Proposition \ref{prop:product_sobolev}, since $u,v\in L^\infty(\Omega,\mm)\cap W^{1,2}(\Omega)$, then $u v\in W^{1,2}(\Omega)$ and
\begin{equation}
\label{eq:chain_localgradient}
    \nabla (u v)=v\nabla u+ u\nabla v,\qquad\text{as elements of } L^2(T\X)\restr{\Omega}.
\end{equation}
Let $g \in \Lip_{bs}(\Omega)$ and compute
\begin{equation}
\begin{aligned}
    \int_\Omega \la \nabla (u v),\nabla g \ra\,\d \mm & \stackrel{\eqref{eq:chain_localgradient}}{=} \int_\Omega \la v\nabla u +u\nabla v,\nabla g \ra\,\d \mm \stackrel{\eqref{eq:chain_localgradient}}{=} \int_\Omega \la \nabla u, \nabla (v g)\ra+\la \nabla v, \nabla (u g)\ra + 2 g\la \nabla u,\nabla v \ra \,\d \mm\\
    & = - \int_\Omega \left( v\Delta\restr{\Omega} u+ u\Delta\restr{\Omega}v + 2 \la \nabla u, \nabla v \ra \right)g\,\d \mm,
\end{aligned}
\end{equation}
where the last equality follows from the fact that $u,v \in D(\Delta\restr{\Omega})$ and $ug,vg \in W^{1,2}_0(\Omega)$. The last term is in $L^2(\Omega,\mm)$ since $|\nabla u| \in L^\infty(\Omega,\mm)$, thus $uv \in D(\Delta \restr{\Omega})$ and \eqref{eq:chain_restricted_laplacian} holds.
\end{proof}

\begin{lemma}
\label{prop:decay_pde}
Let $(\X,\sfd,\mm)$ be an infinitesimally Hilbertian metric measure space and let $\Omega\subset\X$ be open and bounded.
Let $(t \mapsto v_t) \in C([0,1],L^2(\Omega,\mm)) \cap AC_{loc}((0,1),L^2(\Omega,\mm))$ such that $v_t$ is non-negative and $v_t \in D(\Delta\restr{\Omega})$ for a.e.\ $t$,
\begin{equation}
\label{eq:decay_pde}
\begin{cases}
\partial_t v_t = \Delta \restr{\Omega} v_t\\
v_0 = 0.
\end{cases}
\end{equation}
Consider a compact set $K \Subset \Omega$; then, \begin{equation}
    \| v_t \|_{L^1(K,\mm)} =o(t),\qquad\text{as }t\to 0^+.
\end{equation}
\end{lemma}
\begin{proof}
We consider a non-negative $\varphi$ such that $\varphi =1$ on $K$, $\supp{ \varphi} \Subset \Omega$ and $\varphi \in D(\Delta) \subset D(\Delta\restr{\Omega})$. It follows
\begin{equation}
    0 \le  \int_K v_t\,\d \mm \le \int_{\Omega} v_t\varphi\,\d \mm  =:F(t).
\end{equation}
Consider $\eta \in \Lip_{bs}(\Omega)\subset L^\infty(\X,\mm) \cap \Lip(\X)$ such that $\eta = 1$ on $\supp{\varphi}$; then, since $ v_t \in W^{1,2}(\Omega)$, by Theorem \ref{thm:equivalence_local_sobolev}, $\eta v_t \in W^{1,2}(\X)$ and
\begin{equation}
\label{eq:chain_rule_uteta}
    \nabla (v_t \eta) = \eta\nabla v_t + v_t\nabla \eta \qquad \text{as elements of }L^2(T\X)\restr{\Omega}.
\end{equation}
Since $L^2(\Omega,\mm) \in f \mapsto \int_\Omega f\varphi\,\d \mm \in \mathbb{R}$ is Lipschitz and $v \in AC_{loc}((0,1),L^2(\Omega,\mm))$, then $F \in AC_{loc}(0,1)$ and 
\begin{equation}
    \begin{aligned}
        F'(t)& = \int_\Omega \varphi\partial_t v_t\,\d\mm = \int_{\Omega}\varphi \Delta \restr{\Omega} v_t\,\d \mm = -\int_{\Omega} \la \nabla v_t,\nabla \varphi \ra\,\d \mm = -\int_{\Omega} \la \nabla v_t,\nabla \varphi \ra\eta\,\d \mm \\
        & =\int_\Omega v_t\la \nabla \eta, \nabla \varphi \ra\,\d \mm -\int_\Omega \la \nabla (v_t \eta),\nabla \varphi \ra\,\d \mm  = \int_{\Omega} v_t\left( \la \nabla \eta,\nabla \varphi \ra + \eta\Delta\restr{\Omega} \varphi \right)\,\d \mm \\
        & = \int_\Omega v_t\Delta\restr{\Omega} \varphi\,\d \mm\\
    \end{aligned}
\end{equation}
for a.e.\ $t$. We notice that the function $\left( t \mapsto G(t):=\int_\Omega v_t\Delta\restr{\Omega} \varphi\,\d \mm \right) \in C([0,1])$, therefore extending $F$ by $0$ at $t=0$, we have $F \in C^1([0,1])$. Thus,
\begin{equation}
    F(t) = F(0)+F'(0)t+o(t) = o(t)
\end{equation}
yielding $\| v_t \|_{L^1(K,\mm)} = o(t)$. This concludes the proof.
\end{proof}
\begin{remark}
Arguing as in the previous proof, in a smooth Riemannian manifold we would get higher-order estimates, using more regular test functions; indeed, for every $n\in\N$, $\| v_t \|_{L^1(K,\mm)} = o(t^n)$, by considering a test function $\varphi \in C^{\infty}_c(\Omega)$ such that $\varphi =1$ on $K$ and reasoning as in the proof of Proposition \ref{prop:decay_pde}.
\end{remark}
In order to deduce our version of the Kac's principle, we combine Lemma \ref{prop:decay_pde} with a monotonicity result for the Dirichlet heat flow with respect to domain inclusion. Before proving the latter, we highlight the relations between the Laplacian defined on the whole space $\Delta$, the Dirichlet Laplacian $\Omega$ and the restricted Laplacian.
We introduce the following notation.  Let $\Omega_1\subset\Omega_2\subset\X$ be open sets, then for $i=1,2$ we denote by $ \mm_i:=\mm \restr{\Omega_i}$ and by $\pi_{\mm_2,\mm_1} \colon L^0(\mm_2) \to L^0(\mm_1)$  the restriction map defined as $\pi_{\mm_2,\mm_1}([f]_{\mm_2}) = [f]_{\mm_1}$, where $[f]_{\mm_i}$ is the equivalence class of $f$ as an element of $L^0(\mm_i)$.
\begin{lemma}
\label{lemma:restriction_restrlap}
Let $(\X,\sfd,\mm)$ be infinitesimally Hilbertian and let $\Omega_1 \subset \Omega_2 \subset \X$ be open sets.
Let $g \in D(\Delta \restr{\Omega_2})$. Then $\pi_{\mm_2,\mm_1} g \in D(\Delta \restr{\Omega_1})$ and $\Delta \restr{\Omega_1} (\pi_{\mm_2,\mm_1} g) = \pi_{\mm_2,\mm_1}\big( \Delta \restr{\Omega_2} g\big)$.
\end{lemma}

\begin{proof}
We have a natural isometric embedding $i\colon W^{1,2}_0(\Omega_1) \hookrightarrow W^{1,2}_0(\Omega_2)$ by zero extension. Set $\tilde{g}:= \pi_{\mm_2,\mm_1} g\in W^{1,2}(\Omega_1)$ by construction. Then, for every $\varphi \in W^{1,2}_0(\Omega_1)$, we have
\begin{equation*}
    \int_{\Omega_1} \la \nabla \tilde{g}, \nabla \varphi \ra \,\d \mm= \int_{\Omega_2} \la \nabla g, \nabla i(\varphi) \ra\,\d \mm =- \int_{\Omega_2} i(\varphi) \Delta\restr{\Omega_2} g \,\d \mm =- \int_{\Omega_1} \pi_{\mm_2,\mm_1} \big(\Delta\restr{\Omega_2} g\big) \,\varphi\,\d \mm,
\end{equation*}
observing that $i(\varphi)=|\nabla i(\varphi)|=0$, $\mm$-a.e.\ on $\Omega_2\setminus\Omega_1$. Hence, by arbitrariness of $\varphi\in W_0^{1,2}(\Omega_1)$, we conclude that $\tilde{g} \in D(\Delta\restr{\Omega_1})$ and $\Delta\restr{\Omega_1} \tilde{g} = \pi_{\mm_2,\mm_1}\big(\Delta\restr{\Omega_2} g\big)$.
\end{proof}
\begin{proposition}[Domain monotonicity]
\label{prop:domain_monotonicity}
Let $(\X,\sfd,\mm)$ be infinitesimally Hilbertian, let $\Omega_1\subset\Omega_2\subset\X$ be open sets and assume $\Omega_1$ is bounded. Let also $0 \le f \in L^2(\Omega_2)$. Then, for every $t \ge 0$, 
\begin{equation}
    h_t^{\Omega_1} f(x)\le h_t^{\Omega_2} f(x),\qquad\text{for}\ \mm\text{-a.e.}\ x\in\Omega_1.
\end{equation}
\end{proposition}
\begin{proof}
We first prove the proposition for $f \in L^2(\Omega_2)$.
We fix $\epsilon > 0$ and define $\varphi(t):= \sqrt{\epsilon^2+t^2}-\epsilon$ for $t >0$ and $\varphi(t):= 0$ for $t \le 0$. In particular, we have that $\varphi' \le 1$ and $\varphi'' \le \epsilon^{-\frac{1}{2}}$. We define $G(t):= \int_{\Omega_1} \varphi(h_t^{\Omega_1} f- h_t^{\Omega_2} f)\,\d \mm$. We estimate, using that $\varphi$ is $1$-Lipschitz
\begin{equation}
\begin{aligned}
    |G(t)-G(s)| &= \left| \int_{\Omega_1} \varphi(h_t^{\Omega_1} f- h_t^{\Omega_2} f)-\varphi(h_s^{\Omega_1} f- h_s^{\Omega_2} f)\,\d \mm  \right|\\
    &\le \mm(\Omega_1)^{\frac{1}{2}} (\| h_t^{\Omega_1} f-h_s^{\Omega_1} f\|_{L^2(\Omega_1)}+ \| h_t^{\Omega_2} f-h_s^{\Omega_2} f\|_{L^2(\Omega_1)}) \\
    &\le \mm(\Omega_1)^{\frac{1}{2}} (\| h_t^{\Omega_1} f-h_s^{\Omega_1} f\|_{L^2(\Omega_1)}+ \| h_t^{\Omega_2} f-h_s^{\Omega_2} f\|_{L^2(\Omega_2)}) \\
\end{aligned}
\end{equation}
which, using the fact that $(t \mapsto h_t^{\Omega_i} f) \in AC_{loc}((0,+\infty),L^2(\Omega_i))$ for $i = 1,2$, grants that $G \in AC_{loc}((0,+\infty))$, hence differentiable a.e.\ with derivative given by
\begin{equation}
\label{eq:derivative_G_domain_monotonicity}
G'(t) = \int_{\Omega_1} \varphi'(h_t^{\Omega_1} f- h_t^{\Omega_2} f)\,\left( \Delta_{\Omega_1} h_t^{\Omega_1} f - \Delta_{\Omega_2} h_t^{\Omega_2} f  \right)\,\d \mm \quad \text{for a.e. }t.    
\end{equation}
By Lemma \ref{lemma:restriction_restrlap} and \eqref{eq:rel_laplacians}, we have, with a slight abuse of notation, that $h_t^{\Omega_1} f -  h_t^{\Omega_2} f \in D(\Delta \restr{\Omega_1})$ and
\begin{equation}
\label{eq:domain_monotonicity_lapldomains}
     \Delta_{\Omega_1} h_t^{\Omega_1} f - \Delta_{\Omega_2} h_t^{\Omega_2} f =    \Delta\restr{\Omega_1} (h_t^{\Omega_1} f -  h_t^{\Omega_2} f).
\end{equation} 
We claim that $\varphi'(h_t^{\Omega_1} f- h_t^{\Omega_2} f) \in W^{1,2}_0(\Omega_1)$, for every $t>0$. For the sake of notation, let $u = h_t^{\Omega_1}f$ and $v = h_t^{\Omega_2}f$. By the weak maximum principle (cf. Proposition \ref{prop:weak_max_principle}), since $f\geq 0$ $\mm$-a.e. on $\Omega_2$, we deduce
\begin{equation}
    u \ge 0,\quad\mm\text{-a.e.}\ \text{on}\ \Omega_1\qquad\text{and}\qquad v \ge 0,\quad\mm\text{-a.e.}\ \text{on}\ \Omega_2.
\end{equation} 
By definition of local Sobolev space, there exist sequences $\{u_n\}\subset \Lip_{bs}(\Omega_1) $, $\{v_n\}\subset \Lip_{bs}(\Omega_2)$ such that $u_n \to u$ in $W^{1,2}(\Omega_1)$ and $v_n \to v$ in $W^{1,2}(\Omega_2)$. Moreover, we can choose $v_n \ge 0$ (cf. Remark \ref{rmk:improvement_prop_composition}).
Set $\psi := \varphi'$ and define $\psi_n:= \psi\circ(u_n-v_n)$, for every $n\in\N$. On the one hand, since $v_n \ge 0$ and $\psi(t) = 0$ for $t \le 0$, we obtain 
\begin{equation}
\label{eq:support_psi_n}
\psi_n(x) = \psi(-v_n(x)) = 0,\qquad \forall\,x\in(\supp{u_n})^c.
\end{equation}
On the other hand, $\psi_n \in \Lip(\Omega_1)$, indeed, for any $x,y\in\Omega_1$, we have 
\begin{equation}
\label{eq:psi_n_lipschitz}
|\psi_n(x)-\psi_n(y)| \le \Lip(\psi)\, (|u_n(x)-u_n(y)|+|v_n(x)-v_n(y)|) \le \Lip(\psi)\,(\Lip(u_n)+\Lip(v_n))\,\sfd(x,y).
\end{equation}
Combining \eqref{eq:support_psi_n} and \eqref{eq:psi_n_lipschitz}, we prove that $\psi_n \in \Lip_{bs}(\Omega_1)$, for every $n\in\N$.
Moreover, arguing as in Proposition \ref{prop:composition_W120}, we can prove that $\psi_n \to \psi\circ(u-v)$ in $W^{1,2}(\Omega_1)$. This finally proves that $\psi\circ(u-v) \in W^{1,2}_0(\Omega_1)$, as claimed. At this stage, using \eqref{eq:domain_monotonicity_lapldomains}, we can integrate by parts in \eqref{eq:derivative_G_domain_monotonicity}, obtaining
\begin{equation}
G'(t) = -\int_{\Omega_1} \varphi''(h_t^{\Omega_1} f- h_t^{\Omega_2} f)\,|\nabla(h_t^{\Omega_1} f- h_t^{\Omega_2} f)|^2\,\d \mm \le 0.
\end{equation}
Therefore, $0 \le G(t) \le G(0) = 0$ for every $t \ge 0$, hence $G(t) \equiv 0$ yielding that $h_t^{\Omega_1} f \le h_t^{\Omega_2} f$ $\mm$-a.e.\ on $\Omega_1$. This concludes the proof.
\end{proof}

\begin{remark}
\label{remark:from_W120_to_W12}
As a consequence of \cite{AmbrosioGigliSavare11-3} and the reflexivity of $W^{1,2}(\X)$, we have that $W^{1,2}_0(\X) = W^{1,2}(\X)$. Hence, the Dirichlet Laplacian on $\X$ and the Laplacian in the sense of Definition \ref{def:laplacian} coincide, i.e.  $D(\Delta_\X) = D(\Delta)$ and $\Delta_\X = \Delta$. Therefore, the domain monotonicity applies with $\Omega_2=\X$. 
\end{remark}
\begin{corollary}[Kac's principle]
\label{coro:kacs_principle}
Let $(\X,\sfd,\mm)$ be an infinitesimally Hilbertian metric measure space and let $\Omega\subset\X$ be open and bounded.
Let $0 \le f \in L^\infty(\mm)$ and $ t \mapsto h_t f,h_t^{\Omega} f$ being respectively the heat flow of $f$ and the Dirichlet heat flow of $f$ associated to $\Omega$. Consider a compact set $K \Subset \Omega$; then
\begin{equation}
\label{eq:first_order_kacs_principle}
\|h_t f-h_t^{\Omega} f \|_{L^1(K,\mm)} = o(t),\qquad\text{as }t\to 0^+.
\end{equation}
\end{corollary}
\begin{proof}
Set $v_t:=h_t f-h_t^\Omega f$. Then, by the regularity of $h_t f,h_t^\Omega f$, we get $v \in  C([0,1],L^2(\Omega,\mm)) \cap AC_{loc}((0,1),L^2(\Omega,\mm))$. Moreover, using \eqref{eq:rel_laplacians}, we see that $v_t \in D(\Delta\restr{\Omega})$ and \eqref{eq:decay_pde} is satisfied. Finally, applying Proposition \ref{prop:domain_monotonicity} (cf.\ Remark \ref{remark:from_W120_to_W12}) we get that, for every $t$, $v_t \ge 0$ $\mm$-a.e.. Thus, we are in position to apply Proposition \ref{prop:decay_pde}, which shows the validity of \eqref{eq:first_order_kacs_principle}.
\end{proof}
\begin{remark}
Note that it would be possible to deduce a stronger Kac's principle for $\RCD$ spaces, with $L^\infty$-norm in place of $L^1$-norm and with exponential error, by means of stochastic calculus and adapting the proofs of \cite{H95}. We refrain to pursue this strategy here in order to avoid using stochastic analysis tools; indeed, our proof is purely Eulerian and do not require any lower curvature bound. 
\end{remark}

\section{One-dimensional localization of the \texorpdfstring{$\CD$}{CD} condition}
\label{sec:1D-localization}

In this section we summarize some results obtained in \cite{Cav14,CavMon15,CavMil16,CM20}, which show that the $\CD$ condition can be properly localized in one-dimensional sets. In particular, we follow the presentation of \cite{CM20}, where the one-dimensional localization is obtained without assuming $\mm(\X)=1$. \\
For the reminder of the section, we assume the metric measure space $(\X,\sfd,\mm)$ to be such that $(\X,\sfd)$ is geodesic and proper and $\mm$ is $\sigma$-finite with $\supp{\mm}=\X$. Given a $1$-Lipschitz function $u:\X \to \mathbb R$, we associate to it the $\d$-cyclically monotone set
\begin{equation}
    \Gamma_u:= \{(x,y)\in \X\times \X : u(x)-u(y)= \sfd(x,y)\}
\end{equation}
and its transpose $\Gamma_u^{-1}:=\{(x,y)\in \X \times \X:(y,x)\in \Gamma_u\}$. Consequently, we introduce the transport relation $R_u$ and the transport set $T_u$ as 
\begin{equation}
    R_u:= \Gamma_u \cup \Gamma_u^{-1} \quad \text{and} \quad T_u := \pi_1 \big(R_u \setminus \{(x,y)\in \X\times \X:x=y\}\big),
\end{equation}
where $\pi_1$ denotes the projection on the first factor. We denote by $\Gamma_u(x)$ the section of $\Gamma_u$ through $x$ in the first coordinate. Then, we define the set of forward and backward branching point as 
\begin{equation}
    A^+ :=\{x \in T_u : \exists y,z \in \Gamma_u(x), (y,z)\notin R_u\},\quad A^- :=\{x \in T_u : \exists y,z \in \Gamma_u^{-1}(x), (y,z)\notin R_u\}.
\end{equation}
Consider respectively the non-branched transport set and the non-branched transport relation 
\begin{equation}
    T_u^{nb}:= T_u \setminus (A^+ \cup A^-), \qquad R_u^{nb}:= R_u \cap ( T_u^{nb} \times  T_u^{nb}).
\end{equation}
As shown in \cite{Cav14}, $R_u^{nb}$ is an equivalence relation on $T_u^{nb}$ and for every $x\in T_u^{nb}$, $R_u(x):= \{y\in \X : (x,y)\in R_u\}$ is isometric to a closed interval of $\R$. In particular, from the non-branched transport relation we obtain a partition of the non-branched transport set $T_u^{nb}$ into a disjoint family $\{X_\alpha\}_{\alpha \in Q}$ of sets, where $Q$ is a set of indices. Moreover, $\bar X_\alpha$ is isometric to a closed interval of $\R$, for any $\alpha\in Q$.

\begin{remark}
Let us introduce two families of distinguished points of the
transport set: the set of initial and final points. These are defined respectively as follows:
\begin{align}
a = \{x \in T_u\mid \nexists y\in T_u, y \neq x, \, (y,x)\in \Gamma_u\}, \quad b = \{x \in T_u\mid \nexists y\in T_u, y \neq x, \, (x,y)\in \Gamma_u\}.
\end{align}
Note that the sets $a$ and $b$ may or may not be a subset of $T_u^{nb}$, $A^+$ or $A^-$. We will denote by $a(X_\alpha)$, $b(X_\alpha)$ the initial and final point, respectively, of the transport set $X_\alpha$, whenever they exist.
\end{remark}

Consider now the quotient map $\mathfrak{Q}\colon T_u^{nb}\to Q$ associated to the partition, that is 
\begin{equation}
    \mathfrak{Q}(x)=\alpha \iff x\in X_\alpha.
\end{equation}
With this map we can endow $Q$ with the quotient $\sigma$-algebra, that is the finest $\sigma$-algebra on $Q$ for which $\mathfrak{Q}$ is measurable. At this point we can identify a disintegration of the measure $\mm \restr{T_u^{nb}}$ associated to the partition $\{X_\alpha\}_{\alpha \in Q}$. Denote by $\mathcal M^+(\X)$ the set of non-negative Radon measure on $\X$.

\begin{theorem}[{\cite[Thm.\ 3.4]{CM20}}]\label{thm:disintegration}
Let $(\X,\sfd,\mm)$ be a metric measure space and assume $\supp{\mm}=\X$. Let $u\colon\X\rightarrow\R$ be 1-Lipschitz. Let $\sfq \in \mathscr{P}(Q)$ such that $\mathfrak Q_\#( \mm \restr{T_u^{nb}}) \ll \sfq$. Then the measure $\mm$ restricted to the non-branched transport set $T_u^{nb}$ admits the following disintegration 
\begin{equation}\label{eq:disintegration}
    \mm\restr{T_u^{nb}} = \int_Q  \mm_\alpha\,\d\sfq( \alpha),
\end{equation}
and the map $Q \ni \alpha \mapsto \mm_\alpha \in \mathcal M^+(\X)$ satisfies the following
\begin{enumerate}
    \item for every $\mm$ measurable set $B$, the map $\alpha \mapsto \mm_\alpha(B)$ is $\sfq$-measurable,
    \item for $\sfq$-almost every $\alpha \in Q$, the measure $\mm_\alpha$ is concentrated in $\mathfrak{Q}^{-1}(\alpha)$,
    \item for every $\mm$ measurable set $B$ and every $\sfq$-measurable set $C$, it holds that
    \begin{equation}
        \mm (B \cap \mathfrak{Q}^{-1}(C)) = \int_C \mm_\alpha(B) \,\d\sfq( \alpha),
    \end{equation}
    \item For every compact set $K \subset \X$ there exists a constant $C_K >0$ such that $\mm_\alpha(K)\leq C_K$ for $\sfq$-almost every $\alpha\in Q$.
\end{enumerate}
Moreover, for every fixed probability measure $\sfq$ such that $\mathfrak Q_\#( \mm \restr{T_u^{nb}}) \ll \sfq$, the disintegration \eqref{eq:disintegration} is $\sfq$-essentially unique, meaning that, if any other map $Q \ni \alpha \mapsto \tilde\mm_\alpha \in \mathcal M^+(\X)$ satisfies (1)-(4), then $\mm_\alpha=\tilde \mm_\alpha$ for $\sfq$-a.e.\ $\alpha\in Q$.
\end{theorem}

The last statement is quite general and does not require any particular rigidity assumption on the metric measure space $(\X,\sfd,\mm)$. However, it becomes particularly interesting when applied to essentially non-branching $\CD(K,N)$ spaces. We refer the reader to \cite{CM20} for the precise definitions. Note that, In our setting, as ${\sf RCD}(K,N)$ spaces are non-branching (see Theorem.\ \ref{thm:nonbranch_rcd}), and so essentially non-branching, the following results apply.

\begin{proposition}[{\cite[Lem.\ 3.5]{CM20}}]
\label{prop:measure_non_branched_set}
Let $(\X,\sfd,\mm)$ be an essentially non-branching space satisfying the $\CD(K,N)$ condition for some $K\in \R$ and $N \in (1,\infty)$. Then, for every 1-Lipschitz function $u: \X\to \R$, it holds that $\mm(T_u \setminus T_u^{nb})=0$.
\end{proposition}

The disintegration presented in Theorem \ref{thm:disintegration} then allows to localize the $\CD(K,N)$ condition to the one-dimensional elements of the partition $\{X_\alpha\}_{\alpha \in Q}$. 

\begin{theorem}[{\cite[Thm.\ 3.6]{CM20}}]
\label{thm:1d_localization}
Let $(\X,\sfd,\mm)$ be an essentially non-branching space satisfying the $\CD(K,N)$ condition for some $K\in \R$ and $N \in (1,\infty)$. For any 1-Lipschitz function $u:\X\to \R$ and fixed $\sfq \in \mathscr{P}(Q)$ such that $\sfq \ll \mathfrak Q_\#( \mm \restr{T_u^{nb}}) \ll \sfq$, there exists a $\sfq$-essentially unique disintegration 
\begin{equation}
\label{eq:disintegration_mm}
    \mm \restr{T_u} = \int_Q  \mm_\alpha\,\d\sfq( \alpha),
\end{equation}
provided by Theorem \ref{thm:disintegration}. Moreover, for $\sfq$-almost every $\alpha\in Q$, $\mm_\alpha$ is a Radon measure with $\mm_\alpha= h_\alpha \cdot \mathscr H^1 \restr{X_\alpha} \ll \mathscr H^1 \restr{X_\alpha}$ and $(\bar X_\alpha,|\cdot|,\mm_\alpha)$ verifies the $\CD(K,N)$ condition.
\end{theorem}
\begin{remark}
\label{rem:mutual_abs_cont}
We point out that, in the original statement of \cite[Thm.\ 3.6]{CM20}, the measure $\sfq\in\mathscr{P}(Q)$ is assumed to satisfy the property $\mathfrak Q_\#( \mm \restr{T_u^{nb}}) \ll \sfq$. However, in the proof of \cite[Lem.\ 3.3]{CM20}, the authors actually build a measure $\tilde\mm\in\mathscr{P}(\X)$ such that $\tilde \mm \ll \mm \restr{T_u^{nb}} \ll \tilde \mm$. Thus, we deduce that $\sfq:= \mathfrak Q_\# \tilde \mm \in \mathscr{P}(Q)$ satisfies $\sfq \ll \mathfrak Q_\#( \mm \restr{T_u^{nb}}) \ll \sfq$.
\end{remark}
\subsection{The Laplacian of the distance function}
We report here one of the main results contained in \cite{CM20}, namely a representation formula for the Laplacian of the signed distance function from the boundary of an open set. Firstly, we recall the notion of Radon functional.

\begin{definition}[Radon functional]
Let $(\X,\sfd)$ be a locally compact metric space and let $\Omega\subset\X$ be an open set. Then, a \emph{Radon functional} over $\Omega$ is a linear functional $T\colon\Lip_{bs}(\Omega)\rightarrow\R$ such that, for every compact subset $K\Subset\Omega$, there exists a constant $C_K\geq 0$ so that 
\begin{equation}
    |T(f)|\leq C_K \max_{K}|f|,\qquad\forall\,f\in\Lip_{bs}(\Omega)\quad\text{with}\quad\supp{f}\subset K.
\end{equation}
\end{definition}
One can consider the following two operations involving Radon functionals:
\begin{itemize}
    \item[i)] \emph{Multiplication}. Given a Radon functional $T$ over an open set $\Omega\subset\X$, we define, for $v \in {\rm Lip}(\Omega)$, a Radon functional $ v T \colon {\rm Lip}_{bs}(\Omega) \to \mathbb{R}$ as 
    \begin{equation}
    \label{eq:multiplication_radonfunctional}
        [v T](f) := T(vf)\qquad\text{for every }f \in {\rm Lip}_{bs}(\Omega);
    \end{equation}
    \item[ii)] \emph{Pushforward}. Consider a locally compact metric space $(\Y,\sfd_\Y)$. Let $T$ be a Radon functional over an open set $\Omega\subset\X$ and let $\Phi \colon \Omega \to \Y$ be Lipschitz such that $\Phi(\Omega)\subset\Y$ is open. Then, we define the following Radon functional over $\Phi(\Omega)\subset \Y$
    \begin{equation}
    \label{eq:pushforward_radonfunctional}
        [\Phi_{\#} T](f):= T(f \circ \Phi) \qquad \text{for every }f \in {\rm Lip}_{bs}(\Phi(\Omega)).
    \end{equation}
\end{itemize}
Secondly, we introduce the definition of distributional Laplacian, following \cite{CM20} (see also \cite[Def.\ 4.4]{Gigli12}).
\begin{definition}[Distributional Laplacian]
Let $(\X,\sfd,\mm)$ be infinitesimally Hilbertian. Let $\Omega\subset\X$ be open and let $u\in \Lip(\X)$. We say that $u\in D(\pmb\Delta,\Omega)$ if there exists a (unique) Radon functional $T$ over $\Omega$ such that
\begin{equation}
\label{eq:distributional_laplacian}
    T(f)=-\int_\X \la\nabla f,\nabla u\ra\,\d\mm\qquad\forall\,f\in \Lip_{bs}(\Omega).
\end{equation}
In this case, we set $\pmb\Delta u\restr{\Omega}:= T$.
\end{definition}

\begin{remark}
Let $u\in D(\pmb{\Delta},\Omega)\cap W^{1,2}(\Omega)$ and consider its distributional Laplacian $\pmb\Delta u\restr{\Omega}$ as defined above. If there exists $h\in L^2(\Omega,\mm)$ such that
\begin{equation}
    \pmb\Delta u\restr{\Omega}(f)=\int_{\Omega} h f\, \d \mm, \qquad\forall\,f\in \Lip_{bs}(\Omega),
\end{equation} 
we have that $u \in D(\Delta\restr{\Omega})$ and $\Delta\restr{\Omega} u =h$, where $\Delta\restr{\Omega}$ is the restricted Laplacian as in Definition \ref{def:restricted_laplacian}.
\end{remark}

Let $\Omega\subset \X$ be open and consider the signed distance function from $\partial\Omega$, defined as in \eqref{eq:signed_distance_fun}, namely
\begin{equation}
    \delta\colon \X\rightarrow\R;\qquad \delta(x)=\sfd(x,\partial \Omega) \chi_{\Omega}(x)-\sfd(x,\partial\Omega) \chi_{\X \setminus {\Omega}}(x).
\end{equation}
Since $\delta$ is $1$-Lipschitz, we can apply Theorem \ref{thm:1d_localization} and obtain a disintegration of $\mm$, associated to $\delta$. Note that, by construction, we have
\begin{equation}
\label{eq:init_final_pts_delta}
    \delta(a(X_\alpha))\ge 0\qquad\text{and}\qquad\delta(b(X_\alpha))\le 0,
\end{equation}
for any $\alpha\in Q$, cf. \ \cite[Rmk.\ 4.12]{CM20}. Moreover, we have the following representation formula for the Laplacian of $\delta$, in the setting of $\RCD$ spaces.
\begin{proposition}[{\cite[Cor.\ 4.16]{CM20}}]
\label{prop:repr_laplacian_delta}
Let $(\X,\sfd,\mm)$ be an $\RCD(K,N)$ space for $K \in \mathbb{R}$ and $N \in (1,\infty)$. Consider $\delta$ as above and the associated disintegration
\begin{equation}
\label{eq:restriction_mm}
    \mm \restr{\X \setminus \partial\Omega} = \int_Q h_\alpha\,\d \mathscr{H}^1\restr{X_\alpha}\,\d \sfq (\alpha).
\end{equation}
Then $\delta\in D(\pmb\Delta, \X \setminus \partial \Omega)$ and\footnote{With a slight abuse of notation, in what follows, we always omit the restriction to $\X\setminus\partial\Omega$ and write simply $\pmb\Delta\delta$.}
\begin{equation}
\label{eq:repr_laplacian_delta}
\pmb\Delta \delta = -(\log h_\alpha)' \mm\restr{\X \setminus \partial \Omega} - \int_Q h_\alpha [\delta_{a(X_\alpha) \cap \Omega}- \delta_{b(X_\alpha) \cap (\X \setminus \Omega)}]\,\d \sfq (\alpha).
\end{equation}
We denote $[\pmb{\Delta} \delta]^{reg}:=-(\log h_\alpha)'$ $\mm$-a.e..
Moreover, the following comparison holds
\begin{equation}
\label{eq:bound_low_upp_laplacian_distance}
    \begin{aligned}
    & \pmb\Delta \delta \le (N-1) \frac{s'_{\frac{K}{N-1}}(\sfd_{b(X_\alpha)})}{s_{\frac{K}{N-1}}(\sfd_{b(X_\alpha)})} \mm \restr{\X \setminus \partial \Omega}+ \int_Q h_\alpha \delta_{b(X_\alpha) \cap (\X \setminus \Omega)}\,\d \sfq(\alpha),\\
    & \pmb\Delta \delta \ge -(N-1) \frac{s'_{\frac{K}{N-1}}(\sfd_{a(X_\alpha)})}{s_{\frac{K}{N-1}}(\sfd_{a(X_\alpha)})} \mm \restr{\X \setminus \partial \Omega}- \int_Q h_\alpha \delta_{a(X_\alpha) \cap \Omega}\,\d \sfq(\alpha),
    \end{aligned}
\end{equation}
where $\sfd_p(\cdot):=\sfd(p,\cdot)$ and, for $\kappa \in \R$, $s_\kappa\colon [0,J_k]\rightarrow\R$ is defined as
\begin{equation}
\label{eq:model_fun}
    s_\kappa(\theta)=
    \begin{cases}
        \frac{1}{\sqrt\kappa}\sin(\theta\sqrt\kappa), &\text{if }\kappa>0,\\
        \theta, &\text{if }\kappa=0,\\
        \frac{1}{\sqrt{-\kappa}}\sinh(\theta\sqrt{-\kappa}), &\text{if }\kappa<0.
    \end{cases}
\end{equation}
Here $J_{\kappa} := \pi/\sqrt{\kappa}$ for $\kappa >0$ and $J_{\kappa} = +\infty$ otherwise.
\end{proposition}
\begin{remark}
\label{rmk:symmetric_difference}
Note that $\X \setminus \partial \Omega \subset T_\delta$, since for any $x \in \X \setminus \partial \Omega$, there exists $y \in \partial \Omega$ such that $|\delta(x)| = \sfd(x,y)$, meaning that $(x,y) \in R_{\delta}\setminus \{(x,x)\mid x\in\X\}$. As a consequence, it makes sense to consider the disintegration of $\mm\restr{\X\setminus\partial\Omega}$ as in \eqref{eq:restriction_mm}. Moreover, we show in Proposition \ref{prop:bd_measure_zero} that, under a further assumption on $\partial\Omega$, we do not lose information on $\mm$ when we restrict to $\X\setminus\partial\Omega$, since $\mm(T_{\delta}\, \triangle \,(\X \setminus \partial \Omega)) = 0$, where $A \,\triangle\, B= (A \setminus B) \cup (B \setminus A)$.
\end{remark}
\begin{remark}
The result in \cite{CM20} holds for essentially non-branching ${\sf MCP}(K,N)$ (for possibly not single-valued distributional Laplacian). For our purposes, it is enough to consider the subclass of $\RCD(K,N)$ spaces.
\end{remark}
\subsection{1D localization and regularity of \texorpdfstring{$\partial \Omega$}{delOmega}}
We define here some regularity assumptions on the open set $\Omega$ and some consequences.
\begin{definition}[Interior/exterior ball condition]
\label{def:int_ball_condition}
Let $(\X,\sfd)$ be a metric space and let $\Omega\subset\X$ be open. We say that $x \in \partial \Omega$ verifies the interior ball condition with respect to $\Omega$ if $\exists\,p_x \in \Omega$, $r_x \in \mathbb{R}^+$ such that 
\begin{equation}
    B_{r_x}(p_x) \subset \Omega \qquad\text{and}\qquad x \in \partial B_{r_x}(p_x). 
\end{equation}
Similarly, we say that $x \in \partial \Omega$ verifies the exterior ball condition if it verifies the interior ball condition with respect to $\X \setminus \bar{\Omega}$.\\
Let $\epsilon > 0$. We say that $\Omega$ verifies a $\epsilon$-uniform interior ball condition if every $x \in \partial \Omega$ verifies the interior ball condition with respect to $\Omega$ and $r_x \ge \epsilon$, for every $x\in\Omega$.
\end{definition}
\begin{remark}
It is well-known that, in $\mathbb{R}^n$, the uniform interior and exterior ball condition for an open set $\Omega$ is equivalent to the $C^{1,1}$ regularity of $\partial \Omega$, meaning that, up to rotation of coordinates, for every $x =(x',x_n) \in \partial \Omega$ there exists a cylinder $C_{\delta}(x):=B_\delta(x') \times (-\delta,\delta)$ such that $C_{\delta}(x) \cap \Omega$ can be written as a subgraph of a $C^{1,1}$-function $f$ with $|f-x_n| \le \delta$.
In particular, the $L^\infty$ bounds on $f''$ are comparable to $\frac{1}{\epsilon}$, where $\epsilon$ is the uniformity of the ball condition (see \cite[Prop.\ 2.7, Rem.\ 2.8]{JN22} for more details). 
In the non-smooth setting, given any open set $\Omega\subset\X$, the boundary of its $2\epsilon$-enlargement satisfies the $\epsilon$-uniform interior ball condition. 
\end{remark}
\begin{proposition}
\label{prop:bd_measure_zero}
Let $(\X,\sfd,\mm)$ be a geodesic $\PI$ space. Let $\Omega\subset \X$ be open with $\partial \Omega$ bounded. Assume that $\Omega$ verifies the interior (or exterior) ball condition. Then $\mm(\partial \Omega) = 0$.
\end{proposition}
\begin{proof}
Fix $0<r<R$. Recall that a $\PI$ space is locally uniformly doubling (see \eqref{eq:doubling_def}) and consider the constant $C_D=C_D(4R)$ as defined therein. Let $x \in \partial \Omega$, then, by the interior ball condition, there exists a point $p_x\in\Omega$ and a radius $r_x<r$ such that $B_{r_x}(p_x) \subset \Omega$ with $x \in \partial B_{r_x}(p_x)$ (note that we can choose $r_x$ such that $r_x < r$ since $(\X,\sfd)$ is geodesic). This, together with the triangle inequality, implies 
\begin{equation}
\label{eq:inclusions}
    B_{r_x}(p_x)\subset B_{2r_x}(x)\cap\Omega\subset B_{4r_x}(p_x).
\end{equation}
Therefore, combining the doubling condition and \eqref{eq:inclusions}, we obtain the following:
\begin{equation}
\label{eq:estimates_measure}
     \mm(B_{2r_x}(x)\cap\Omega)\ge\mm( B_{r_x}(p_x))\ge C_D^{-2} \mm(B_{4r_x}(p_x))\ge C_D^{-2} \mm(B_{2r_x}(x)).
\end{equation}
Repeating this procedure for any $x\in\partial\Omega$, we obtain a covering of $\partial\Omega$ consisting of balls of radius $2r_x<2r$. Thus, since $\X$ is also a separable metric space, we can apply Vitali covering Lemma, so there exists $\{x_i\}_{i\in\N}\subset\partial\Omega$ such that
\begin{equation}
\label{eq:vitali_covering}
     B_{2r_i}(x_i)\cap B_{2r_j}(x_j)=\emptyset\quad \forall\,i\neq j\qquad\text{and}\qquad \partial\Omega\subset\bigcup_{x\in\partial\Omega}B_{2r_x}(x)\subset\bigcup_{i\in\N}B_{10r_i}(x_i)
\end{equation}
having denoted $r_i:=r_{x_i}$, for every $i\in\N$. We are in position to prove the statement, indeed using the doubling property, together with \eqref{eq:estimates_measure} and \eqref{eq:vitali_covering}, we have:
\begin{equation}
\begin{split}
    \mm(\partial \Omega) &\le \mm\bigg(\bigcup_{i\in\N}B_{16r_i}(x_i)\bigg)\le C_D^{-4}\sum_{i=1}^\infty \mm(B_{2r_i}(x_i)) \le C_D^{-2} \sum_{i=1}^\infty \mm(B_{2r_i}(x_i)\cap\Omega) \\
                         &\le C_D^{-2} \mm(B_{2r}(\partial \Omega) \cap \Omega),
\end{split}
\end{equation}
where the last term is finite thanks to boundedness of $\partial\Omega$.
Taking the limit as $r \to 0$, we conclude the proof. The argument for the case where $\Omega$ satisfies the exterior ball condition follows verbatim.
\end{proof}
\begin{remark}
\label{rmk:mesasure_level_sets}
We point out that, under the assumptions of Proposition \ref{prop:bd_measure_zero}, also $\mm(\{\delta=r\})=0$ for every $r>0$ outside of a countable set. This follows by monotonicity of the function $r\mapsto \mm(\{\delta>r\})$, which has an at most countable set of discontinuities.  
\end{remark}

We now introduce a condition involving both the metric and the reference measure, which we call measured interior geodesic condition, in short $\mIGC{\epsilon}$. To this aim, let us define the set
\begin{equation}
    O_\epsilon := \{ x \in \Omega\mid \exists\, \gamma\colon[0,1]\rightarrow\X \text{ such that }x \in \gamma, \gamma_0 \in \partial \Omega,\, \delta(\gamma_1) = \ell(\gamma) \ge \epsilon \}.
\end{equation}
\begin{definition}[$\mIGC{\epsilon}$ condition]
\label{def:migc}
Let $(\X,\sfd,\mm)$ be a metric measure space, let $\Omega\subset\X$ be open and let $\epsilon>0$. We say that $\partial \Omega$ satisfies the measure theoretic $\epsilon$-\emph{interior geodesic condition}, or $\mIGC{\epsilon}$ condition for brevity, if 
\begin{equation}
 \mm(\{ 0 < \delta < \epsilon \} \setminus O_\epsilon) = 0 
\end{equation}
Similarly, we say that $\partial\Omega$ satisfies the measure theoretic $\epsilon$-\emph{exterior geodesic condition}, or simply $\mEGC{\epsilon}$ condition, if the above definition holds with $\X\setminus \bar{\Omega}$ in place of $\Omega$.
\end{definition}
\begin{remark}
Note that for example a ball of radius $\epsilon$ satisfies the $\mIGC{\epsilon}$ condition, if the radial geodesics cover the ball itself, up to negligible sets. 
\end{remark}
We now investigate the equivalence of the $\mIGC{\epsilon}$ condition to the fact that almost all transport rays have length greater or equal than $\epsilon$ inside $\Omega$. Let us define the set 
\begin{equation}
Q_\epsilon:= \{ \alpha \in Q \,|\, \mathscr{H}^1(X_\alpha \cap \Omega) \ge \epsilon \}.    
\end{equation}
Let us relate the sets $O_\epsilon$ and $Q_\epsilon$.
\begin{proposition}
\label{lemma:relation_ray_geo}
Let $(\X,\sfd)$ be a geodesic non-branching metric space and let $u\colon\X\rightarrow\R$ be a $1$-Lipschitz function. Let $\Omega\subset\X$ be open and $\epsilon>0$. Then, with the notation above, 
\begin{equation}
\label{eq:relation_O_eps_Q_eps}
    \Omega \cap \mathfrak{Q}^{-1}(Q_\epsilon) = O_\epsilon \cap T_\delta^{nb}.
\end{equation}
\end{proposition}
For the proof of the proposition, we need the following preliminary lemma. 

\begin{lemma}
\label{lemma:inclusions}
Let $(\X,\sfd)$ be a geodesic non-branching metric space and let $u\colon\X\rightarrow\R$ be a $1$-Lipschitz function. Then, 
\begin{equation}
    A^+\subset a,\qquad\text{and}\qquad A^-\subset b.
\end{equation}
\end{lemma}

\begin{proof}
We prove the first inclusion, the second one being analogous. Let $x\in a^c$ and assume by contradiction that $x\in A^+$. In particular, this means that $x\in T_u$ and, by definition of $a$, there exists a point $y\in T_u$ such that $x\neq y$ and $(y,x)\in \Gamma_u$, while by definition of $A^+$, there exist $z,w\in \Gamma_u(x)$ with $(z,w)\notin R_u$. We claim that $x$ is an intermediate point along geodesics joining $y,z$ and $y,w$. Indeed, by construction $x\neq y$, $x\neq z$, and also, since $(y,x),(x,z)\in \Gamma_u$, we have
\begin{equation}
\label{eq:transport_rel_xyz}
    u(y)-u(z)=u(y)-u(x)+u(x)-u(z)=\sfd(y,x)+\sfd(x,z)\ge\sfd(y,z),
\end{equation}
by triangle inequality. Now, recall that $u$ is $1$-Lipschitz, thus in \eqref{eq:transport_rel_xyz} the equality holds, proving the claim for the triple $y,x,z$. Analogously, the same holds for the triple $y,x,w$. 
Now let $\gamma$ be a geodesic connecting $y$ and $x$, $\eta_1$ be a geodesic connecting $x$ and $z$ and $\eta_2$ be a geodesic connecting $x$ and $w$. By the claim above the curves $\gamma \cup \eta_i$ for $i=1,2$ are geodesics, hence since $z\neq w$, we obtain a contradiction with the non-branching assumption. 
\end{proof}

\begin{proof}[Proof of Proposition \ref{lemma:relation_ray_geo}]
Firstly, we show that $\Omega \cap \mathfrak{Q}^{-1}(Q_\epsilon) \subset O_\epsilon \cap T_\delta^{nb}$. Let $z \in \Omega \cap \mathfrak{Q}^{-1}(Q_\epsilon)$. By definition of $\mathfrak{Q}$, $\mathfrak{Q}^{-1}(Q_\epsilon) \subset T_\delta^{nb}$, hence $z \in X_\alpha$ for some $\alpha\in Q_\epsilon$. Now $\bar X_\alpha$ is a minimizing geodesic for $\delta$ connecting $a(X_\alpha)$ to $\partial \Omega$ (recall that by \eqref{eq:init_final_pts_delta}, $\bar X_\alpha\cap\partial\Omega\neq\emptyset$), in particular there exists $y_\alpha\in\partial\Omega$ such that 
\begin{equation}
\delta(a(X_\alpha)) = \sfd(a(X_\alpha),y_\alpha) \ge \mathscr{H}^1(X_\alpha\cap\Omega)\ge \epsilon,    
\end{equation}
giving that $z \in O_\epsilon$. Secondly, we prove the converse inclusion $O_\epsilon \cap T_\delta^{nb} \subset \Omega \cap \mathfrak{Q}^{-1}(Q_\epsilon)$. Let $x \in O_\epsilon \cap T_\delta^{nb}$; then, by definition of $O_\epsilon$, $x\in\Omega$ and there exists a geodesic $\gamma\colon[0,1]\rightarrow\X$ such that
\begin{equation}
x \in \gamma,\quad\gamma_0 \in \partial \Omega,\quad\text{and}\quad \delta(\gamma_1) = \ell(\gamma) \ge \epsilon.     
\end{equation}
Note that for any $t,s\in(0,1)$, $(\gamma_t,\gamma_s)\in R_\delta$, hence $\gamma_t\in (a\cup b)^c$ for any $t\in (0,1)$, being an interior point of a geodesic. Therefore, Lemma \ref{lemma:inclusions} ensures that $\gamma_t\in (A^+\cup A^-)^c$ for every $t\in (0,1)$, meaning that ${\rm Int}(\gamma)\subset T_\delta^{nb}$. As a consequence there exists $\alpha\in Q$ such that ${\rm Int}(\gamma)\subset X_\alpha$ and, furthermore we have
\begin{equation}
    \mathscr{H}^1(X_\alpha\cap\Omega)\ge \ell(\gamma)\ge \epsilon,
\end{equation}
hence $\alpha\in Q_\epsilon$. We are left to show that $x\in X_\alpha$. If $x\in {\rm Int}(\gamma)$, there is nothing to prove. If $x=\gamma_1$, recall that $x\in T_\delta^{nb}$ and, by construction, $(\gamma_t,x)\in R_\delta$ for every $t\in (0,1)$. This implies that $x$ belongs to the same orbit of $\gamma_t$, namely $x\in X_\alpha$, or equivalently $x\in \mathfrak{Q}^{-1}(Q_\epsilon)$. This concludes the proof.
\end{proof}
\begin{proposition}
\label{lem:equivalence_migceps_lengthtransp}
Let $(\X,\sfd,\mm)$ be an $\RCD(K,N)$ space with $K\in\R$ and $N\in(1,\infty)$. Let $\Omega\subset\X$ be open and let $\epsilon>0$. Set also $Q':=\mathfrak{Q}(\Omega\cap T_\delta^{nb})$. Then, $\partial \Omega$ satisfies the $\mIGC{\epsilon}$ condition if and only if $\sfq(Q' \setminus Q_\epsilon) = 0$.
\end{proposition}
\begin{proof}
Since $\sfq \ll \mathfrak{Q}_\#\mm \restr{T_\delta^{nb}} \ll \sfq$, cf. Remark \ref{rem:mutual_abs_cont}, it is enough to prove that $\partial\Omega$ verifies the $\mIGC{\epsilon}$ condition if and only if $(\mathfrak{Q}_\#\mm \restr{T_\delta^{nb}})(Q' \setminus Q_\epsilon) =0$, i.e.\ $\mm(T_\delta^{nb} \cap \mathfrak{Q}^{-1}(Q' \setminus Q_\epsilon)) =0$. First of all, observe that $Q'=\mathfrak{Q}(\{0<\delta<\epsilon\}\cap T_\delta^{nb})$. Indeed, if $x\in\Omega\setminus\{0<\delta<\epsilon\}\cap T_\delta^{nb}$, then $x\in X_\alpha$ for some $\alpha\in Q'$. However, $\bar\X_\alpha$ is a minimizing geodesic for $\delta$ joining $x$ and $\partial\Omega$. Thus, $X_\alpha\cap\{0<\delta<\epsilon\}\neq \emptyset$, and therefore $\alpha\in \mathfrak{Q}(\{0<\delta<\epsilon\}\cap T_\delta^{nb})$, proving the claimed equality. Second of all, using Proposition \ref{lemma:relation_ray_geo} (which applies since $\RCD(K,N)$ spaces are non-branching by Theorem \ref{thm:nonbranch_rcd}), we deduce that
\begin{equation}
    \mm(T_\delta^{nb} \cap \mathfrak{Q}^{-1}(Q' \setminus Q_\epsilon)) =\mm(T_\delta^{nb} \cap \{0<\delta<\epsilon\}\setminus\mathfrak{Q}^{-1}( Q_\epsilon))\stackrel{ \eqref{eq:relation_O_eps_Q_eps}}{=}\mm(T_\delta^{nb} \cap \{0<\delta<\epsilon\}\setminus O_\epsilon ).
\end{equation}
Finally, applying Proposition \ref{prop:measure_non_branched_set}, together with Remark \ref{rmk:symmetric_difference}, we have that $\mm((T_\delta^{nb}\cap\Omega)\,\triangle\, \Omega)=0$, hence,
\begin{equation}
    \mm(T_\delta^{nb} \cap \mathfrak{Q}^{-1}(Q' \setminus Q_\epsilon)) =\mm(\{0<\delta<\epsilon\}\setminus O_\epsilon ).
\end{equation}
Now the proof of the statement easily follows from the definition of $\mIGC{\epsilon}$ condition.
\end{proof}
Thanks to Proposition \ref{lem:equivalence_migceps_lengthtransp}, we can immediately deduce regularity properties for $\pmb\Delta\delta$ in the interior of $\Omega$. Indeed, on the one hand, we have mild integrability of its regular part, on the other hand, its singular part is separated from $\partial\Omega$.
\begin{corollary}
\label{cor:mild_integrability_laplacian}
Let $(\X,\sfd,\mm)$ be an $\RCD(K,N)$ space with $K \in \mathbb{R}$ and $N \in (1,\infty)$.
Let $\Omega\subset\X$ be open and assume that $\partial \Omega$ satisfies the $\mIGC{\epsilon}$ condition. Then $[\pmb\Delta \delta]^{reg} \in L^1(\{ 0 < \delta <\epsilon \})$.
\end{corollary}
\begin{proof}
By Proposition \ref{lem:equivalence_migceps_lengthtransp}, the $\mIGC{\epsilon}$ condition implies that for $\sfq$-a.e.\ $\alpha \in Q'$ we have
\begin{equation}
\mathscr{H}^1(X_\alpha \cap \Omega) \ge \epsilon,
\end{equation}
where $Q'=\mathfrak{Q}(\Omega\cap T_\delta^{nb})$. In particular, it holds that $\sfd(a(X_\alpha), b(X_\alpha)) \ge \epsilon$, for $\sfq$-a.e.\ $\alpha\in Q'$. Then, applying \cite[Lem.\ 2.16]{CM20} for $\alpha\in Q'$, we get that there exists a constant $C= C(\epsilon,K,N)$ such that
\begin{equation}
\label{eq:uniform_estimate_seminorm_halpha}
    \| h_\alpha' \|_{L^1(0,\epsilon)} \le \frac{C}{\epsilon},\qquad\text{for $\sfq$-a.e. }\alpha\in Q'.
\end{equation}
Note that $X_\alpha\cap\{0<\delta<\epsilon\}\neq\emptyset$ if and only if $\alpha\in Q'$, thus we may compute
\begin{equation}
    \int_{\{ 0 < \delta <\epsilon \}} |[\pmb\Delta \delta]^{reg}|\,\d \mm =\int_{Q'} \int_0^\epsilon |(\log h_\alpha)'| |h_\alpha|\,\d t\,\d \sfq(\alpha) =  \int_{Q'} \int_0^\epsilon |h_\alpha'|\,\d t\,\d \sfq(\alpha) \stackrel{\eqref{eq:uniform_estimate_seminorm_halpha}}{\le}\sfq(Q')\,\frac{C}{\epsilon} \le \frac{C}{\epsilon}.
\end{equation}
This concludes the proof.
\end{proof}
\begin{corollary}
\label{coro:uniform_ball_cond}
Let $(\X,\sfd,\mm)$ be a $\RCD(K,N)$ for $K \in \mathbb{R}$ and $N \in (1,\infty)$ and let $\Omega\subset\X$ be open. Assume that $\partial\Omega$ verifies the $\mIGC{\epsilon}$. Then,
\begin{equation}
\sfq(\{ \alpha \in Q \mid a(X_\alpha) \in B_{\epsilon}(\partial \Omega) \cap \Omega \}) = 0.    
\end{equation}
Similarly, if $\partial\Omega$ verifies the $\mEGC{\epsilon}$ condition, then $\sfq(\{ \alpha \in Q \mid b(X_\alpha) \in B_{\epsilon}(\partial \Omega) \cap (\X \setminus \bar{\Omega}) \}) = 0$.
\end{corollary}
\begin{remark}
\label{rmk:IGCeps_to_reg_Deltadelta}
Under the $\mIGC{\epsilon}$ condition for $\partial\Omega$, $\pmb{\Delta} \delta$ as a Radon functional on $\{ 0 < \delta <\epsilon \}$ is equal to $\pmb{\Delta} \delta = [\pmb{\Delta} \delta]^{reg} \mm$.
Indeed, by Corollary \ref{coro:uniform_ball_cond}, we have that, for every $\varphi \in {\rm Lip}_{bs}(\{ 0 < \delta < \epsilon\})$, $\varphi(a(X_\alpha))=0$ for $\sfq$-a.e.\ $\alpha$, hence
\begin{equation}
    \int\varphi\,\d[\pmb{\Delta} \delta]^{sing}=-\int h_\alpha(a(X_\alpha))\,\varphi(a(X_\alpha))\,\d\sfq(\alpha) =0.
\end{equation}
Thus, the representation formula \eqref{eq:repr_laplacian_delta} has the form
\begin{equation}
\label{eq:IGCeps_to_reg_Deltadelta}
    \la \pmb\Delta\delta,\varphi\ra=\int_\Omega \varphi[\pmb\Delta\delta]^{reg}\,\d\mm.
\end{equation} 
\end{remark}

We consider the map $\proj \colon T_\delta^{nb} \to \partial \Omega$ that projects a point $x\in T_\delta^{nb}$ on the foot of the unique geodesic realizing $\delta(x)$, namely $\proj(x)$ is the only $y \in \partial \Omega$ such that $\sfd(x,y)=\delta(x)$.
We stress that the definition is well-posed, indeed, if there exist $y_1,y_2 \in \partial \Omega$, $y_1 \neq y_2$ such that $\sfd(x,y_1)=\sfd(x,y_2)=\delta(x)$, then
\begin{equation}
(x,y_i) \in \Gamma_\delta, \quad i=1,2\qquad\text{and}\qquad(y_1,y_2) \notin R_\delta.    
\end{equation}
Hence $x \in A^+\cup A^-$, giving a contradiction since $x \in T_\delta^{nb}$.
\begin{proposition}
\label{prop:surjectivity_projection}
Let $(\X,\sfd,\mm)$ be an $\RCD(K,N)$ space with $K\in\R$ and $N\in(1,\infty)$. Let $\Omega\subset\X$ be open and assume $\partial \Omega$ satisfies the $\mIGC{\epsilon}$ condition, for some $\epsilon>0$. Then $\proj\restr {O_\epsilon \cap \{ 0 <\delta <\epsilon\}}$ is surjective.
\end{proposition}
\begin{proof}
We firstly show that $\proj( O_\epsilon \cap \{ 0 <\delta <\epsilon\})$ is closed in $\partial\Omega$. Let $\{z_n\}_{n\in\N} \subset \proj( O_\epsilon \cap \{ 0 <\delta <\epsilon\})$ be such that $z_n \to z$, for some $z\in\partial\Omega$. For any $n\in\N$, there exists $x_n\in O_\epsilon \cap \{ 0 <\delta <\epsilon\}\cap T_\delta^{nb}$ such that $\proj(x_n)=z_n$ and, by definition of $O_\epsilon$, there exists $\gamma^n\colon [0,1]\rightarrow\X$ such that (up to reparameterization and suitable restriction of $\gamma^n$)
\begin{equation}
    x_n\in\gamma^n,\qquad \gamma_0^n=z_n\qquad\text{and}\qquad \delta(\gamma_1^n)=\ell(\gamma^n)=\epsilon.
\end{equation}
Thus, applying Ascoli--Arzel\'a theorem, up to a (not relabeled) subsequence, the sequence of curves $\{\gamma^n\}_{n\in\N}$ uniformly converges to a geodesic $\gamma\colon[0,1]\rightarrow\X$ such that
\begin{equation}
    \gamma_0=z\qquad\text{and}\qquad\delta(\gamma_1)=\ell(\gamma)=\epsilon.
\end{equation}
We claim that $\gamma_t\in T_\delta^{nb}$ for any $t\in (0,1)$. Indeed, passing to the limit the following relation
\begin{equation}
    \delta(\gamma^n_t)=\delta(\gamma^n_t)-\delta(z_n)=\sfd(\gamma^n_t,z_n),
\end{equation}
we deduce that $(\gamma_t,z)\in\Gamma_\delta$, and so $\gamma_t\in T_\delta$. Since, for $t\in (0,1)$, $\gamma_t\in (a\cup b)^c$, Lemma \ref{lemma:inclusions} implies the claim. Finally, this shows that $\gamma_t\in O_\epsilon \cap \{ 0 <\delta <\epsilon\}\cap T_\delta^{nb}$ and $\proj(\gamma_t)=z$.

To conclude, we argue by contradiction, assuming there exists $z \in \partial \Omega \setminus \proj(O_\epsilon \cap \{ 0 <\delta <\epsilon\})$. Then, since $\proj(O_\epsilon \cap \{ 0 <\delta <\epsilon\})$ is closed, there exists $r>0$ such that $B_r(z)\cap \partial \Omega \subset \proj(O_\epsilon \cap \{ 0 <\delta <\epsilon\})^c$. However, by assumption $\mm(\{ 0<\delta<\epsilon \} \setminus O_\epsilon)=0$, thus there exists $y \in B_{r/2}(z) \cap O_\epsilon \cap \{0<\delta < \epsilon\} \cap T_\delta^{nb}$, and this means that
\begin{equation}
    \delta(y)\leq \sfd(y,z)< \frac{r}{2}\qquad\text{and}\qquad\proj(y)\in B_r(z)\cap\partial\Omega.
\end{equation}
This gives a contradiction since $B_r(z)$ had empty intersection with the image of $\proj$.
\end{proof}

We relate the $\mIGC{\epsilon}$ with the uniform interior ball condition.
\begin{proposition}
\label{prop:ball_vs_IGC_epsilon}
Let $(\X,\sfd,\mm)$ be an $\RCD(K,N)$ space with $K\in\R$ and $N\in(1,\infty)$. Let $\Omega\subset\X$ be open. The following hold:
\begin{enumerate}
    \item If $\partial \Omega$ satisfies the $\epsilon$-uniform interior ball condition and every point of $\partial \Omega$ satisfies the exterior ball condition. Then it satisfies the $\mIGC{\epsilon}$ condition;
    \item if $\partial \Omega$ satisfies $\mIGC{\epsilon}$-condition, then it satisfies the $\epsilon$-uniform interior ball condition.
\end{enumerate}
\end{proposition}
\begin{proof}
We prove $(1)$ and, to do so, we show that 
\begin{equation}
    O_\epsilon\cap\{0<\delta<\epsilon\} = \{0<\delta<\epsilon\}.
\end{equation}
Let $x\in\{0<\delta<\epsilon\}$, then there exists a geodesic $\gamma\colon[0,1]\rightarrow\X$ such that $\gamma_0\in\partial\Omega$, $\gamma_1=x$ and $\delta(\gamma_1)=\ell(\gamma)$. By assumptions, there exists $p\in\Omega$ such that $B_\epsilon(p)$ verifies the interior ball condition at the point $\gamma_0$, and there exist $q\in\X\setminus\bar\Omega$, $r>0$ such that $B_{r}(q)$ verifies the exterior ball condition at the point $\gamma_0$. Let $\eta\colon[0,1]\rightarrow\X$ and $\tilde\eta\colon[0,1]\rightarrow\X$ denote the geodesics such that $\eta_0=\gamma_0$, $\eta_1=p$ and $\tilde\eta_0=\gamma_0$, $\tilde\eta_1=q$, respectively. By construction, we have
\begin{equation}
\label{eq:length_eta}
    \delta(p)=\ell(\eta)=\epsilon\qquad\text{and}\qquad-\delta(q)=\ell(\tilde\eta)=r,
\end{equation}
thus, $\sfd(p,q)\le \sfd(p,\eta_0)+\sfd(q,\tilde\eta_0)=r+\epsilon$. We claim that equality holds. Assume by contradiction that this is not the case and consider $\tilde\gamma\colon[0,1]\rightarrow\X$ geodesic joining $p$ and $q$. Then, there exists $t\in (0,1)$ such that $\tilde\gamma_t\in\partial\Omega$ and, in addition
\begin{equation}
    \delta(p)\le\sfd(p,\tilde\gamma_t)<\epsilon \qquad\text{or}\qquad-\delta(q)\le\sfd(q,\tilde\gamma_ t)<r.
\end{equation}
This implies that the equalities in \eqref{eq:length_eta} can not hold at the same time, giving a contradiction. Hence, if we consider the arc-length reparameterizations of $\eta$ and $\tilde\eta$ (without relabeling), the curve $\eta \cup \tilde{\eta}\colon [0,r+\epsilon] \to \X$ defined as
\begin{equation}
(\eta \cup \tilde{\eta})_t:= \begin{cases}
\eta_{r-t} & \text{on }[0,r],\\
\tilde{\eta}_{t-r} &\text{on }[r,r+\epsilon].
\end{cases}    
\end{equation}
is a geodesic between $p$ and $q$. In an analogous way, the concatenation $\gamma\cup\tilde\eta$ of $\gamma$ and $\tilde\eta$ is a geodesic between $x$ and $q$. Therefore, since $\RCD(K,N)$ spaces are non-branching (see Theorem \ref{thm:nonbranch_rcd}), we must have $\gamma \subset \eta$. Since \eqref{eq:length_eta} holds, $x\in O_\epsilon$. This concludes the proof of $(1)$.\\
We prove $(2)$. Let us consider $z\in \partial \Omega$. Hence, by an application of Proposition \ref{prop:surjectivity_projection}, we know that there exists $x \in O_\epsilon\cap \{ 0<\delta <\epsilon\}\cap T_\delta^{nb}$ being such that $\proj(x)=z$. Let us consider the geodesic $\gamma$ such that $\gamma_0 \in \partial \Omega$, $x\in \gamma$ and $\ell(\gamma) = \epsilon$. It can be readily checked that $B_\epsilon(\gamma_1) \subset \Omega$, thus proving the ball condition of radius $\epsilon$ at $z$. By arbitrariness of the point $z$, we conclude.    
\end{proof}

\begin{example}
The following two examples show that Proposition \ref{prop:ball_vs_IGC_epsilon} is sharp in the sense that 
the converse to $(1)$ and $(2)$ does not hold.\\
For item $(1)$, consider $(\mathbb{R}^2,|\cdot|,\mathcal{L}^2)$ and the open set $\Omega := (\mathbb{R}^2 \setminus [0,1]^2) \cap B$, where $B$ is a large ball enclosing $[0,1]^2$. 
We claim that, for $\epsilon \ll 1$, we have $O_\epsilon \cap \{ 0 < \delta < \epsilon \} = \{ 0 < \delta < \epsilon \}$. Indeed, for every $x \in \{ 0 < \delta < \epsilon \}$, we consider the segment connecting $x$ to the closest point in $\partial \Omega$. This segment can be extended to a segment of length $\epsilon$ which can be chosen to be the geodesic in the definition of $O_\epsilon$ at the point $x$. Hence, in particular, $\partial \Omega$ satisfies the $\mIGC{\epsilon}$ condition.
However, this set does not satisfy the exterior ball condition at the point $0$ (see Figure \ref{fig:A}).\\
For item $(2)$, consider $(\mathbb{R}^3,|\cdot|,\mathcal{L}^3)$ and the open set $\Omega:= B_1(0) \setminus S$, where the closed set $S\subset\R^3$ is defined as follows: let $\epsilon>0$ sufficiently small, then 
\begin{equation}
    S: = \bar B_{1-4\epsilon}(0) \cap (\{ x=z=0 \} \cup \{ y=z=0 \}).
\end{equation}
Observe that $\Omega$ satisfies the $\epsilon$-uniform interior ball condition, yet it does not satisfy the $\mIGC{\epsilon}$ condition. 
Indeed, to check this, consider the ball $B_\rho(0)$ with $\rho \ll \epsilon$ and notice that $\mathcal{L}^3(B_\rho(0) \cap (O_\epsilon \setminus \{0 <\delta <\epsilon\}))>0$.
The crucial remark is that, for a point $x \in S$ close to the origin, a tangent ball of radius $\epsilon$ can be chosen only tangent to the plan $\{ z = 0\}$ (see Figure \ref{fig:B}). Thus, the ball condition on the set $S$ does not give any control on the geodesics on all $B_\rho(0)$.
\begin{figure}[ht!]
     \centering
     \begin{subfigure}[b]{0.45\textwidth}
         \centering
         \includegraphics[scale=0.67]{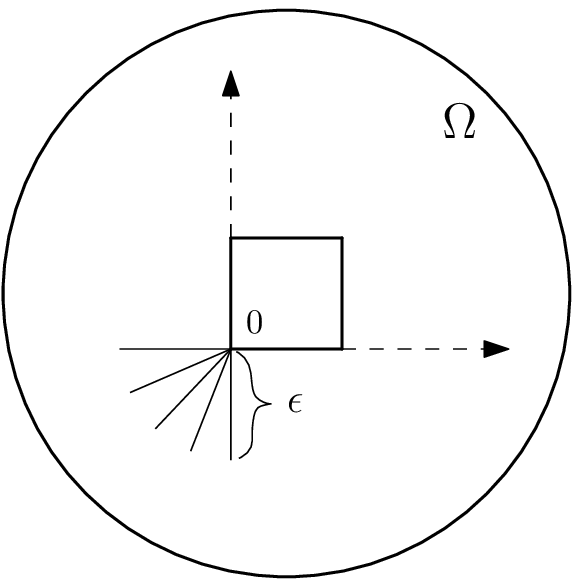}
          \caption{$\Omega = (\mathbb{R}^2 \setminus [0,1]^2) \cap B\subset\R^2$}
         \label{fig:A}
     \end{subfigure}
      \hspace{-0.5cm}
     \begin{subfigure}[b]{0.45\textwidth}
         \centering
         \includegraphics[scale=0.83]{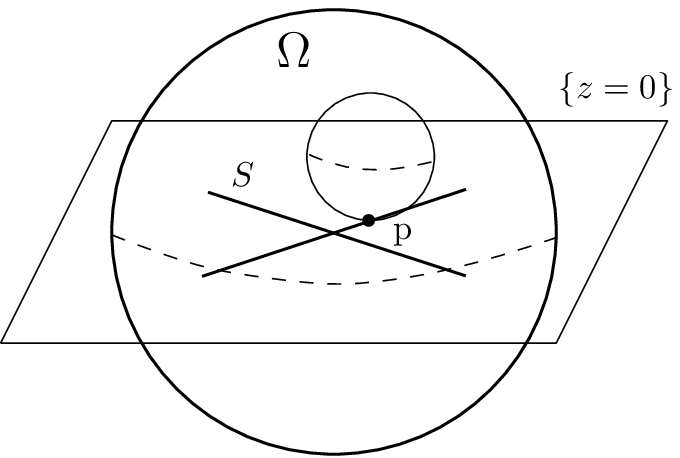}
          \caption{$\Omega= B_1(0) \setminus S\subset\R^3$}
         \label{fig:B}
     \end{subfigure}
    \caption{Counterexamples regarding strict implications in Proposition \ref{prop:ball_vs_IGC_epsilon}.}
\end{figure}
\end{example}
\section{A mean value lemma in \texorpdfstring{$\RCD(K,N)$}{RCD(K,N)} spaces}
\label{sec:mean_value_lemma}
Let $(\X,\sfd,\mm)$ be a $\RCD(K,N)$ for $K \in \mathbb{R}$ and $N \in (1,\infty)$ and let $\Omega\subset \X$ be open and bounded. Let $\delta\colon \X\rightarrow\R$ be the signed distance function from $\partial\Omega$ as in \eqref{eq:signed_distance_fun}, and for a fixed $v\in D\big(\Delta_{\restr{\Omega}}\big)$, define: 
\begin{equation}
\label{eq:def_F}
    F(r):=\int_{\{\delta>r\}}v\,\d\mm,\qquad r\geq 0.
\end{equation}
The goal of this section is to compute the second distributional derivative of the function $F$, to deduce a mean value lemma in $\RCD$ setting. First of all, we recall the following version of the coarea formula: this is a consequence of \cite[Prop.\ 4.2]{Mir03} and of the following identity for the total variation of $f$:
\begin{equation}
    |Df|= |\nabla f|\mm,\qquad\forall\, f\in \Lip_{loc}(\X) \cap BV(\X),
\end{equation}
which holds for proper $\RCD(K,\infty)$ spaces, see \cite{GigliHan14}.

\begin{proposition}[Coarea formula]
\label{prop:coarea}
Let $(\X,\sfd,\mm)$ be an $\RCD(K,N)$ space.
Let $u \in \Lip_{loc}(\X) \cap BV(\X)$. Then $\{u > r \}$ has finite perimeter for a.e. $r \in (0,\infty)$ and for every $f \in L^1(|\nabla u| \mm)$ it holds that
\begin{equation}
\label{eq:coarea_original}
\int \varphi(u) f |\nabla u|\,\d \mm = \int_0^{\infty} \varphi(r) \int f \,\d {\rm Per}(\{ u > r \},\cdot)\,\d r
\end{equation}
for every $\varphi\colon (0,\infty) \to \R$ Borel and bounded.
\end{proposition}
\begin{remark}[Coarea formula for the signed distance]
\label{rmk:coarea_distance}
Let $\Omega\subset\X$ be open and bounded and define the signed distance $\delta$ as in \eqref{eq:signed_distance_fun}. We wish to apply
Proposition \ref{prop:coarea} to $\delta$. In this case, take $\eta \in \Lip_{bs}(\X)$ such that $\eta \equiv 1$ on $\Omega$, then $\eta \delta \in \Lip_{bs}(\X) \cap BV(\X)$, therefore the set
\begin{equation}
    \{\delta>r\}=\{\delta>r\}\cap\Omega=\{\eta\delta>r\},
\end{equation}
has finite perimeter for a.e. $r\in (0,\infty)$. Moreover, for $f \in L^1(\Omega,\mm)$, denote by $\tilde f\in L^1(\mm)$ its extension by $0$ to $\X$, thus, we apply \eqref{eq:coarea_original} as follows: 
\begin{equation}
\label{eq:coarea}
\begin{aligned}
\int_{\Omega} \varphi(\delta) f |\nabla \delta|\,\d \mm &=
\int \varphi(\eta\delta) \tilde f |\nabla (\eta \delta)|\,\d \mm = \int_0^{\infty} \varphi(r) \int \tilde f \,\d {\rm Per}(\{ \eta \delta > r \},\cdot)\,\d r \\
&= \int_0^\infty \varphi(r)\int_{\Omega} f\,\d {\rm Per}(\{ \delta > r\},\cdot)\,\d r.
\end{aligned}
\end{equation}
\end{remark}
The one-dimensional localization and the coarea formula are two possible ways of parameterizing the tubular neighbourhood of $\partial \Omega$ and their relation is shown in the following proposition.
\begin{proposition}
\label{prop:parameterization_tubular_neighborhood}
Let $(\X,\sfd,\mm)$ be an $\RCD(K,N)$ space for $K \in \mathbb{R}$ and $N \in (1,\infty)$. Let $\Omega\subset\X$ be open. Consider the disintegration of $\mm$ as in \eqref{eq:restriction_mm}. Then, for every Borel function $f \colon \X\rightarrow\R$,
\begin{equation}
    \int f \,\d {\rm Per}(\{\delta > s \},\cdot) = \int_Q f(\gamma_\alpha(s))\,h_{\alpha}(s)\,\chi_{\{ \tilde{\alpha}\mid s \in I_{\tilde{\alpha}} \}}(\alpha)\,\d \sfq(\alpha),\qquad\text{for a.e. }s\in\R,
\end{equation}
where $\gamma_\alpha\colon \bar I_\alpha\rightarrow \bar X_\alpha$ is an arc-length parameterization of $X_\alpha$ such that $0\in \bar I_\alpha$ and $\gamma_\alpha(0)\in\partial\Omega$, $\forall\,\alpha \in Q$. 
\end{proposition}
\begin{proof}
For any $\alpha\in Q$, since $\gamma_\alpha$ has metric speed equal to $1$ and $\gamma_\alpha(0) \in \partial \Omega$, we have that, for every $0 \le s \in \bar I_\alpha$,
\begin{equation}
    \delta(\gamma_\alpha(s)) = \delta(\gamma_\alpha(s))- \delta(\gamma_\alpha(0)) = \sfd(\gamma_\alpha(s), \gamma_\alpha(0)) = s
\end{equation}
and similarly for $0 \ge s \in \bar I_\alpha$, we get $\delta(\gamma_{\alpha}(s)) = s$ for every $s \in I_\alpha$. For any $s\in\R$, set 
\begin{equation}
    F(s):= \int_Q f(\gamma_\alpha(s))\,h_\alpha(s)\chi_{\{\tilde{\alpha}\mid s \in I_{\tilde{\alpha}} \}}(\alpha)\,\d \sfq(\alpha)\qquad\text{and}\qquad G(s):=\int_Q f\,\d {\rm Per}(\{ \delta > s\},\cdot).
\end{equation} 
Then, letting $s_1<s_2 \in \R$, we compute: 
\begin{equation}
\begin{aligned}
\label{eq:integrated_rel_delta_gammaalpha}
   \int_{s_1}^{s_2} F(s)\,\d s &=
   \int_Q \int_{I_\alpha} \chi_{(s_1,s_2)}(\delta(\gamma_\alpha(s)))\, f(\gamma_\alpha(s))\,h_{\alpha}(s)\,\d s \,\d\sfq(\alpha)\\
   & = \int \int_{I_\alpha} \left( \chi_{\{ s_1 < \delta < s_2\}} f \right)(\gamma_\alpha(s))h_\alpha(s)\,\d s\,\d \sfq(\alpha) = \int \int \chi_{\{ s_1 < \delta < s_2 \}} f \,\d\mm_\alpha \,\d \sfq(\alpha) \\
   &= \int_{\{ s_1 < \delta < s_2 \}} f \,\d \mm,
\end{aligned}
\end{equation}
since by definition, $\left(\gamma_\alpha\right)_\#( h_\alpha \,\d s)=\mm_\alpha$.
Finally, using coarea formula, we obtain
\begin{equation}
     \int_{s_1}^{s_2} F(s)\,\d s = \int_{\{ s_1 < \delta < s_2 \}} f\,\d \mm  \stackrel{\eqref{eq:coarea}}{=} \int_{s_1}^{s_2} G(s)\,\d s.
\end{equation}
We showed that, for every open interval $I$, $\int_I F\,\d s = \int_I G\,\d s$. Therefore, since the last property is stable under limits of increasing sequence of sets, we are in position to apply monotone class theorem and get that, for every Borel set $B \subset \R$, $\int_B F \,\d s= \int_B G\,\d s$. Thus, $F = G$ for a.e.\ $s$, as claimed.
\end{proof}
We apply coarea formula to compute the first derivative of the function $F$ defined in \eqref{eq:def_F}.
\begin{proposition}[First derivative of $F$]
\label{cor:Fprime}
Let $(\X,\sfd,\mm)$ be a $\RCD(K,N)$ for $K \in \mathbb{R}$ and $N \in (1,\infty)$, let $\Omega\subset\X$ be open and bounded and consider $v \in L^1(\Omega,\mm)$. Then, $\{ \delta > r \}$ is a set of finite perimeter for a.e.\ $r$, $F\in{\rm AC}((0,+\infty),\R)$ and 
\begin{equation}
\label{eq:Fprime}
    F'(r)= -\int_\Omega v\,\d {\rm Per}(\{\delta>r\},\cdot),\qquad\text{for a.e. }r\in (0,+\infty)
\end{equation}
\end{proposition}
\begin{proof}
In view of Remark \ref{rmk:coarea_distance}, we have that for a.e.\ $r>0$ the set $\{ \delta > r \}$ is of finite perimeter. 
We prove the second part of the statement. We notice that $|\nabla \delta|=1$, $\mm$-a.e. from Proposition \ref{prop:eikonal_weak}. Moreover, we compute for $r,h>0$, 
\begin{align}
    |F(r+h)-F(r)|&= F(r)-F(r+h)=\int_{\{\delta>r\}}v\,\d\mm-\int_{\{\delta>r+h\}}v\,\d\mm=
    \int_{\{r<\delta<r+h\}}v\,\d\mm\\
    &=\int_\Omega\chi_{(r,r+h)}(\delta) v|\nabla\delta|\,\d\mm\stackrel{\eqref{eq:coarea}}{=}\int_0^\infty\chi_{(r,r+h)}(s)\int_\Omega v\,\d {\rm Per}(\{\delta>s\},\cdot)\,\d s\\
    &=\int_r^{r+h}\int_\Omega v\,\d {\rm Per}(\{\delta>s\},\cdot)\,\d s.
\end{align}
Since $v\in L^1(\Omega,\mm)$, the function $s\mapsto \int v\,\d {\rm Per}(\{\delta>s\},\cdot)$ is integrable on $(0,\infty)$, we deduce that $F\in {\rm AC}((0,+\infty),\R)$ and \eqref{eq:Fprime} holds.
\end{proof}

\begin{proposition}[Second derivative of $F$]
\label{prop:Fsecond}
Let $(\X,\sfd,\mm)$ be a $\RCD(K,N)$ for $K \in \mathbb{R}$ and $N \in (1,\infty)$, let $\Omega\subset\X$ be open and bounded and let $v \in D(\Delta \restr{\Omega})\cap \Lip(\Omega)$. Then
\begin{equation}
\label{eq:Fsecond_original}
    F''= \int_{\{\delta > \cdot\}} \Delta \restr{\Omega} v \,\d \mm -\delta_\#(v \pmb{\Delta} \delta)
\end{equation}
holds in the sense of distributions on $(0,\infty)$, where, according to \eqref{eq:multiplication_radonfunctional} and \eqref{eq:pushforward_radonfunctional}, $\delta_\#(v \pmb{\Delta} \delta)$ is the Radon functional over $(0,\infty)$ defined by
\begin{equation}
    \delta_\#(v \pmb{\Delta} \delta)(\varphi)=-\int\la \nabla [v (\varphi\circ\delta)],\nabla \delta\ra\,\d\mm,\qquad\forall\,\varphi\in C_c^\infty(0,\infty). 
\end{equation}
\end{proposition}
\begin{proof}
Let us compute the second derivative of the function $F$ in the sense of distributions. Fix $\varphi\in C_c^\infty(0,\infty)$, then, applying Proposition \ref{cor:Fprime}, we have:
\begin{equation}
\label{eq:terms}
\begin{split}
\la F'',\varphi\ra &=\int_0^\infty\varphi'(r) \int v \,{\d \rm Per}(\{ \delta > r \},\cdot)\, \d r \stackrel{ \eqref{eq:coarea}}{=} \int_{\Omega} (\varphi'\circ\delta) v |\nabla \delta|\,\d \mm \stackrel{\eqref{eq:eikonal_weak}}{=} \int_{\Omega} (\varphi'\circ\delta) v |\nabla \delta|^2\,\d \mm
    \\
    &= \int_{\Omega}  v \la \nabla ( \varphi\circ\delta ),\nabla \delta \ra \,\d \mm \stackrel{\text{Prop.}\,\ref{prop:product_sobolev}}{=}\underbrace{- \int_{\Omega}  (\varphi\circ\delta) \la \nabla v,\nabla \delta \ra \,\d \mm}_{(\mathrm{A})}+\int_{\Omega}   \la \nabla [ v(\varphi\circ\delta) ],\nabla \delta \ra \,\d \mm
\end{split}
\end{equation}
For the term $(\mathrm{A})$ in \eqref{eq:terms}, we would like to apply the definition of restricted Laplacian, since $v\in D(\Delta\restr{\Omega})$. To do so, define: 
\begin{equation}
    \psi(r)=\int_0^r\varphi(s)\,\d s,\qquad \forall\,r>0. 
\end{equation}
Notice that $\psi\in C^\infty(0,\infty)$ but it does not have compact support. Nonetheless, its support is separated from $0$, therefore $\psi\circ\delta\in \Lip_{bs}(\Omega)$ and, by chain rule, the following identity holds
\begin{equation}
   \nabla(\psi\circ\delta)=(\psi'\circ\delta)\nabla\delta=(\varphi\circ\delta)\nabla\delta.
\end{equation}
Therefore, we may discuss the term $(\mathrm{A})$ in \eqref{eq:terms} as follows: 
\begin{align}
  \int_{\Omega} (\varphi\circ\delta) \la \nabla v, \nabla \delta \ra \,\d \mm &= \int_{\Omega} \la \nabla v, \nabla(\psi\circ \delta) \ra \,\d \mm\\ 
  &=-\int_\Omega (\psi\circ \delta)\Delta\restr{\Omega}v\,\d\mm= -\int_{\Omega} \int_0^\infty \chi_{(0,\delta(x))}(r) \varphi(r)\,\d r \Delta \restr{\Omega} v(x)\,\d \mm(x)\\ 
  &= -\int_0^\infty \varphi(r) \int_{\{ \delta > r \}} \Delta \restr{\Omega} v\,\d \mm\, \d r.
\end{align}
Therefore, the following holds
\begin{equation}
\label{eq:first_computation_secondder}
    \la F'',\varphi\ra= \int_0^\infty \varphi(r)\,\int_{\{\delta > r\}} \Delta \restr{\Omega} v \,\d \mm\,\d r + \int_\Omega \la \nabla [v (\varphi \circ \delta)], \nabla \delta \ra\,\d \mm, \qquad\forall\, \varphi \in C^\infty_c(0,\infty).
\end{equation}
Regarding the last term in \eqref{eq:first_computation_secondder}, notice that $v,\delta\in\Lip(\Omega)$, thus, according to \eqref{eq:multiplication_radonfunctional} and \eqref{eq:pushforward_radonfunctional}, it corresponds to the Radon functional $-\delta_\#(v\pmb\Delta\delta)$. Indeed we have 
\begin{equation}
\label{eq:delta_pushforward_Deltadelta}
   \delta_\#(v\pmb\Delta\delta)(\varphi)\stackrel{\eqref{eq:pushforward_radonfunctional}}{=} (v\pmb\Delta\delta)(\varphi\circ\delta)\stackrel{\eqref{eq:multiplication_radonfunctional}}{=}\pmb\Delta\delta(v(\varphi\circ\delta))=-\int_{\Omega}   \la \nabla [ v(\varphi\circ\delta) ],\nabla \delta \ra \,\d \mm,
\end{equation}
by definition of distributional Laplacian \eqref{eq:distributional_laplacian}, since $v(\varphi\circ\delta)\in\Lip_{bs}(\Omega)$. This concludes the proof.
\end{proof}

\begin{remark}
Let $(M,g)$ be a complete Riemannian manifold and let $\Omega\subset M$ be open and bounded, equipped with a smooth measure $\mm$. Assuming that $\partial\Omega$ is smooth, $\delta$ is a smooth function in a tubular neighborhood $U=\{0\le\delta<\epsilon\}$ of $\partial\Omega$. Therefore, the term $\delta_\#(v\pmb{\Delta}\delta)$ can be rewritten in a simpler form. Indeed, first of all, for every $\phi\in \Lip_{bs}(U)$, we have that
\begin{equation}
\label{eq:reg_lapl_delta}
    \pmb\Delta \delta(\phi)=\int_U\phi\Delta\delta\,\d \mm,
\end{equation}
and thus, for $\varphi\in C_c^\infty((0,\epsilon))$, we deduce that
\begin{align}
       \delta_\#(v\pmb{\Delta}\delta)(\varphi) &= \int_{\Omega} (\varphi\circ\delta) v \,\Delta\delta \,\d \mm \stackrel{\text{Prop.}\, \ref{prop:coarea}}{=}\int_0^\infty \varphi(r) \int v \,\Delta\delta\d {\rm Per}(\{ \delta > r\},\cdot)\d r
       \\
       &=\int_0^\infty \varphi(r) \int_{\{\delta=r\}} v \,\Delta\delta\d\sigma_r\d r,
\end{align}
where $\sigma_r$ is the Riemannian perimeter measure on the level sets of $\delta$. In our setting, it is still possible to prescribe regularity on $\Omega$ to obtain an analogue to \eqref{eq:reg_lapl_delta}. 
\end{remark}
%

The next goal is to prove \eqref{eq:Fsecond_original} dropping the assumption of the Lipschitz regularity of $v$. To do so, we need the following approximation lemma.
\begin{lemma}
\label{lemma:approximation_to_ext_radfctl}
Let $(\X,\sfd,\mm)$ be an $\RCD(K,N)$ space and let $\epsilon>0$. Let $\Omega \subset \X$ be an open and bounded set and $v \in L^\infty(\Omega,\mm) \cap D(\Delta \restr{\Omega})$ with ${\rm supp} (v) \subset \{ 0 < \delta < \epsilon \}$. Then, there exists a sequence $\{v_n\} \subset {\rm Lip}(\Omega)$ with ${\rm supp}(v_n)\subset \{ 0 < \delta < \epsilon \}$ and $\| v_n \|_{ L^\infty(\Omega,\mm)} \le \| v \|_{ L^\infty(\Omega,\mm)}$ such that $v_n \to v$ in $W^{1,2}(\Omega)$.
\end{lemma}
\begin{proof}
We can assume, without loss of generality, that $v\ge 0$ (otherwise it is enough to write $v=v^+-v^-$ and use linearity).
Let $\eta \in {\rm Lip}_{bs}(\Omega)$ such that $\eta = 1$ on ${\rm supp}(v)$ and ${\rm supp}(\eta) \subset \{ 0 < \delta < \epsilon \}$. Since $v \in W^{1,2}(\Omega)$, by definition $\eta v\in W^{1,2}(\X)$. Thus, there exists a sequence $\{u_n\} \subset {\rm Lip}(\X) \cap W^{1,2}(\X)$ such that $u_n \to \eta v$ in $W^{1,2}(\X)$. Define $\tilde{u}_n:=\eta u_n$, we show that $\tilde{u}_n \to v$ in $W^{1,2}(\Omega)$. Indeed, $\tilde{u}_n-v = \eta u_n-\eta^2 v = \eta (u_n-\eta v)$, thus $\|\tilde{u}_n-v \|_{L^2(\Omega,\mm)} \le \|u_n-\eta v \|_{L^2(\mm)}\, \| \eta\|_{L^\infty(\mm)} \to 0$ as $n \to \infty$ and
\begin{equation}
\| |\nabla(\tilde u_n-v)| \|_{L^2(\Omega,\mm)} \le \| |\nabla ( u_n-\eta v)| \|_{L^2(\mm)}\,\|\eta\|_{ L^\infty(\mm)}+\| u_n-\eta v \|_{L^2(\mm)}\,\||\nabla\eta| \|_{L^\infty(\mm)} \to 0,
\end{equation}
as $n \to \infty$. To conclude, define $v_n:=\tilde{u}_n \wedge v$: since $v\ge 0$, we have that $\supp{v_n}\subset\supp{\tilde u_n}$ and $\| v_n \|_{ L^\infty(\Omega,\mm)} \le \| v \|_{ L^\infty(\Omega,\mm)}$ by construction. Thus, it is only left to prove that $\| v_n - v \|_{W^{1,2}(\Omega)}\to 0$. Therefore, we rewrite
\[ v_n -v = \tilde{u}_n \wedge v-v = \frac{\tilde{u}_n+v-|\tilde{u}_n-v|}{2}-v = \frac{\tilde{u}_n-v}{2}-\frac{|\tilde{u}_n-v|}{2},\]
thus easily proving that $v_n \to v$ in $L^2(\Omega,\mm)$. Using the locality property of the gradient of elements of $W^{1,2}(\Omega)$, we get $\| | \nabla (\tilde{u}_n\wedge v-v) | \|_{L^2(\Omega,\mm)} = \| |\nabla (\tilde{u}_n-v)|\,\chi_{\{ \tilde{u}_n<v \}} \|_{L^2(\Omega,\mm)}\to 0$ as $n \to \infty$, concluding the proof.
\end{proof}
\begin{corollary}
\label{coro:Fsecond}
Let $(\X,\sfd,\mm)$ be an $\RCD(K,N)$ space. Let $\Omega\subset X$ be an open and bounded. Assume $\partial\Omega$ satisfies the $\mIGC{\epsilon}$ condition. Then, for every $v \in D(\Delta\restr{\Omega})\cap L^\infty(\Omega,\mm)$ such that $\supp {v}\subset\{0\le \delta <\epsilon\}$, 
\begin{equation}
\label{eq:Fsecond}
    F''= \int_{\{\delta > \cdot\}} \Delta \restr{\Omega} v \,\d \mm -\delta_\#(v [\pmb\Delta\delta]^{reg}\,\mm)
\end{equation}
in the sense of distributions in $(0,\infty)$.
\end{corollary}
\begin{proof}
Given $v \in L^\infty(\Omega,\mm) \cap D(\Delta \restr{\Omega})$, we can repeat verbatim the proof of Proposition \ref{prop:Fsecond} up to the identity \eqref{eq:first_computation_secondder}.
To conclude, it is enough to show that, for every $v \in L^\infty(\Omega,\mm) \cap D(\Delta \restr{\Omega})$,
\begin{equation}
    \int \la \nabla [v (\varphi \circ \delta)], \nabla \delta \ra\,\d \mm =- \delta_\# (v [\pmb{\Delta}\delta]^{reg}\mm)(\varphi).
\end{equation}
Indeed, since $\supp {(v(\varphi\circ\delta))}\subset\{0<\delta<\epsilon\}$, up to multiplying for a suitable cut-off function, we may assume that $\supp {v}\subset\{0<\delta<\epsilon\}$.
We consider the approximation  given by Lemma \ref{lemma:approximation_to_ext_radfctl} obtaining a sequence $\{v_n\}\subset\Lip_{bs}(\Omega)$ such that
\begin{equation}
    \supp{v_n}\subset\{0<\delta<\epsilon\},\quad \|v_n\|_{L^\infty(\Omega,\mm)}\le \|v\|_{L^\infty(\Omega,\mm)}\quad\text{and}\quad v_n\xrightarrow[W^{1,2}(\Omega)]{}v \text{ as }n\to \infty.
\end{equation}
Since $v_n\in\Lip_{bs}(\Omega)$, recalling \eqref{eq:delta_pushforward_Deltadelta}, we obtain 
\begin{equation}
    \int \la \nabla [v_n (\varphi \circ \delta)],\nabla \delta\ra\,\d \mm =- \delta_\#(v_n\,\pmb{\Delta} \delta)(\varphi) =- \int v_n\,(\varphi \circ \delta)\,[\pmb{\Delta} \delta]^{reg}\,\d \mm
\end{equation}
where in the second equality, we used \eqref{eq:IGCeps_to_reg_Deltadelta}, which holds since $v_n(\varphi\circ\delta)\in\Lip_{bs}(\{0<\delta<\epsilon\})$ and the $\mIGC{\epsilon}$ condition holds on $\partial \Omega$. Then, we take the limit as $n\to \infty$, using dominated convergence theorem on the right hand side, concluding the claim and therefore the proof, using that $\int v\,(\varphi \circ \delta)\,[\pmb{\Delta} \delta]^{reg}\,\d \mm = \delta_\# (v [\pmb{\Delta} \delta]^{reg}\mm)(\varphi)$.
\end{proof}

\section{First-order asymptotics}
\label{sec:first_order_asymptotics}

This section is devoted to the proof of Theorem \ref{thm:first_order_asymptotics}. We start by providing a few technical tools that are used in the proof.  

\begin{theorem}[Weak Gauss-Green formula]
\label{thm:gauss_green_distancefromaset}
Let $(\X,\sfd,\mm)$ be an $\RCD(K,N)$ space and let $\Omega\subset\X$ be open and bounded. Then, for every $w\in D(\div)$, it holds
\begin{equation}
\label{eq:gauss_green_distancefromaset}
\int_{\{ \delta > r \}} \div(w)\,\d \mm=-\int\langle w,\nabla \delta\rangle
\,\d {\rm Per}(\{ \delta > r \},\cdot),\quad\text{ for a.e.\ }r>0.
\end{equation}
\end{theorem}

\begin{proof}
Let $\varphi\in C_c^\infty(0,\infty)$, then, by coarea formula, we have 
\begin{equation}
\label{eq:psi_div}
    \int_0^\infty \varphi(s)\int_{\{\delta>s\}}\div(w)\d\mm\,\d s=\int \int_0^{\delta(x)}\varphi(s)\,\d s\,\div(w)\d\mm(x)=\int (\psi\circ\delta)\, \div(w)\,\d\mm,
\end{equation}
where $\psi(t)=\int_0^t\varphi(s)\d s$. Since, $w\in D(\div)$, we may integrate by parts the last integral in \eqref{eq:psi_div}, obtaining:
\begin{equation}
    \int (\psi\circ\delta)\, \div(w)\,\d\mm=-\int \la\nabla(\psi\circ\delta),w\ra\,\d\mm=-\int (\varphi\circ\delta)\la\nabla\delta,w\ra\,\d\mm.
\end{equation}
An application of the coarea formula, cf. Proposition \ref{prop:coarea}, together with Proposition \ref{prop:eikonal_weak}, finally yields:
\begin{equation}
    \int_0^\infty \varphi(s)\int_{\{\delta>s\}}\div(w)\,\d\mm\,\d s=-\int_0^\infty  \varphi(s)\int\la\nabla\delta,w\ra
\,\d {\rm Per}(\{ \delta > r \},\cdot)\,\d s.
\end{equation}
Since $\varphi$ is arbitrary, we conclude. 
\end{proof}

\begin{remark}
This result can be regarded as a weaker form of the Gauss-Green formula of \cite[Prop.\ 2.30]{APP21}, since is formulated for \emph{almost every} level set of $\delta$. The advantage is that we don't need neither the notion of Sobolev vector field, nor the one of its trace. Moreover, in the smooth setting, formula \eqref{eq:gauss_green_distancefromaset} identifies the (outward-pointing) unit normal to $\{\delta>s\}$ as $-\nabla\delta$. 
\end{remark}

Under the regularity assumption $\mIGC{\epsilon}$ for $\partial\Omega$, we deduce a second variation formula for the measure of the super level set of the distance function, as a corollary of Theorem \ref{thm:gauss_green_distancefromaset}. 

\begin{corollary}
\label{coro:Linfty_est_perimeter}
Let $(\X,\sfd,\mm)$ be an $\RCD(K,N)$ space for $K\in\R$ and $N\in (1,\infty)$ and let $\Omega\subset\X$ be open and bounded. 
Assume $\partial\Omega$ satisfies the $\mIGC{\epsilon}$ condition. Then, the function $r\mapsto {\rm Per}(\{\delta>r\})$ is $W^{1,1}(0,\epsilon)$ and its distributional derivative is 
\begin{equation}
\label{eq:distributional_perimeter}
    \frac{d}{dr}{\rm Per}(\{\delta>r\})=\int[\pmb\Delta\delta]^{reg}\,\d{\rm Per}(\{\delta>r\},\cdot),\qquad\text{for a.e. }r\in (0,\epsilon).
\end{equation}
In particular, $(0,\infty) \ni r\mapsto {\rm Per}(\{\delta>r\})$ has a representative which is continuous at the point $r=0$ and $\Omega$ has finite perimeter in $\X$.
\end{corollary}
\begin{proof}
We first show that $(0,\epsilon) \ni r \mapsto {\rm Per}(\{ \delta > r \}) \in \R$ belongs to $W^{1,1}_{loc}(0,\epsilon)$\footnote{Given an open interval $I\subset \R$, we recall that $f\in W^{1,1}_{loc}(I)$ if, for any $J\Subset I$ open interval, $f\in W^{1,1}(J)$.}. Fix $J=(\bar s,\bar r)\Subset (0,\epsilon)$. Consider a good cut-off $\phi\in \Lip_{bs}(\Omega) \cap D(\Delta)$ (as in Proposition \ref{good_cut_off}), identically equal to $1$ on $\delta^{-1}(J)$ and with compact support in $\{\bar s/2<\delta<(\epsilon+\bar r)/2\}$. We claim that
\begin{equation}
\label{eq:claim_phidelta}
    \phi\delta\in D(\Delta)\qquad\text{and}\qquad \Delta(\phi\delta)=\delta\Delta\phi+2\la\nabla\phi,\nabla\delta\ra+\phi[\pmb\Delta\delta]^{reg}.
\end{equation}
Indeed $\phi\delta\in \Lip_{bs}(\Omega)$ and, for any $g\in \Lip_{bs}(\X)$,
\begin{align}
    \int_\X\la\nabla(\phi\delta),\nabla g\ra\,\d\mm &=\int_\X\la\phi\nabla\delta,\nabla g\ra\,\d\mm+\int_\X\la\delta\nabla\phi,\nabla g\ra\,\d\mm\\
    &=\int_\X\la\nabla\delta,\nabla (\phi g)\ra\,\d\mm+\int_\X\la\nabla\phi,\nabla (\delta g)\ra\,\d\mm-2\int_\X g\la\nabla\phi,\nabla\delta\ra\,\d\mm\\
    &=-\la \pmb\Delta\delta,g\phi\ra-\int_\X\delta g\Delta \phi\,\d\mm-2\int_\X g\la\nabla\phi,\nabla\delta\ra\,\d\mm.
\end{align}
We discuss the term involving the Radon functional $\pmb\Delta\delta$. By our choice of $\phi$, it follows that, for any $g\in \Lip_{bs}(\X)$, $
\supp{\phi g}\subset\supp{\phi}\cap\Omega\subset B_{\epsilon}(\partial\Omega)\cap\Omega.$
Therefore, using \eqref{eq:IGCeps_to_reg_Deltadelta} we get
\begin{equation}
\label{eq:assolutamente_da_spostare}
    \la \pmb\Delta\delta,g\phi\ra=\int_\Omega g\phi[\pmb\Delta\delta]^{reg}\,\d\mm.
\end{equation}
In addition, recalling that $Q'=\mathfrak{Q}(\Omega\cap T_\delta^{nb})$, on the one hand, for $\sfq$-a.e.\ $\alpha\in Q'$, $\sfd(a(X_\alpha),x)>(\epsilon-\bar r)/2$ for any $x\in\supp{\phi}$ by Proposition \ref{lem:equivalence_migceps_lengthtransp}, on the other hand, recalling that $\delta(b(X_\alpha))\leq 0$, we have that, for $\sfq$-a.e.\ $\alpha\in Q'$, $\sfd(b(X_\alpha),x)>\bar s/2$ for any $x\in\supp{\phi}$. As a consequence, using the bounds \eqref{eq:bound_low_upp_laplacian_distance} on the regular part of $\pmb\Delta\delta$ (note that the disintegration of $\mm\restr{T_\delta^{nb}}$ in $\Omega$ involves only transport rays in $Q'$), we obtain $\phi[\pmb\Delta\delta]^{reg}\in L^\infty(\Omega,\mm)$, proving the claim \eqref{eq:claim_phidelta}.
Since $\phi\delta\in D(\Delta)$, we can apply \eqref{eq:gauss_green_distancefromaset} with $w=\nabla(\phi\delta)$. We fix a representative of $\la \nabla(\phi\delta),\nabla\delta\ra$; then, there exists $N_0\subset J$ with $\mathcal{L}^1(N_0) = 0$ such that 
\begin{equation}
    \int_{\{\delta>r\}} \div(w)\,\d \mm = \int\langle \nabla(\phi\delta),\nabla \delta\rangle
    \,\d {\rm Per}(\{ \delta > r \},\cdot), \qquad\text{for every }r \in J \setminus N_0.
\end{equation}
Furthermore, by Remark \ref{rmk:mesasure_level_sets}, there exists $N_1\subset (0,\epsilon)$ with $\mathcal{L}^1(N_0) = 0$ such that $\mm(\{ \delta =r \}) = 0$, for every $r \in (0,\epsilon) \setminus N_1$. Set $N:= N_1 \cup N_0$; then, for every $s,r \in J \setminus N$ with $s<r$, we obtain: 
\begin{equation}
\label{eq:perimeter_laplacian}
\begin{split}
    \int_{\{ s<\delta < r \}}[\pmb\Delta\delta]^{reg}\,\d\mm& =\int_{\{ s<\delta < r \}} \div(w)\,\d \mm =-\int_{\{ \delta > r \}} \div(w)\,\d \mm+\int_{\{\delta>s\}} \div(w)\,\d \mm\\
    &=\int\langle \nabla(\phi\delta),\nabla \delta\rangle
\,\d {\rm Per}(\{ \delta > r \},\cdot)-\int\langle \nabla(\phi\delta),\nabla \delta\rangle
\,\d {\rm Per}(\{ \delta > s \},\cdot)\\
    &={\rm Per}(\{ \delta > r \})-{\rm Per}(\{ \delta > s \}).
\end{split}
\end{equation}
From this identity, applying coarea formula \eqref{eq:coarea}, we finally obtain, for every $s,r \in J \setminus N$ with $s<r$,
\begin{equation}
    {\rm Per}(\{ \delta > r \}) - {\rm Per}(\{ \delta > s \})=\int_s^r\, \int[\pmb\Delta\delta]^{reg}\,\d{\rm Per}(\{\delta>t\},\cdot)\,\d t.
\end{equation}
Exchanging the roles of $r,s$, we deduce an analogue formula for every $s,r \in J \setminus N$ with $r<s$.
Hence, for a fixed $s\in J\setminus N$, we have that, for every $r\in J\setminus N$,
\begin{equation}
\label{eq:absolutely_continuous_representative_rhs_perimeter}
{\rm Per}(\{ \delta > r \})  ={\rm Per}(\{ \delta > s \})+ \int_s^r\, \int[\pmb\Delta\delta]^{reg}\,\d{\rm Per}(\{\delta>t\},\cdot)\,\d t.
\end{equation}
Since $\|[\pmb{\Delta} \delta]^{reg}\|_{L^1(\{ 0 < \delta <\epsilon \})} < \infty$ by Corollary \ref{cor:mild_integrability_laplacian}, the right hand side of \eqref{eq:absolutely_continuous_representative_rhs_perimeter} extends to an absolutely continuous function on all $J$, being the integral of a function in $L^1(J)$, plus a constant. Therefore, also recalling that $ {\rm Per}(\{ \delta > \cdot \})\in L^1(0,\epsilon)$ by coarea, it belongs to $W^{1,1}(J)$. In addition, \eqref{eq:absolutely_continuous_representative_rhs_perimeter} is an absolutely continuous representative of $r\mapsto {\rm Per}(\{\delta>r\})$, thus its distributional derivative on $J$ is given by formula \eqref{eq:distributional_perimeter}. Since the distributional derivative in \eqref{eq:distributional_perimeter} is actually in $L^1(0,\epsilon)$ by Corollary \ref{cor:mild_integrability_laplacian}, we deduce that ${\rm Per}(\{\delta>\cdot\})\in W^{1,1}(0,\epsilon)$, as claimed. Finally, the lower semicontinuity of the perimeter implies that ${\rm Per}(\{ \delta > 0 \}) = {\rm Per}(\Omega)<\infty$.
\end{proof}

Recall the definition of $e(t,r,s)$, the Neumann heat kernel on the half-line, that is 
\begin{equation}
    e(t,r,s) = \frac{1}{\sqrt{4 \pi t}} \left(e^{-\frac{(r-s)^2}{4t}}+e^{-\frac{(r+s)^2}{4t}}\right)\quad\text{for }(t,r,s) \in (0,\infty)\times [0,\infty)\times [0,\infty).
\end{equation}

With a slight abuse of notation, we denote the space $L^\infty_t\left((0,\infty),L^1_r(0,\infty)\right)$ as $L^\infty_t L^1_r$.

\begin{lemma}[Duhamel's principle]
\label{prop:duhamel}
Let $\varsigma(t,r) \in L^\infty_t L^1_r$ and $v_0,v_1 \in C^\infty(0,\infty) \cap L^\infty(0,\infty)$. Then, there exists a unique weak solution $v\in L^\infty_{t,{loc}}L^\infty_r$ to the non-homogeneous heat equation on the half-line:
\begin{equation}
\label{eq:inh_heat_halfline}
\begin{cases}
(\partial_t -\partial_r^2)v(t,r) = \varsigma(t,r) &\qquad \text{ in } \mathscr{D}'((0,\infty)\times(0,\infty)),\\
v(0,r) = v_0(r) &\qquad \text{for }r >0,\\
\partial_r v(t,0) = v_1(t) &\qquad \text{for }t >0.\\
\end{cases}
\end{equation}
In particular, $v\in C^\infty((0,\infty)\times[0,\infty))$ and we have 
\begin{equation}
\label{eq:sol_inh_heat_halfline}
    v(t,r) = \int_0^\infty e(t,r,s)\,v_0(s)\,\d s +\int_0^t\int_0^\infty e(t-\tau,r,s)\varsigma(\tau,s)\,\d s\,\d \tau - \int_0^t e(t-\tau,r,0)\,v_1(\tau)\,\d \tau.
\end{equation}
\end{lemma}
\begin{proof}
If $\varsigma$ is smooth, then a classical solution $v$ to \eqref{eq:inh_heat_halfline} exists and is unique. In particular, $v$ is given by \eqref{eq:sol_inh_heat_halfline} and $v\in C^\infty((0,\infty)\times[0,\infty))$. We refer to \cite{Evans10} for further details.
Thus, by linearity, it is enough to study the case where $v_0 = v_1 = 0$: we wish to prove that
\begin{equation}
\label{eq:sol_inh_heat_halfline_data0}
    v(t,r) = \int_0^t\int_0^\infty e(t-\tau,r,s)\varsigma(\tau,s)\,\d s\,\d \tau
\end{equation} 
is the unique solution to \eqref{eq:inh_heat_halfline} in this situation.
We begin by proving existence in $L^\infty_{t,loc} L^\infty_r$. Consider a family of mollifiers $\{\eta_\epsilon\}$ such that $\eta_\epsilon \ge 0$, $\| \eta_\epsilon \|_{L^1(\mathbb{R})} = 1$ and ${\rm supp} \,\eta_\epsilon \subset (-\epsilon,\epsilon)$, and define 
\begin{equation}
T_\epsilon(f)(t,r):= \int \int f(\tau,s)\,\eta_\epsilon(t-\tau)\,\eta_\epsilon(r-s)\,\d \tau\,\d s\,\qquad\forall\,f \in L^1_{loc}((0,\infty)\times (0,\infty)),
\end{equation}
with the convention that $f(t,\cdot)=0$ for $t<0$. Set $\varsigma_\epsilon := T_\epsilon(\varsigma)$, then, by standard properties of the convolution, $\varsigma_\epsilon \in L^\infty_{t}L^1_r$ with
\begin{equation}
    \| \varsigma_\epsilon \|_{L^\infty_t L^1_r} \le \| \varsigma \|_{L^\infty_t L^1_r}\qquad\text{and}\qquad\varsigma_\epsilon \xrightarrow{\epsilon\to0}\varsigma\quad \text{in}\quad L^1_{{loc},t}L^1_r.
\end{equation}
For any $\epsilon>0$, denote by $v_\epsilon$ the unique solution to \eqref{eq:inh_heat_halfline} with $\varsigma_\epsilon$ as source term and $v_0=v_1=0$, that is $v_\epsilon(t,r)=\int_0^t\int_0^\infty e(t-\tau,r,s)\varsigma_\epsilon(\tau,s)\,\d s\,\d \tau$.
For a fixed $T>0$, we estimate for $t \le T$
\begin{equation}
|v_\epsilon(t,r)|\le \int_0^t \| e(t-\tau,\cdot,\cdot) \|_{L^\infty}\,\d \tau\,\| \varsigma_\epsilon \|_{L^\infty_t L^1_r} \le \int_0^t \frac{1}{\sqrt{\pi(t-\tau)}}\,\d \tau\,\| \varsigma \|_{L^\infty_t L^1_r} \le C(T)\,\| \varsigma \|_{L^\infty_t L^1_r},
\end{equation}
therefore $\{v_\epsilon\}$ is equibounded in $L^\infty((0,T) \times (0,\infty))$, hence there exists $\tilde v \in L^\infty((0,T) \times (0,\infty))$ such that $v_\epsilon \xrightharpoonup{*} \tilde v$ in $L^\infty((0,T) \times (0,\infty))$, which together with the convergence of $\varsigma_\epsilon$ gives that $\tilde v(t,r)=v(t,r)$, where $v$ is defined in \eqref{eq:sol_inh_heat_halfline_data0}. Moreover, by passing in the limit in the distributional formulation for $v_\epsilon$, we have that $v$ solves \eqref{eq:inh_heat_halfline}, giving existence of a weak solution.
We prove uniqueness in $L^\infty_{t,loc}L^\infty_r$. Let $v,\tilde v \in L^\infty((0,\infty)\times(0,\infty))$ be two solutions to \eqref{eq:inh_heat_halfline} and set $w:= v-\tilde v$. Then, by linearity, $w$ solves \eqref{eq:inh_heat_halfline} with $v_0 = v_1 = 0$ and $\varsigma = 0$. We consider the family $\{\eta_\epsilon\}$ as above and define $w_\epsilon := T_\epsilon(w)$ which belongs to $C^\infty((0,\infty)\times(0,\infty))$ and solves \eqref{eq:inh_heat_halfline}. On the one hand, by uniqueness of classical solutions, $w_\epsilon =0$, on the other hand $w_\epsilon \xrightharpoonup{*} w$ in $L^\infty((0,T)\times(0,\infty))$. Therefore, $w = 0$.
To conclude, notice that by the explicit formula \eqref{eq:sol_inh_heat_halfline} and the regularity properties of the Neumann heat kernel $e$, we get that $v\in C^\infty((0,\infty)\times[0,\infty))$.
\end{proof}
Regarding the notation in the next theorem, we refer the reader to Definitions \ref{def:dirichlet_heat_flow} and \ref{def:heat_content}.
\begin{theorem}
\label{thm:first_order_asymptotics_main}
Let $(\X,\sfd,\mm)$ be an $\RCD(K,N)$ space for $K\in\R$ and $N\in (1,\infty)$ and let $\Omega\subset\X$ be open and bounded. 
Assume $\partial\Omega$ satisfies the $\mIGC{\epsilon}$ condition. Moreover, assume that there exists $\rho >0$ such that 
\begin{equation}
\label{eq:higher_int_delta}
    [\pmb{\Delta} \delta]^{reg} \in L^{1+\rho}(\{ 0 < \delta < \epsilon \}).
\end{equation}
Then, the heat content associated with $\Omega$ admits the following asymptotic expansion
\begin{equation}
\label{eq:heat_content_asymptotics_mainthm}
    Q_\Omega(t)=\mm(\Omega)-\sqrt{\frac{4t}{\pi}}{\rm Per}(\Omega)+O\left(t^{\frac{2(1+\rho)-1}{2(1+\rho)}}\right)\qquad\text{as }t\to 0^+.
\end{equation}
\end{theorem}
\begin{remark}
Let $E\subset\Omega$ be  a closed set with ${\rm Cap}(E)=0$, if \eqref{eq:heat_content_asymptotics_mainthm} holds for $\Omega$, then the same expansion holds for the set $\Omega\subset E$ with the same reminder term. Indeed, on the one hand, due to Remark \ref{rem:heat_content_under_changeofcaprep}, $Q_{\Omega}(t)=Q_{\Omega\setminus E}(t)$ for every $t>0$. On the other hand, 
\begin{equation}
    \mm(\Omega)=\mm(\Omega\setminus E)\qquad\text{and}\qquad {\rm Per}(\Omega)={\rm Per}(\Omega\setminus E)
\end{equation}
since $\mm(\Omega\, \triangle\, (\Omega\setminus E))=0$.
As a consequence of this remark, the expansion in \eqref{eq:heat_content_asymptotics_mainthm} holds when considering as $\Omega$ the set in the second picture in Figure \ref{fig:B}.

\end{remark}
\begin{proof}[Proof of Theorem \ref{thm:first_order_asymptotics_main}]
Consider the compact set $E:=\overline{ \{ \delta > \epsilon/2 \}}\subset\Omega$ and let $\varphi$ be a good cut-off function (in the sense of Proposition \ref{good_cut_off}) such that $\varphi = 1$ on $E$ and $\supp{\varphi} \subset \Omega$.
We define $\phi:= 1-\varphi$ and, since $\varphi \in \Lip_{bs}(\Omega) \cap D(\Delta)$, $\Delta \varphi \in L^\infty(\mm)$, we have that $\phi \in D(\Delta\restr{\Omega})$ with $\Delta\restr{\Omega} \phi \in L^\infty(\Omega,\mm)$. Letting $u_t:=h^\Omega_t\chi_\Omega$, we compute
\begin{equation}
    \int_\Omega (1-u_t)\,\d \mm = \int_\Omega (1-u_t) \phi\,\d \mm +\int_\Omega (1-u_t)\varphi\,\d \mm = (\mathrm{A})+(\mathrm{B}).
\end{equation}
As consequence of Lemma \ref{coro:kacs_principle}, $(\mathrm{B}) = o(t)$, as $t\to 0^+$. We focus on $(\mathrm{A})$. Defining the auxiliary function $F(t,r):= \int_{\{ \delta > r \}}(1-u_t) \phi \,\d \mm $, we can rewrite $(\mathrm{A}) = F(t,0)$. Therefore, in order to conclude, it is enough to show that 
\begin{equation}
\label{eq:final_claim_F}
F(t,0) =\sqrt{\frac{4t}{\pi}}{\rm Per}(\Omega) + O\left(t^{\frac{2(1+\rho)-1}{2(1+\rho)}}\right),\qquad\text{as }t\to 0^+.
\end{equation}
To deduce such an equality, we apply the Duhamel's principle, cf. Lemma \ref{prop:duhamel}, to the function $F(t,r)$, then evaluate at $r=0$. 
First of all, the initial datum is given by $F(0,r) = 0$. Second of all, for the boundary condition, Proposition \ref{cor:Fprime} gives that $\partial_r F(t,0) = -\int \phi\,(1- u_t)\,\d {\rm Per}(\Omega,\cdot)$. In addition, since $u_t \in W^{1,2}_0(\Omega)$, we can apply Proposition \ref{prop:trace_W120} and obtain $u_t = 0$, ${\rm Per}(\Omega,\cdot)$-a.e., thus, recalling that $\phi\equiv 1$ on $\partial\Omega$, $(1-u_t)\phi = 1$, ${\rm Per}(\Omega,\cdot)$-a.e.. Therefore, $\partial_r F(t,0) = - {\rm Per}(\Omega)$ for every $t>0$.
Third of all, for what concerns the source term, note that we may apply Corollary \ref{coro:Fsecond} to $F$, since $(1-u_t)\phi\in D(\Delta\restr{\Omega})\cap L^\infty(\Omega,\mm)$ (see Propositions \ref{prop:weak_max_principle} and \ref{prop:chain_rule_restricted_laplacian}), obtaining that
\begin{equation}
\label{eq:foa_1}
    \partial^2_{r} F(t,\cdot)= \int_{\{\delta > \cdot\}} \Delta \restr{\Omega} [\phi(1-u_t)] \,\d \mm -\delta_\#(\phi(1-u_t) [\pmb{\Delta} \delta]^{reg}\mm)
\end{equation}
holds in the sense of distribution. Moreover, since $u_t$ solves \eqref{eq:dirichlet_heat_equation}, we get that
\begin{equation}
\label{eq:foa_2}
    \partial_t F(t,r) = -\int_{\{ \delta > r\}} \phi \Delta \restr{\Omega} u_t\,\d \mm,
\end{equation}
where $\Delta \restr{\Omega} u_t =\Delta_\Omega u_t$ as a consequence of \eqref{eq:rel_laplacians}.
Using \eqref{eq:foa_1} and \eqref{eq:foa_2}, we obtain 
\begin{equation}
\label{eq:foa_3}
\begin{aligned}
    \partial_t F(t,r) -\partial^2_{r} F(t,r) &= -\int_{\{ \delta > r\}} \phi \Delta \restr{\Omega} u_t\,\d \mm - \int_{\{\delta > r\}} \Delta \restr{\Omega} [\phi(1-u_t)] \,\d \mm +\delta_\#(\phi(1-u_t) [\pmb{\Delta} \delta]^{reg}\mm)\\
    & \stackrel{\eqref{eq:chain_restricted_laplacian}}{=}  \int_{\{\delta >r\} }(2 \la \nabla u_t, \nabla \phi \ra-(1-u_t)\Delta \restr{\Omega} \phi)\,\d \mm + \delta_\#(\phi(1-u_t) [\pmb{\Delta} \delta]^{reg}\mm) =:\mu_t(r).
\end{aligned}
\end{equation}
We now show that $\mu_t$ satisfies the hypotheses of Lemma \ref{prop:duhamel}. We write $\mu_t = \mu^1_t+\mu^2_t$, where $\mu_t^1:=g_t \mathcal{L}^1= \left(\int_{\{\delta >\cdot\} }(2 \la \nabla u_t, \nabla \phi \ra-(1-u_t)\Delta \restr{\Omega} \phi)\,\d \mm\right) \mathcal{L}^1$ and $\mu_t^2:=\delta_\#(\phi(1-u_t)[\pmb{\Delta} \delta]^{reg}\mm)$. We claim that the total variation of $\mu_t$ is uniformly bounded, namely
\begin{equation}
\label{eq:firstclaim_tv}
    \sup_{t>0} |\mu_t|(0,\infty)< \infty.
\end{equation} 
We compute the total variation of $\mu_t^1$: notice that, up to a cut-off argument with a good cut-off function which is $1$ on $\Omega$, we have $\phi\in D(\Delta)\cap L^\infty(\mm)$, so by definition one has $\nabla\phi\in D(\div)$ and $\Delta\phi=\div( \nabla \phi )$, $\mm$-a.e. in $\Omega$. Analogously, we can extend $1-u_t$ to a function in $W^{1,2}(\X)\cap L^\infty(\mm)$. Thus, by Proposition \ref{prop:leibniz_divergence}, $(1-u_\tau)\nabla\phi\in D(\div)$ and we can apply the weak Gauss-Green formula, cf. Theorem \ref{thm:gauss_green_distancefromaset}, obtaining
\begin{equation}
    \begin{aligned}
    \label{eq:computation_for_R1}
    g_t(r) &=\int_{\{\delta >r\} }(2 \la \nabla u_t, \nabla \phi \ra-(1-u_t)\Delta \restr{\Omega} \phi)\,\d \mm=\int_{\{ \delta > r \}}\big((1-u_t)\Delta\restr{\Omega}\phi-2\div((1-u_t)\nabla\phi)\big)\,\d \mm\\
    & =\int_{\{ \delta > r \}}(1-u_t)\Delta\restr{\Omega}\phi\,\d \mm-\int 2(1-u_t) \la \nabla\phi, \nabla \delta \ra\,\d {\rm Per}(\{\delta > r \},\cdot)
    \end{aligned}
\end{equation} 
for a.e.\ $r>0$. Hence, using the weak maximum principle, cf. Proposition \ref{prop:weak_max_principle}, and coarea formula \eqref{eq:coarea}, we have
\begin{equation}
\begin{aligned}
    \|g_{t}\|_{L^1(0,\infty)} &\le \int_0^\infty \int 2 |1-u_t|\,|\nabla \phi|\,\d {\rm Per}(\{ \delta > s \},\cdot)\,\d s + \int_0^\infty \int_{\{ \delta > s \}} |1-u_t||\Delta\restr{\Omega} \phi|\,\d \mm\,\d s\\
    &\le 2 {\rm Lip}(\phi) \mm(\Omega)+\| \Delta\restr{\Omega} \phi \|_{L^\infty(\Omega,\mm)}\,\mm(\Omega)\,{\rm diam}(\Omega),
\end{aligned}
\end{equation}
and this bound is uniform as $t>0$. This proves that $g_t(r) \in L^\infty_t L^1_r$. It remains to treat the term $\mu_t^2$.
Recall that, given two metric spaces $(\Y,\sfd_{\Y}),({\rm Z},\sfd_{\rm Z})$, a signed measure $\mu$ on $\Y$ and a Borel map $f \colon \Y \to {\rm Z}$, then $|f_{\#}\mu| \le f_{\#}|\mu|$. 
Hence, using once again the weak maximum principle, we get:
\begin{equation}
\begin{aligned}
    |\mu_t^2|(0,\infty) \le \delta_{\#}(|\phi(1-u_t)[\pmb{\Delta} \delta]^{reg}| \mm)(0,\infty) = \int_{\Omega} |\phi(1-u_t) [\pmb{\Delta} \delta]^{reg}|\,\d\mm \le \| \phi \|_{L^\infty} \|[\pmb{\Delta} \delta]^{reg}\|_{L^1(\{ 0 < \delta < \epsilon \})}.
\end{aligned}
\end{equation}
This implies that $\sup_t |\mu_t^2|(0,\infty)<\infty$, proving \eqref{eq:firstclaim_tv}. 

We claim that $\mu_t\ll\mathcal{L}^1$: $\mu_t^1$ is already absolutely continuous. The same is true for $\mu_t^2$, for every $t>0$. Indeed, let $N \in \mathscr{B}(\R)$ such that $\mathcal{L}^1(N) = 0$. If $t \notin N$, then ${\rm Per}(\{ \delta > t \},\delta^{-1}(N)) =0$, being ${\rm Per}(\{ \delta > t \},\cdot)$ concentrated on the set $\{ \delta = t \}$. Therefore, by coarea formula, we have 
\begin{equation}
    \mm(\delta^{-1}(N)) = \int_0^\infty {\rm Per}(\{ \delta > t \},\delta^{-1}(N))\,\d t = \int_N {\rm Per}(\{ \delta > t \},\delta^{-1}(N))\,\d t=0.
\end{equation}
As a consequence, $\delta^{-1}(N)$ is $\mm$-negligible, hence we finally deduce 
\begin{equation}
    \delta_\#(\phi(1-u_t)[\pmb{\Delta} \delta]^{reg}\mm)(N) = \int_{\delta^{-1}(N)}\phi(1-u_t)[\pmb{\Delta} \delta]^{reg}\,\d\mm = 0,
\end{equation}
proving that $\mu_t^2 \ll \mathcal{L}^1$, and so $\mu_t \ll \mathcal{L}^1$. Finally, combining this property and \eqref{eq:firstclaim_tv}, we get that 
\begin{equation}
    \mu_t=\varsigma\mathcal{L}^1\qquad\text{and}\qquad\varsigma(t,r)\in L_t^\infty L^1_r.
\end{equation}
We are in position to apply Lemma \ref{prop:duhamel} with $\mu_t$ as above, $v_0(r) :=F(0,r)=0$ and $v_1(t) := \partial_rF(t,0)= -{\rm Per}(\Omega)$, and we have
\begin{multline}
    F(t,0) = \overbrace{\int_0^t \int_0^\infty e(t-\tau,0,s) g_\tau(s)\,\d s\,\d \tau}^{\mytag{$(\mathrm{R}_1)$}{termR1}}\\
     +\underbrace{\int_0^t \int_0^\infty e(t-\tau,0,s)\,\d\delta_\#[(1-u_\tau)\phi [\pmb{\Delta} \delta]^{reg}\mm](s)\,\d \tau}_{\mytag{$(\mathrm{R}_2)$}{termR2}}+\underbrace{\int_0^t e(t-\tau,0,0)\,{\rm Per}(\Omega)\,\d \tau}_{\mytag{$(\mathrm{P})$}{termP}}.
     \end{multline}
The principal term \ref{termP} is explicitly given by: 
\begin{equation}
\text{\ref{termP}} = \int_0^t \frac{1}{\sqrt{\pi(t-\tau)}}\,\d \tau\,{\rm Per}(\Omega) = \sqrt{\frac{4t}{\pi}}\,{\rm Per}(\Omega).
\end{equation}
In order to conclude, it is enough to prove that $|\text{\ref{termR1}}|+|\text{\ref{termR2}}|= O\Big(t^{\frac{2(1+\rho)-1}{2(1+\rho)}}\Big)$. We start with \ref{termR1}:
by \eqref{eq:computation_for_R1}, we can rewrite this term as follows
\begin{equation}
    \label{eq:expression_R1}
    \text{\ref{termR1}} = \int_0^t \int_0^\epsilon e(t-\tau,0,s)\Bigg(\int_{\{ \delta > s \}}(1-u_\tau)\Delta\restr{\Omega}\phi\,\d \mm-\int 2(1-u_\tau)\la\nabla\phi,\nabla\delta\ra\,\d {\rm Per}(\{\delta > s\},\cdot)\Bigg)\,\d s\,\d \tau
\end{equation}
Therefore, after the change of variable $\tau=t\theta$, we have that
\begin{equation}
\begin{aligned}
\frac{\text{\ref{termR1}}}{t} =&\underbrace{\int_0^1 \int_0^\epsilon e(t(1-\theta),0,s)\int_{\{ \delta > s \}}(1-u_{t\theta})\Delta\restr{\Omega}\phi\,\d \mm\,\d s\,\d \theta}_{\mytag{$(\mathrm{R}_{11})$}{termR11}}\\
&\underbrace{-\int_0^1 \int_0^\epsilon e(t(1-\theta),0,s)\int 2(1-u_{t \theta})\la \nabla\phi, \nabla \delta \ra\,\d {\rm Per}(\{\delta > s\},\cdot)\,\d s\,\d \theta}_{\mytag{$(\mathrm{R}_{12})$}{termR12}}.
\end{aligned}
\end{equation}
We first consider the term \ref{termR11}. For simplicity of notations, we define $\tilde{G}_t(s):=\int_{\{\delta > s\}} (1-u_t)\Delta\restr{\Omega} \phi\,\d \mm$ and $G_t(s): = \chi_{(0,\epsilon)}(s)\,\tilde{G}_t(s)+\chi_{(-\epsilon,0)}(s)\,\tilde{G}_t(-s)$ and
\begin{equation}
\label{eq:formula_computation_R11}
    \int_0^1 \left[ \int_0^\epsilon \frac{1}{\sqrt{\pi t (1-\theta)}}e^{-\frac{s^2}{4t(1-\theta)}} \tilde{G}_{t\theta}(s) \,\d s\right]\,\d \theta = \int_0^1 \left(P_{t(1-\theta)} G_{t\theta}\right)(0)\,\d \theta,
\end{equation}
where $P_t$ denotes the heat flow on $\R$. Fix $\theta\in (0,1)$. We check that $\| G_{t\theta} \|_{L^\infty} \to 0$ as $t \to 0$, indeed for $s\in\R$, 
\begin{equation}
    |G_{t\theta}(s)| \le \int_{\{\epsilon/2 < \delta < \epsilon\}} |1-u_t|\,|\Delta\restr{\Omega} \phi|\,\d \mm \le \int_{\{\epsilon/2 < \delta < \epsilon\}} |1-u_t|\,\d \mm \,\| \Delta\restr{\Omega} \phi \|_{L^\infty(\{ \epsilon/2 <\delta < \epsilon \})} \stackrel{\eqref{eq:first_order_kacs_principle}}{=} o(t).
\end{equation}
Moreover, by the maximum principle, we have that $\| P_{t(1-\theta)} G_{t\theta} \|_{L^\infty} \le \| G_{t\theta} \|_{L^\infty}$ for every $t>0$, thus
$|P_{t(1-\theta)} G_{t\theta}(0)| \to 0$, as $t\to 0^+$.
Therefore, since
\begin{equation}
\begin{aligned}
 |P_{t(1-\theta)} G_{t\theta}(0)| &\le \|P_{t(1-\theta)} G_{t\theta}\|_{L^\infty} \le \| G_{t\theta} \|_{L^\infty} \le \int_{\{\epsilon/2 < \delta < \epsilon\}} |1-u_t|\,\d \mm \,\| \Delta\restr{\Omega} \phi \|_{L^\infty(\{ \epsilon/2 <\delta < \epsilon \})} \\
 &\le \mm(\{\epsilon/2 < \delta < \epsilon\})\,\,\| \Delta\restr{\Omega} \phi \|_{L^\infty(\{ \epsilon/2 <\delta < \epsilon \})}
\end{aligned}
\end{equation}
we can apply dominated convergence theorem in \eqref{eq:formula_computation_R11} and we get that $|\text{\ref{termR11}}|= o(1)$. We study the term \ref{termR12}. We compute, for $t$ sufficiently small,
\begin{equation}
\begin{aligned}
    |\text{\ref{termR12}}|&\,\,\,= \left|\int_0^1 \int_0^\epsilon e(t(1-\theta),0,s)\int 2(1-u_{t \theta})\la \nabla\phi,\nabla \delta \ra\,\d {\rm Per}(\{\delta > s\},\cdot)\,\d s\,\d \theta\right| \\
    &\stackrel{\eqref{eq:coarea}}{=} \left|\int_0^1 \int_\Omega 2 e(t(1-\theta),0,\delta) (1-u_{t \theta})\la \nabla\phi,\nabla \delta \ra\,\d \mm\,\d \theta\right|\\
    & \,\,\,\le \Lip\, \phi\, \int_0^1 \int_{\{ \frac{\epsilon}{2} < \delta < \epsilon \}} e^{-\frac{\delta^2}{4t (1-\theta)}} \frac{2}{\sqrt{\pi t (1-\theta)}} (1-u_{t\theta})\,\d \mm \,\d \theta \\
    & \,\,\,\le  \frac{2\Lip\, \phi}{\sqrt{\pi t}}\int_0^1 \frac{1}{\sqrt{1-\theta}}\, \| 1-u_{t\theta} \|_{L^1(\{ \epsilon /2 < \delta < \epsilon\})}\,\d \theta  \stackrel{\eqref{eq:first_order_kacs_principle}}{\le} C \sqrt{t}\int_0^1 \frac{\theta}{\sqrt{1-\theta}}\,\d \theta  = C \sqrt{t},\\
\end{aligned}
\end{equation}
where $C>0$ is constant independent of $t$. This shows that $|\text{\ref{termR1}}|=o(t)$ as $t\to 0^+$.

We study the term \ref{termR2}. Since $\partial \Omega$ verifies $\mIGC{\epsilon}$ condition, by Corollary \ref{coro:Linfty_est_perimeter}, we have that 
\begin{equation}
    \label{eq:perimeter_linf}
    \| {\rm Per}(\{ \delta > s \})\|_{L^\infty(0,\epsilon)}<\infty.
\end{equation}
Hence, for $t$ small, we compute
\begin{align}
    |\text{\ref{termR2}}|&=\left|\int_0^t \frac{1}{\sqrt{\pi(t-\tau)}} \int_{\{ 0<\delta <\epsilon\}} e^{-\frac{\delta^2}{4(t-\tau)}}\,[\pmb\Delta\delta]^{reg}\,\d \mm \,\d \tau\right| \\
    &\le \int_0^t \frac{1}{\sqrt{\pi(t-\tau)}} \left( \int_{\{ 0 < \delta < \epsilon \}}e^{-\frac{\delta^2}{4(t-\tau)}\frac{1+\rho}{\rho}} \,\d\mm\right)^{\frac{\rho}{1+\rho}} \left( \int_{\{ 0 < \delta < \epsilon \}} |[\pmb\Delta\delta]^{reg} |^{1+\rho}\,\d \mm \right)^{\frac{1}{1+\rho}}\,\d \tau\\
    &\stackrel{\eqref{eq:higher_int_delta}}{\le} C\int_0^t \frac{1}{\sqrt{\pi(t-\tau)}}\left( \int_0^\epsilon e^{-\frac{s^2}{4(t-\tau)}\frac{1+\rho}{\rho}} {\rm Per}(\{ \delta > s\})\,\d s \right)^{\frac{\rho}{1+\rho}}\,\d \tau\\
    & \stackrel{\eqref{eq:perimeter_linf}}{\le} C \int_0^t \frac{1}{\sqrt{t-\tau}} \left( \int_0^\epsilon e^{-\frac{s^2}{4(t-\tau)}\frac{1+\rho}{\rho}}\,\d s \right)^{\frac{\rho}{1+\rho}}\,\d \tau \\
    &\le C \int_0^t \frac{1}{\sqrt{t-\tau}} \left( \int_0^\infty e^{-s^2}\,\d s \right)^{\frac{\rho}{1+\rho}}\left( \sqrt{\frac{\rho(t-\tau)}{1+\rho}} \right)^{\frac{\rho}{1+\rho}}\,\d \tau\\
    &\le C\int_0^t (t-\tau)^{-\frac{1}{2(1+\rho)}}\,\d \tau \le C\int_0^t \tau^{-\frac{1}{2(1+\rho)}}\,\d \tau\le C \frac{t^{1-\frac{1}{2(1+\rho)}}}{1-\frac{1}{2(1+\rho)}} = C t^{\frac{2(1+\rho)-1}{2(1+\rho)}},
\end{align}
where the constant $C>0$ changes line by line and is independent of $t$. This proves that $|\text{\ref{termR2}}|=O\Big(t^{\frac{2(1+\rho)-1}{2(1+\rho)}}\Big)$ as $t\to 0^+$, and thus the same holds for $|\text{\ref{termR1}}|+|\text{\ref{termR2}}|$. Finally, this implies that \eqref{eq:final_claim_F} is verified, concluding the proof.
\end{proof}
\appendix

\section{On the definitions of local Sobolev spaces}
\label{app:local_sobolev_spaces}

For the sake of completeness, we compare the relevant definitions of local Sobolev spaces present in literature and we prove they are equivalent. The following definition can be found in \cite[Def.\ 2.14]{AmbrosioHonda18}.

\begin{definition}
\label{def:local_sob_AH}
Let $\X$ be a metric measure space and let $\Omega\subset \X$ be open. We say that $f\in L^2(\Omega,\mm)$ belongs to $H^{1,2}(\Omega)$ if the following conditions hold: 
\begin{itemize}
    \item[i)] $\varphi f\in W^{1,2}(\X)$, for any $\varphi\in \Lip_{bs}(\Omega)$;
    \item[ii)] $|D f|\in L^1(\Omega,\mm)$.
\end{itemize}
We equip this space with the following norm 
\begin{equation}\label{eq:norm_sob_space_AH}
    \|f\|^2_{H^{1,2}(\Omega)}=\|f\|^2_{L^2(\Omega,\mm)}+\int_\Omega |D f|^2\,\d\mm.
\end{equation}
\end{definition}
For $f\in L^2(\Omega,\mm)$ verifying item $(1)$ of Definition \ref{def:local_sob_AH}, the local minimal weak upper gradient $|Df|$ can be defined as follows: choose a sequence $\{\chi_n\}\subset\Lip_{bs}(\Omega)$ such that $\{\chi_n=1\}\nearrow \Omega$, and set 
\begin{equation}
\label{eq:local_up_grad}
    |D f|=|D(f\chi_n)|,\qquad \mm\text{-a.e. on }\{\chi_n=1\}\setminus\{\chi_{n-1}=1\},
\end{equation}
where $f\chi_n\in W^{1,2}(\X)$ by item $(1)$ of Definition \ref{def:local_sob_AH}. It can be easily checked that \eqref{eq:local_up_grad} is independent on the choice of the sequence $\{\chi_n\}$. 
\begin{theorem}
\label{thm:equivalence_local_sobolev}
Let $(X,\sfd,\mm)$ be a metric measure space and let $\Omega\subset X$ be open. Then $W^{1,2}(\Omega)=H^{1,2}(\Omega)$ and it holds that
\begin{equation}
\label{eq:equiv_norm_W_H}
    \frac{1}{2}\int_\Omega |D f|^2\,\d\mm = {\rm Ch}_{\Omega}(f).
\end{equation}
\end{theorem}
\begin{proof}
First of all, we prove the inclusion $W^{1,2}(\Omega)\subset H^{1,2}(\Omega)$. Let $f\in W^{1,2}(\Omega)$, then, by definition, there exists an optimal sequence $\{f_k\}\subset \Lip_{loc}(\Omega)$ such that
\begin{equation}
    f_k\xrightarrow[L^2(\Omega,\mm)]{} f\qquad\text{and}\qquad  2{\rm Ch}_{\Omega}(f)=\lim_{k\to \infty}\int_\Omega (\lip\, f_k)^2\,\d\mm.
\end{equation}
As a consequence, for any $\varphi\in \Lip_{bs}(\Omega)$, the sequence $\{\varphi f_k\}\subset\Lip_{bs}(X)$ approximates $\varphi f$ in $L^2(X,\mm)$, and also 
\begin{equation}
    {\rm Ch}(\varphi f)\leq \limi_{k\to \infty}{\rm Ch}(\varphi f_k)\leq \limi_{k\to \infty}\int_\Omega\left[\varphi^2(\lip\, f_k)^2+f_k^2(\lip\,\varphi)^2\right]\,\d\mm<+\infty.
\end{equation}
Thus, item $(1)$ of Definition \ref{def:local_sob_AH} is verified. Now consider a sequence $ \{\chi_n\}$ as in \eqref{eq:local_up_grad}, then, denoting by $\Omega_n=\{\chi_n=1\}\setminus\{\chi_{n-1}=1\}$, we have 
\begin{align}
    \int_\Omega |Df|\,\d\mm &=\sum_{n=1}^\infty \int_{\Omega_n} |D(f \chi_n)|^2\,\d\mm\leq
    \sum_{n=1}^\infty \limi_{k\to \infty } \int_{\Omega_n} (\lip (f_k \chi_n))^2\,\d\mm\\ 
    & \leq 2 \sum_{n=1}^\infty \limi_{k\to \infty } \int_{\Omega_n}\left[\chi_n^2(\lip f_k)^2+f_k^2(\lip \chi_n)^2\right]\,\d\mm\\
    &= 2 \sum_{n=1}^\infty \limi_{k\to \infty } \int_{\Omega_n}(\lip f_k)^2\,\d\mm\leq 2 \limi_{k\to\infty}\int_\Omega(\lip f_k)^2\,\d\mm<+\infty.
\end{align}
This settles also item $(2)$ of Definition \ref{def:local_sob_AH}, proving that $f\in H^{1,2}(\Omega)$ and $\le$ in \eqref{eq:equiv_norm_W_H} holds.
For the converse inclusion, let $f \in H^{1,2}(\Omega)$; using a Meyers-Serrin type-strategy, we find a sequence $\{f_m\} \subset {\rm Lip}_{bs}(\Omega)$ such that $f_m \to f$ in $H^{1,2}(\Omega)$. Using the $L^2(\Omega,\mm)$-lower semicontinuity of ${\rm Ch}_{\Omega}$ we have that:
\begin{equation}
    {\rm Ch}_\Omega (f)\le \limi_{m\to\infty} {\rm Ch}_\Omega (f_m) = \limi_{m\to\infty}\frac{1}{2}\int_\Omega  |D f_m|^2\,\d \mm=\frac{1}{2}\int_{\Omega} |D f|^2\,\d \mm<\infty,
\end{equation}
where the second to last equality holds since $f_m \in W^{1,2}(\X)$, with bounded support in $\Omega$. This implies that $f\in W^{1,2}(\Omega)$ and $\ge$ in \eqref{eq:equiv_norm_W_H} holds.
\end{proof}
Assume $(\X,\sfd,\mm)$ is infinitesimally Hilbertian. We can define the gradient operator restricted to $\Omega$, $\nabla\colon H^{1,2}(\Omega)\to L^2(T\X)\restr{\Omega}$ as follows: let $f\in H^{1,2}(\Omega)$ and pick a sequence $\{\chi_n\}$ as above, then set
\begin{equation}
\label{eq:def_grad_omega}
\nabla f := \sum_{n=1}^\infty \nabla (f\chi_n)\chi_{\{\chi_n=1\}\setminus\{\chi_{n-1}=1\}}.
\end{equation}
On the one hand, the series \eqref{eq:def_grad_omega} converges in $L^2(T\X)$; indeed,  
\begin{equation}
\label{eq:aux_seq_grad_omega}
    \left|\sum_{n=m}^k \nabla (f\chi_n)\chi_{\Omega_n}\right|\le \sum_{n=m}^k \left|\nabla (f\chi_n)\right|\chi_{\Omega_n}=|Df|\sum_{n=m}^k \chi_{\Omega_n},
\end{equation}
where $\Omega_n=\{\chi_n=1\}\setminus\{\chi_{n-1}=1\}$. Therefore, by dominated convergence theorem and using $|Df|\in L^2(\Omega,\mm)$, we conclude that the sequence \eqref{eq:aux_seq_grad_omega} is Cauchy, thus proving the claim. On the other hand, by locality of the global gradient in $L^2(T\X)$, \eqref{eq:def_grad_omega} does not depend on the choice of the sequence $\{\chi_n\}$.
We also point out that, whenever $f \in W^{1,2}(\X)$, its restriction in $\Omega$ belongs to $H^{1,2}(\Omega)$; in this case, \eqref{eq:def_grad_omega} is consistent with the gradient of $f \in W^{1,2}(\X)$ multiplied by $\chi_{\Omega}$.

\begin{remark}
\label{rmk:W12_Hilbert}
It is convenient to introduce Definition \ref{def:local_sob_AH}, since if $\X$ is infinitesimally Hilbertian, meaning that $W^{1,2}(\X)$ is a Hilbert space, the same property is inherited by $H^{1,2}(\Omega)$. Indeed, it can be readily checked that
\begin{equation}
\label{eq:from_carre_to_scprod}
    |Df|^2 = \la \nabla f, \nabla f \ra  \quad \mm\text{-a.e.\ in }\Omega,\text{ if }f \in H^{1,2}(\Omega).
\end{equation}
In particular, this implies that $H^{1,2}(\Omega)$ is a Hilbert space, since $f\mapsto \||Df|\|^2_{L^2(\Omega,\mm)}$ defines a quadratic form. Thus, by Theorem \ref{thm:equivalence_local_sobolev}, we immediately deduce that $W^{1,2}(\Omega)$ is a Hilbert space. 
\end{remark}

\begin{proposition}
\label{prop:product_sobolev}
Let $(\X,\sfd,\mm)$ be an infinitesimally Hilbertian metric measure space and let $\Omega\subset \X$ be open. Let $u,v\in H^{1,2}(\Omega)\cap L^\infty(\Omega,\mm)$, then $uv\in H^{1,2}(\Omega)$ and
\begin{equation}
\label{eq:chain_localgradient_appendix}
    \nabla (u v) = v\nabla u + u \nabla v.
\end{equation}
\end{proposition}
\begin{proof}
For a given $\varphi\in \Lip_{bs}(\Omega)$, choose $\psi\in \Lip_{bs}(\Omega)$ such that $\psi\equiv 1$ on ${\rm supp}(\varphi)$. Then, 
\begin{equation}
\varphi uv=\varphi\psi uv=(\varphi u)(\psi v)\in W^{1,2}(\X),   
\end{equation}
by basic properties of $W^{1,2}(\X)$. Then, item $(1)$ of Definition \ref{def:local_sob_AH} is verified. On the other hand, by locality of the minimal weak upper gradient, it follows that $|D(uv)|\in L^1(\Omega,\mm)$ and also item $(2)$ of Definition \ref{def:local_sob_AH} is satisfied.
It is left to prove \eqref{eq:chain_localgradient_appendix}.
We consider $\{\chi_n\}$ as above, with $\chi_n\ge 0$, and 
\begin{equation}
    u  v  \chi_n^2 = (u  \chi_n) (v  \chi_n).
\end{equation}
Then, using the definition \eqref{eq:def_grad_omega} of gradient restricted to $\Omega$, we compute
\begin{equation}
\begin{aligned}
    \nabla (u  v) &= \sum_{n=1}^\infty \nabla (u  v   \chi_n^2) \chi_{\{ \chi_n^2 = 1 \} \setminus \{ \chi_{n-1}^2 = 1 \}} = \sum_{n=1}^\infty \left( \nabla (u \chi_n) v \chi_n+ u \chi_n  \nabla (v \chi_n) \right) \chi_{\{ \chi_n = 1 \} \setminus \{ \chi_{n-1} = 1 \}} \\
    & = \left( \sum_{n=1}^\infty \nabla (u \chi_n) \chi_{\{ \chi_n = 1 \} \setminus \{ \chi_{n-1} = 1 \}} \right) v + u \left( \sum_{n=1}^\infty \nabla (v \chi_n) \chi_{\{ \chi_n = 1 \} \setminus \{ \chi_{n-1} = 1 \}} \right)\\
    &= v\nabla u +u  \nabla v.
\end{aligned}
\end{equation}
This concludes the proof. 
\end{proof}
We introduce here a useful calculus rule for $W^{1,2}_0(\Omega)$. We refer to Section \ref{sec:heat_on_Omega} for the precise definition. 

\begin{proposition}
\label{prop:composition_W120}
Let $(\X,\sfd,\mm)$ be a metric measure space and let $\Omega\subset \X$ be open. Let $\psi \in C^1(\mathbb{R})$ with $\psi(0)=0$ and bounded derivative. Then, for every $f \in W^{1,2}_0(\Omega)$ we have $\psi \circ f \in W^{1,2}_0(\Omega)$.
\end{proposition}
\begin{proof}
Since $f\in W^{1,2}_0(\Omega)$, there exists a sequence $\{f_n\} \in \Lip_{bs}(\Omega)$ such that $f_n \to f$ in $W^{1,2}(\X)$. Then $\psi \circ f_n\in \Lip_{bs}(\Omega)$ and, in addition $\supp{(\psi\circ f_n)}\subset \supp{f_n}\Subset \Omega$, using that $\psi(0)=0$. To conclude, it is enough to check that $\psi \circ f_n \to \psi \circ f$ in $W^{1,2}(\X)$. First of all, we have 
\begin{equation}
    \| \psi\circ f_n - \psi\circ f \|^2_{L^2(\Omega,\mm)} \le \Lip(\psi)\,\| f_n- f \|^2_{L^2(\Omega,\mm)}\xrightarrow{n\to\infty}0.
\end{equation}
Second of all, by the chain rule $\d(\psi\circ f_n)=(\psi'\circ f_n)\d f_n$, $\mm$-a.e. in $\X$, therefore we have 
\begin{align}
    \int_{\Omega} |\d (\psi\circ f_n-&\psi\circ f)|_*^2\,\d \mm = \int_{\Omega} |(\psi'\circ f_n)\d f_n - (\psi'\circ f)\d f|_*^2\,\d \mm\\
    &\le  2\int_{\Omega} |\psi'\circ f_n-\psi'\circ f|^2 |\d f|_*^2\,\d \mm +2\int_{\Omega} |\psi'\circ f_n|^2 |\d (f_n - f)|_*^2\,\d \mm \\
    & \le 2\int_{\Omega} |\psi'\circ f_n-\psi'\circ f|^2 |\d f|_*^2\,\d \mm +  2\Lip (\psi)^2\,\||\d f_n - \d f|_*\|_{L^2(\Omega,\mm)}^2.\label{eq:continuity_psi}
\end{align}
Taking the limit as $n \to \infty$, the second term converges to $0$, since $f_n \to f$ in $W^{1,2}_0(\Omega)$. For the first term, notice that $\psi'\in C(\R)$, thus $\psi'\circ f_n \to \psi'\circ f$ $\mm$-a.e. up to a subsequence, and 
\begin{equation}
    |\psi'\circ f_n-\psi'\circ f|^2 |\d f|_*^2\le 4 \Lip(\psi)|\d f|_*^2\in L^1(\Omega,\mm),
\end{equation}
hence we conclude by dominated convergence theorem. Notice that, this same argument proves that, for any subsequence, there exists a further subsequence converging to $\psi\circ f$ in $W^{1,2}(\X)$. As a consequence, the full sequence converges in $W^{1,2}(\X)$.
\end{proof}
\begin{remark}
\label{rmk:improvement_prop_composition}
Adapting the proof of Proposition \ref{prop:composition_W120}, one can prove that, for any $v\in W_0^{1,2}(\Omega)$ such that $v\ge 0$, $\mm$-a.e., there exists an approximating sequence $\{v_n\}\subset\Lip_{bs}(\Omega)$ consisting of non-negative functions.
Indeed, let $\{u_n\}\subset\Lip_{bs}(\Omega)$ such that $u_n\to u$ in $W_0^{1,2}(\Omega)$ and consider $\psi(t)=\max\{0,t\}\in \Lip(\R)$. 
Then, $v_n:=u_n^+ =\psi\circ u_n\in \Lip_{bs}(\Omega)$. To prove $v_n\to v$ in $W_0^{1,2}(\Omega)$, we follow the same strategy of Proposition \ref{prop:composition_W120}: the only problem may arise at the discontinuity point of $\psi'$. However, since by locality $\d v=0$, $\mm$-a.e. in $\{v=0\}$, in \eqref{eq:continuity_psi}, we have
\begin{equation}
    \int_{\Omega} |\psi'\circ v_n-\psi'\circ v|^2 |\d v|_*^2=\int_{\{v>0\}} |\psi'\circ v_n-\psi'\circ v|^2 |\d v|_*^2,
\end{equation}
and, up to a subsequence, we have dominated convergence as before.
\end{remark}
\bibliographystyle{alphaabbr}
\bibliography{biblio.bib}

\end{document}